\documentclass[11pt]{amsart}

\usepackage{amssymb, latexsym, amsfonts, pigpen, enumitem, mathrsfs, amsthm, bbm, stackrel, etoolbox, faktor, xfrac, nicefrac, longtable}
\usepackage[matrix,arrow,curve]{xy}
\usepackage[colorlinks=true,pagebackref=true,citecolor=cyan,linkcolor=magenta]{hyperref}
\hypersetup{linktocpage}

\usepackage{verbatim}

\usepackage{tikz}
\usetikzlibrary{matrix,arrows}
\usepackage{tikz-cd}
\usepackage{stmaryrd}
\usepackage{listings}

\usepackage[letterpaper,top=0.9in, bottom=0.9in, left=0.9in, right=0.9in]{geometry}

\numberwithin{equation}{section}

\newtheorem{theorem}[equation]{Theorem}
\newtheorem{lemma}[equation]{Lemma}
\newtheorem{proposition}[equation]{Proposition}
\newtheorem{corollary}[equation]{Corollary}
\newtheorem{theoremintr}{Theorem}

\theoremstyle{definition}
\newtheorem{definition}[equation]{Definition}
\newtheorem{example}[equation]{Example}
\newtheorem{remark}[equation]{Remark}
\newtheorem{remarkintr}{Remark}

\newcommand{\forg}{\mathrm{forg}}
\newcommand{\bforg}{\mathrm{forg}'} 
\newcommand{\einv}{\nicefrac{1}{e}}
\newcommand{\slice}{\mathrm{s}}
\newcommand{\effcov}{\mathrm{f}}
\newcommand{\Bnd}{\mathrm{B}}
\newcommand{\Cyc}{\mathrm{Z}}
\newcommand{\Ker}{\mathrm{Ker}}
\newcommand{\Hml}{\mathrm{H}}
\newcommand{\Jhom}{\mathrm{J}}
\newcommand{\mot}{\mathrm{mot}}
\newcommand{\Spc}{\mathbf{Spc}}
\newcommand{\Spt}{\mathbf{Sp}}
\newcommand{\Smk}{\mathbf{Sm}_S}
\newcommand{\SHk}{\mathbf{SH}(S)}
\newcommand{\sph}{\mathbbm{1}}
\newcommand{\Th}{{\mathrm{Th}}}
\newcommand{\ggl}{\gamma}
\newcommand{\gwall}{\gamma^{\mathrm{W}}}
\newcommand{\SL}{\mathrm{SL}}
\newcommand{\GL}{\mathrm{GL}}
\newcommand{\KWall}{\mathrm{K^{Wall}}}
\newcommand{\KWallrk}{\mathrm{K^{Wall}_{rk=0}}}
\newcommand{\Krk}{\mathrm{K_{rk=0}}}
\newcommand{\KSL}{\mathrm{K^{SL}}}
\newcommand{\KSLrk}{\mathrm{K^{SL}_{rk=0}}}
\newcommand{\PShv}{\mathrm{PSh}}
\newcommand{\WGr}{\mathrm{WGr}}
\newcommand{\BettiC}{\mathrm{Re_{B\mathbb{C}}}}
\newcommand{\C}{\mathbb{C}}
\newcommand{\R}{\mathbb{R}}
\newcommand{\BU}{\mathrm{BU}}
\newcommand{\Eone}{\mathbb{E}_1}
\newcommand{\Einf}{\mathbb{E}_\infty}
\newcommand{\RP}{\mathbb{R}\mathrm{P}}
\newcommand{\CP}{\mathbb{C}\mathrm{P}}
\newcommand{\Sp}{{\mathrm{Sp}}}
\newcommand{\Gr}{{\mathrm{Gr}}}

\newcommand{\Proj}{\mathbb{P}}
\newcommand{\Z}{{\mathbb{Z}}}

\newcommand{\HZ}{{\mathrm{H\Z}}}
\newcommand{\MGL}{\mathrm{MGL}}
\newcommand{\MSL}{\mathrm{MSL}}
\newcommand{\MSp}{\mathrm{MSp}}
\newcommand{\MSU}{\mathrm{MSU}}
\newcommand{\MW}{\mathrm{MWL}}
\newcommand{\MU}{\mathrm{MU}}
\newcommand{\KQ}{\mathrm{KQ}}
\newcommand{\GW}{\mathrm{GW}}
\newcommand{\Vect}{{\rm Vect}}
\newcommand{\W}{\mathrm{W}}
\newcommand{\K}{\mathrm{K}}
\newcommand{\A}{\mathbb{A}}
\newcommand{\NN}{\mathbb{N}}

\newcommand{\thc}{{th}}
\newcommand{\id}{\operatorname{id}}
\newcommand{\colim}{\operatorname*{colim}}

\newcommand{\cofib}{\operatorname*{cofib}}
\newcommand{\fib}{\operatorname*{fib}}
\newcommand{\KW}{\mathrm{KW}}
\newcommand{\struct}{\mathcal{O}}
\newcommand{\Gm}{{\mathbb{G}_m}}
\newcommand{\SH}{\mathbf{SH}}
\newcommand{\BO}{\mathrm{BO}}
\newcommand{\T}{\mathrm{T}}
\newcommand{\Spec}{\operatorname{Spec}}

\newcommand{\Pic}{\operatorname{Pic}}
\newcommand{\etale}{\'etale }
\newcommand{\la}{\langle}
\newcommand{\ra}{\rangle}
\newcommand{\EE}{\mathcal{E}}
\newcommand{\etatop}{\eta_{\mathrm{top}}}

\makeatletter
\NewCommandCopy\@@pmod\pmod
\DeclareRobustCommand{\pmod}{\@ifstar\@pmods\@@pmod}
\def\@pmods#1{\mkern4mu({\operator@font mod}\mkern 6mu#1)}
\makeatother

\newcommand{\pour}{\ar@{}[ur]|(0.2){\text{\pigpenfont G}}}
\newcommand{\podr}{\ar@{}[dr]|(0.2){\text{\pigpenfont A}}}

\newcommand{\bigslant}[2]{{\left.\raisebox{.2em}{$#1$}\middle/\raisebox{-.2em}{$#2$}\right.}}

\title{The geometric diagonal of the special linear algebraic~cobordism}
\author{Egor Zolotarev}
\address{LMU M\"unchen, Mathematisches Institut, Theresienstr. 39, 80333 M\"unchen, Germany}
\email{\href{mailto:zolotarev@math.lmu.de}{zolotarev@math.lmu.de}, \href{mailto:zolotarev-egv@yandex.ru}{zolotarev-egv@yandex.ru}}

\begin{document}
\begin{abstract}
The motivic version of the $c_1$-spherical cobordism spectrum is constructed. A connection of this spectrum with other motivic Thom spectra is established. Using this connection, we compute the $\Proj^1$-diagonal of the homotopy groups of the special linear algebraic cobordism $\pi_{2*,*}(\MSL)$ over a local Dedekind domain $k$ with $1/2\in k$ after inverting the exponential characteristic of the residue field of $k$. We discuss the action of the motivic Hopf element $\eta$ on this ring, obtain a description of the localization away from $2$, and compute the $2$-primary torsion subgroup. The complete answer is given in terms of the special unitary cobordism ring. An important component of the computation is the construction of Pontryagin characteristic numbers with values in the hermitian K-theory. We also construct Chern numbers in this setting, prove the motivic version of the Anderson--Brown--Peterson theorem, and briefly discuss classes of Calabi--Yau varieties in the SL-cobordism ring.
\end{abstract}
\maketitle
\tableofcontents
\section{Introduction}
The computation of various cobordism rings was one of the main directions of research in homotopy theory in the 1960s. The reason why it is connected to algebraic topology is the link between cobordism theories and Thom spectra, known as the Pontryagin--Thom construction. In modern terms, it states that the cobordism theory $\Omega^G_*$ as a generalized homology theory is isomorphic to the generalized homology theory represented by the Thom spectrum $\mathrm{M}G$. Therefore, the computation of the respective cobordism ring is equivalent to the study of the homotopy groups of a spectrum, which is a problem of stable homotopy theory.

The simplest of cobordism theories, unoriented cobordism, was the subject of Thom's seminal paper \cite{Thom}, who completely calculated the ring $\pi_*(\mathrm{MO})$. In the complex case, Milnor \cite{Mil60} and Novikov \cite{Nov0} obtained a complete description of the unitary cobordism ring $\pi_*(\MU)$, and later on, Quillen proved that this ring is isomorphic to the coefficient ring of the universal formal group law \cite{Quillen}. These results led to the emergence of the Adams--Novikov spectral sequence and the chromatic point of view, which have contributed immensely to the study of the stable  homotopy category, see \cite{Ra04}. The description of the oriented cobordism ring $\pi_*(\mathrm{MSO})$ was treated by Novikov \cite{Nov0, Nov} (the ring structure modulo
torsion) and by Wall \cite{Wall} (completely). For this purpose, Wall introduced the cobordism theory of manifolds with $\RP^1$-reduction, which sits in between oriented and unoriented cobordism theories. In \cite{CF}, Conner and Floyd used this idea in the unitary context, introducing the cobordism theory of complex manifolds with $\CP^1$-reduction $\pi_*(\W)$. As an application, they computed the cobordism ring of manifolds with a stable special unitary structure $\pi_*(\MSU)$ (see also \cite{CLP}, \cite{CP23} for a more modern exposition). The latter two constructions are the main topological insights for this paper. All of the above calculations are presented uniformly in Stong's book \cite{Sto}.

In the setting of motivic homotopy theory, the Thom spectrum $\MGL$ was introduced by Voevodsky in his ICM address \cite{Voe98}. This spectrum is the universal oriented commutative ring spectrum \cite{PPR08}, where ``oriented'' means that it possesses Thom classes for vector bundles. In the same paper with the definition, Voevodsky proposes a conjecture that the $\Proj^1$-diagonal of the coefficient ring of $\MGL$ over a regular local ring should be isomorphic to the coefficient ring of the universal formal group law. This motivic version of the Quillen theorem was proved over fields of characteristic zero by Hopkins and Morel (unpublished) and over fields away from the characteristic by Hoyois \cite{Hoy15}. The result of Hoyois was further generalized by Spitzweck to Noetherian local rings which are regular over discrete valuation rings \cite{SpiMGL} (away from the characteristic of the residue field). These computations lead to a large number of applications in motivic homotopy theory, see e.g., \cite{Chowt,RSO19,RSO24}. In addition, the construction of this spectrum serves as an inspiration for the Levine--Morel algebraic cobordism theory \cite{LevMor}, which is extensively studied now and has many applications to the problems of algebra and algebraic geometry \cite{LevPand,Vishik,SS21}. The connection between $\MGL$ and Levine--Morel algebraic cobordism $\Omega^*$ over fields of characteristic zero can be viewed as the motivic version of the Pontryagin--Thom theorem \cite{Levine}. However, this comparison is actually a posteriori, since the only known proof uses computations of the corresponding coefficient rings.

The story continues with the definition of the special linear and symplectic motivic Thom spectra due to Panin and Walter \cite{PW22}. Similarly to Voevodsky's algebraic cobordism, these spectra are universal ring spectra among those admitting Thom classes for oriented (or symplectic) vector bundles. An important difference between the motivic situation and the topological one, that such weakly oriented spectra play more important role here (even rationally); see \cite{ALP, DFJK}. Also, we should note that the spectra $\MSL$ and $\MSp$ are ``closer'' to the motivic sphere spectrum $\sph$ than $\MGL$. Although the definition of these Thom spectra is standard, little is known about them at the moment. The goal of this work is to present a computation of the $\Proj^1$-diagonal of the coefficient ring of the special linear algebraic cobordism spectrum $\MSL$ over some bases. This computation can be viewed as an $\SL$-analogue of the Voevodsky conjecture. Nevertheless, unlike in the case of $\GL$-cobordism, the answer for $\MSL$ depends on the base even for fields, as can be seen from Yakerson's description of the zero homotopy module \cite{Yakerson}.

Let us point out what is known in the literature on this issue apart from the zero homotopy group. To the best of our knowledge, historically the first partial computation appeared in the paper of Levine, Yang, and Zhao \cite{LYZ}. Assuming that the base is a spectrum of a perfect field of exponential characteristic $e$, they obtain a description of the localization away from $2e$ of the ``constant part'' modulo some divisible subgroup (see Remark \ref{remark_intro_LYZ}(2)). For this purpose, they used the motivic Adams spectral sequence. The only other advancement is the computation by Bachmann and Hopkins of the homotopy groups of the $\eta$-periodization over fields of characteristic different from $2$ \cite{BHop}, which was further generalized by Bachmann to Dedekind domains in which $2$ is invertible \cite{BacDVR}. In our terms, this can be restated as the computation of the stabilization of the $\Proj^1$-diagonal with respect to the multiplication by the motivic Hopf element $\eta$. Notice that in contrast to the classical picture, the Hopf element is not nilpotent here.
\subsection{Overview of results}
Now we formulate the main results of the present paper. For notation and conventions see \S \ref{table_of_not}.

To describe the main idea, let us look at the real and complex Betti realizations of $\MSL$. They are given by the oriented Thom spectrum $\mathrm{MSO}\in\Spt$ and the special unitary Thom spectrum $\MSU\in\Spt$ respectively. Because of this it seems reasonable to try to adopt Wall, Conner, and Floyd's constructions to motivic reality. This is done in this paper: we define the $c_1$-spherical algebraic cobordism spectrum $\MW$ over an arbitrary base scheme $S$ such that its real and complex Betti realizations are given by the cobordism spectra of manifolds with $\RP^1$ and $\CP^1$-reduction respectively. Furthermore, there are forgetful morphisms \[\MSL\xrightarrow{\forg} \MW\xrightarrow{\bforg} \MGL.\] The first technical issue of the present paper is the computation of the respective cofibers. We summarize these results in the following theorem. The first point states that we can think of the $c_1$-spherical algebraic cobordism spectrum as a ``geometric'' model for the cofiber $\MSL/\eta$. 
\begin{theoremintr}[Corollary \ref{cor:cofiber} and Theorem \ref{thm:fiber_seq}]\label{theorem_A}
    Let $S$ be a base scheme. Then there are cofiber sequences of motivic spectra over $S$ \begin{enumerate} \item $\Sigma^{1,1}\MSL\xrightarrow{\eta}\MSL\xrightarrow{\forg}\MW$, \item $\MW\xrightarrow{\bforg}\MGL\xrightarrow{\Delta}\Sigma^{4,2}\MGL$, \end{enumerate} where $\eta$ is the motivic Hopf element and $\Delta$ is the cohomological operation that corresponds to the characteristic class $c_1(\det\ggl)\cdot c_1(\det\ggl^\vee)$ under the Thom isomorphism $\MGL^{4,2}(\MGL)\cong \MGL^{4,2}(\mathrm{B}\GL)$.
\end{theoremintr}
\setcounter{remarkintr}{1}
\begin{remarkintr}
    The Betti realizations of the cofiber sequence (2) are split (see Proposition \ref{appendix:split_cofib} for the complex case). The same happens in the motivic context, at least over some bases. This fact and its consequences will be explored elsewhere.
\end{remarkintr}
After proving this result, we move on to more specific calculations. To be more precise, we use the second cofiber sequence to compute some homotopy groups of the $c_1$-spherical algebraic cobordism spectrum, restricting ourselves to the case of a local Dedekind domain $k$ (which is a field or a discrete valuation ring); see Propositions \ref{prop:vanishing}, \ref{prop:wall_geom_part}.  This is necessary since we use the isomorphism (see \cite[\S6]{SpiMGL},\cite[Proposition 8.2]{Hoy15}): \[\pi_{2*}(\MU)[\einv]\xrightarrow{\simeq} \pi_{2*,*}(\MGL)[\einv],\] where $e$ is the exponential characteristic of the residue field of $k$. As the first application, we prove the following theorem.
\setcounter{theoremintr}{2}
\begin{theoremintr}[Theorem \ref{theorem:eta_stab} and Corollary \ref{cor:modulo_eta_tors}]
    Suppose that $k$ is a local Dedekind domain and $e$ is the exponential characteristic of the residue field of $k$. Then the multiplication by the motivic Hopf element \[\eta\colon\pi_{2n+m,n+m}(\MSL)[\einv]\to \pi_{2n+m+1,n+m+1}(\MSL)[\einv],\] is an epimorphism if $m=0$ and an isomorphism if $m>0$. In particular, there is an equality of the torsion subgroups ${}_{\eta}\pi_{2*,*}(\MSL)[\einv]={}_{\eta^N}\pi_{2*,*}(\MSL)[\einv]$ for $N\geq 1$. If in addition $e\neq 2$, then there is an isomorphism of rings \[\bigslant{\pi_{2*,*}(\MSL)}{{}_{\eta}\pi_{2*,*}(\MSL)}[\einv]\cong \W(k)[\einv][y_4,y_8,\dots],\ \text{where}\ \mathrm{deg}(y_i)=(2i,i). \]
\end{theoremintr}
This theorem can be viewed as a lift of the $\eta$-periodic answer of Bachmann and Hopkins to the geometric diagonal. If one thinks about the real Betti realization, their answer should be seen as an analog of $\pi_*(\mathrm{MSO})[\nicefrac{1}{2}]$, and this theorem improves it to $\pi_*(\mathrm{MSO})/{}_2\pi_*(\mathrm{MSO})$. Moreover, the part about the equality of annihilators corresponds to the statement that the $2$-primary torsion in $\pi_*(\mathrm{MSO})$ is of exponent $2$. The proof of this theorem uses the cofiber sequence (1) and vanishing areas in the homotopy groups of $\MW[\einv]$. 

The further computation is given by the pedantic analysis of the exact sequence of homotopy groups induced by the cofiber sequence (1). For this purpose, we use two new key ingredients. The first one is a motivic version of the Conner--Floyd homology, which by definition stands on the second page of the $\eta$-Bockstein spectral sequence for $\MSL$. These homology groups are computable on the geometric diagonal; see Theorem \ref{theorem:conner_floyd_hom}. The second necessary component that we introduce is the Pontryagin characteristic numbers with values in the hermitian K-theory. We prove that these characteristic numbers determine some homotopy groups of $\MSL$; see Corollary \ref{cor:pontryagin_8n+1}. Using these tools allows us to get a complete answer to the question at hand. We summarize it in the following theorem. Below $\mathrm{I}_\MSL(k)$ denotes the graded subgroup of $\pi_{2*,*}(\MSL)$ which in degree $n$ is given by $\eta\cdot\pi_{2n-1,n-1}(\MSL)$ if $n$ is divisible by $4$ and zero otherwise. 
\begin{theoremintr}[Proposition \ref{prop:pi_0,0(MSL)}, Theorems \ref{theorem:complete_modulo_imslk}, \ref{theorem:complete_as_pullback}]\label{theorem:D}
    Suppose that $k$ is a local Dedekind domain and $e\neq 2$ is the exponential characteristic of the residue field of $k$. The ring $\pi_{2*,*}(\MSL)[\einv]$ admits a canonical structure of $\GW(k)$-algebra.
    \begin{enumerate}
        \item The subgroup $\mathrm{I}_\MSL(k)[\einv]$ is a graded ideal and there is an isomorphism of graded $\GW(k)$-algebras \[\bigslant{\pi_{2*,*}(\MSL)}{\mathrm{I}_\MSL(k)}[\einv]\cong \pi_{2*}(\MSU)[\einv],\] where the $\GW(k)$-algebra structure on the right hand side is induced by the rank homomorphism. If $k=\C$ then the complex Betti realization functor induces such an isomorphism.
        \item There is a cartesian square of graded $\GW(k)$-algebras \begin{center} \begin{tikzcd} {\pi_{2*,*}(\MSL)[\einv]} \arrow[dr, phantom, "\lrcorner", very near start] \arrow[dd, two heads] \arrow[rr, two heads] & & {\pi_{2*}(\MSU)[\einv]} \arrow[dd, two heads] \\  & {} & & \\{\W(k)[\einv][y_4,y_8,\dots]} \arrow[rr, "\overline{\mathrm{rk}}", two heads] &  & {\Z/2[y_4,y_8,\dots]},\end{tikzcd} \end{center} where the top arrow is defined via (1), and the vertical maps are quotients by the annihilators of $\eta\in\pi_{1,1}(\MSL)[\einv]$ and $\etatop\in\pi_1(\MSU)[\einv]$ respectively.
    \end{enumerate}
\end{theoremintr}
Roughly speaking, the above theorem says that the quotient by the ideal $\mathrm{I}_\MSL(k)[\einv]$ is the rigid part of the answer which turns out to be isomorphic to the topological one (at least after inverting $e$), and to recover the complete picture, we need to attach to it a certain number of the fundamental ideals. Below we provide implications and explanations for the obtained result.
\setcounter{remarkintr}{4}
\begin{remarkintr}\label{remark_intro_LYZ}
    \begin{enumerate}
        \item Looking at degree $0$, we get an isomorphism of rings (see Proposition \ref{prop:pi_0,0(MSL)} for a precise statement) \[\pi_{0,0}(\MSL)[\einv]\cong \W(k)[\einv]\times_{\Z/2} \Z[\einv]\cong \GW(k)[\einv],\] which generalizes Yakerson's computation of the zero homotopy group from fields of characteristic zero \cite[Proposition 3.6.3]{Yakerson}.
        \item Localizing the above cartesian square away from $2$, the bottom right corner becomes trivial, and we obtain an isomorphism \[\pi_{2*,*}(\MSL)[\nicefrac{1}{2e}]\cong \Z[\nicefrac{1}{2e}][x_2,x_3,\dots]\times \W(k)[\nicefrac{1}{2e}][y_4,y_8,\dots],\] see Corollary \ref{corollary:local_away_from2} and Remark \ref{rmk:away_2_LYZ}. Of course, this decomposition into the product of two rings corresponds to Morel's splitting of the motivic stable homotopy category (see \cite[Remark 4]{ALP}, \cite{DFJK}) \[\SH(k)[\nicefrac{1}{2}]=\SH(k)[\nicefrac{1}{2}]^+\times \SH(k)[\nicefrac{1}{2}]^-.\] In the case of a perfect field, Levine, Yang, and Zhao compute the respective plus part modulo a maximal subgroup that is $l$-divisible for all primes $l$ different from $2$ and $e$. They conjectured that this subgroup should be zero \cite[Remark 1.2]{LYZ}. In particular, our result proves this conjecture.
        \item It follows from the previous remark and the usual considerations of the Witt ring of $k$ that the $\Proj^1$-diagonal of $\MSL$ does not contain odd torsion (different from $e$). The structure of the $2$-primary torsion is investigated in detail in Corollaries \ref{cor:2_torsion} and \ref{cor:2tors_multiples}.
        \item It can be seen from the pullback diagram, that the ideal $\mathrm{I}_\MSL(k)[\einv]$ is given by $\mathrm{I}(k)[\einv]^{p(\frac{n}{4})}$ in degrees $n\equiv 0\pmod*{4}$, and is trivial otherwise. Here $p(m)$ is the number of partitions of $m$. In particular, if $k$ is a quadratically closed field we have an isomorphism \[\pi_{2*,*}(\MSL)[\einv]\cong \pi_{2*}(\MSU)[\einv].\]
    \end{enumerate}
\end{remarkintr}
Combining the last two points of the remark, we get a complete additive structure of $\pi_{2*,*}(\MSL)[\einv]$. In the following table we summarize the answer for the first few groups (we set $\pi_{2n,n}=\pi_{2n,n}(\MSL)$ and omit inverting $e$ to simplify the formulas):
\begin{center}
\begin{tabular}{|c|c|c|c|c|c|c|c|c|c|}
    \hline $\pi_{0,0}$ & $\pi_{2,1}$ & $\pi_{4,2}$ & $\pi_{6,3}$ & $\pi_{8,4}$ & $\pi_{10,5}$ & $\pi_{12,6}$ & $\pi_{14,7}$ & $\pi_{16,8}$ & $\pi_{18,9}$ \\
    \hline $\GW(k)$ & $\Z/2$ & $\Z$ & $\Z$ & $\GW(k)\oplus \Z$ & $\Z^2\oplus\Z/2$ & $\Z^4$ & $\Z^4$ & $\GW(k)^2\oplus\Z^5$ & $\Z^8\oplus (\Z/2)^2$ \\
    \hline
\end{tabular}
\end{center}
However, the ring structure of the $\Proj^1$-diagonal of $\MSL$ is complicated. This is the case even for $k=\C$ since an explicit description of the ring $\pi_{2*}(\MSU)$ is unknown; see Remark \ref{appendix:image_msu}. In light of this, the above answer seems to be the best possible modulo topological issues.

Let us comment on the isomorphism between $\pi_{2*,*}(\MSL)/\mathrm{I}_\MSL(k)[\einv]$ and $\pi_{2*}(\MSU)[\einv]$ stated in the previous theorem. The strategy is to prove a rigidity statement, which is that the quotient ring is stable under base change along homomorphisms of local Dedekind domains. Then we show the claim for $k=\C$ and extend it using the rigidity for various base changes. Therefore, in order to obtain the result for fields of positive characteristic, we need to include the case of discrete valuation rings.

In addition to this complete computation, we prove the motivic version of the Anderson--Brown--Peterson theorem that states that the geometric diagonal of $\MSL[\einv]$ is determined by the $\HZ$-character\-istic numbers and the $\KQ$-characteristic numbers; see Theorem \ref{theorem:motivic_abp}. We also compute the image of the canonical map $\pi_{2*,*}(\MSL)[\einv]\to \pi_{2*,*}(\MGL)[\einv].$ In the case of a field, we deduce formulas for the characteristic numbers of classes of smooth projective Calabi--Yau varieties and show that the plus part \[\bigslant{\pi_{2*,*}(\MSL)}{\mathrm{I}_\MSL(k)}[\nicefrac{1}{2e}]\cong\Z[\nicefrac{1}{2e}][x_2,x_3,\dots]\] is generated by such classes under the additional assumption that the field $k$ is infinite; see Proposition \ref{prop:char_numbers} and Corollary \ref{cor:calabi-yau_gen}. We also describe the classes of smooth projective Calabi--Yau varieties of dimension $\leq 2$ in $\MSL[\nicefrac{1}{e}]$; see Proposition \ref{prop:dim_zero_and_one}.
\begin{remarkintr}
    The only reason why we invert $e$ in all the results is the status of the Voevodsky conjecture about the $\Proj^1$-diagonal of $\MGL$. If it is ever proved integrally, then all results of the present paper will be valid without inversion of $e$.
\end{remarkintr}
\subsection{Organization}
The beginning of each section contains more detailed information about its contents. In Section \ref{section-2}, we recall basics about the motivic Thom functor formalism and construct the $c_1$-spherical algebraic cobordism spectrum $\MW$. In Section \ref{section-3}, we prove Theorem \ref{theorem_A}. Starting from Section \ref{section-4}, we assume that our base scheme is the spectrum of a local Dedekind domain. There we compute some homotopy groups of $\MW$ and introduce the algebraic Conner--Floyd homology. In Section \ref{section-5}, we lift the $\eta$-periodic computations of Bachmann and Hopkins to the geometric part and introduce Pontryagin characteristic numbers with values in the coefficient ring of an $\SL$-oriented homotopy commutative ring spectrum. In Section \ref{section-6}, we use the previous results to compute the $2$-primary torsion subgroup and obtain a complete answer for the $\Proj^1$-diagonal of $\MSL$. In Section \ref{section-7}, we define Chern numbers and use them to prove the motivic version of the Anderson--Brown--Peterson theorem.

In Appendix \ref{appendix_A}, we recall basic facts about the hermitian K-theory spectrum and compute its geometric diagonal over a regular local ring in which $2$ is invertible. In Appendix \ref{appendix_B}, we summarize all topological results that are used in the main part of the text for the convenience of the reader. 
\subsection{Table of notation}\label{table_of_not}

Throughout the paper, we employ the following notation and conventions.

\begin{longtable}{p{3.5cm}|l}
    $k$ & local Dedekind domain, i.e., a field or a discrete valuation ring \\
	base scheme $S$ & quasi-compact quasi-separated scheme $S$ \\
    $\mathbf{Sch}$ & category of qcqs schemes \\ 
	$\Smk$ & category of smooth, qcqs schemes over $S$ \\
    $\PShv(\Smk)$ & $\infty$-category of (space-valued) presheaves  on $\Smk$ \\
    $\mathrm{B}G$ & sheaf that classifies Nisnevich $G$-torsors for a group scheme $G$ \\
    $\mathbf{H}(S)$ & $\infty$-category of motivic spaces over $S$ \cite[\S 3]{Voe98}, \cite[\S 3.2]{MV99} \\
    $\Th_X(E)$ & Thom space of a vector bundle $E\to X$ \\
	$\SHk$ & stable $\infty$-category of motivic spectra over $S$ \cite[\S 5]{Voe98}, \cite[\S 4.1]{BH21} \\
    $\Sigma^{i,j}$ & $(i,j)$-suspension endofunctor of $\SHk$ \\
    $\Sigma^\infty$, $\Omega^\infty$ & infinite $\Proj^1$-suspension and $\Proj^1$-loop functors \\
    $\sph$ & motivic sphere spectrum \\
    $\MGL$, $\MSL$ & algebraic cobordism and special linear algebraic cobordism \cite[\S 4]{PW22} \\
    $\KQ$, $\KW$ & hermitian $\K$-theory spectrum and Witt spectrum, see Appendix \ref{appendix_A} \\
    $\mathrm{H}A$ & Spitzweck's motivic cohomology spectrum with $A$-coefficients \cite{SpiHZ} \\
    $\EE\in\mathrm{CAlg(h\SHk)}$ & homotopy commutative ring spectrum $\EE$ \\
    $\EE\in\mathrm{CAlg(}\SHk)$ & $\Einf$-ring spectrum $\EE$ \\
    $[-,-]$ & homotopy classes of maps in $\SHk$ \\
    $\pi_{i,j}(\EE)$, $\underline{\pi}_{i,j}(\EE)$ & bigraded homotopy groups and sheaves of a spectrum $\EE$ \\
    $\pi_{2*,*}(\EE)$ & $\Proj^1$-diagonal/geometric diagonal/geometric part of a spectrum $\EE$ \\
    $\pi_{i,j}(\alpha)$ or $\alpha_*$ & for $\alpha\in [\EE,\mathcal{F}]$ the induced map $\pi_{i,j}(\EE)\to \pi_{i,j}(\mathcal{F})$ \\
    $\EE_{i,j}(-)$, $\EE^{i,j}(-)$ & homology and cohomology theory represented by a spectrum $\EE$ \\
    $\eta\in\pi_{1,1}(\sph)$ & motivic Hopf element, i.e., stabilization of $\A^2\setminus\{0\}\to\Proj^1$ \cite[\S6.2]{Mor03} \\
    ${}_{\eta^m}\pi_{2*,*}(\MSL)$ & annihilator $\mathrm{Ann}_{\pi_{2*,*}(\MSL)}(\eta^m)=\{\alpha\in\pi_{2*,*}(\MSL)\,|\,\alpha\cdot\eta^m=0\}$ \\
    $\EE[\eta^{-1}]$ & $\colim[\EE\xrightarrow{\eta} \Sigma^{-1,-1}\EE\xrightarrow{\eta}\Sigma^{-2,-2}\EE\xrightarrow{\eta} \cdots]$ \\
    $\W(-)$, $\GW(-)$, $\mathrm{I}(-)$ & Witt ring, Grothendieck--Witt ring and fundamental ideal \\
    ${}_{l^m}A$ & torsion subgroup $\{a\in A\,|\,a\cdot l^m=0\}$ of an abelian group $A$ \\
    $p(n)$ & number of partitions of $n$ \\
    $\Spt$ & stable $\infty$-category of spectra \\
    $\etatop\in\pi_1(\sph_{\mathrm{top}})$ & classical Hopf element \\
    $\etatop\in \pi_{1,0}(\sph)$ & image of $\etatop$ under the constant functor $\Spt\to \SHk$ \cite[Def. 4.5]{Ana21} \\
    $\MU$, $\MSU$ & complex cobordism and special unitary cobordism, see Appendix \ref{appendix_B} \\
    $\mathrm{KO}$ & real $\K$-theory spectrum \\
\end{longtable}
\subsection{Acknowledgements} 
I am deeply grateful to Alexey Ananyevskiy for introducing me to the subject of this paper, numerous conversations, and constant support during the work. I would like to thank sincerely Andrei Lavrenov, Ivan Panin, and Oliver R\"ondigs for helpful discussions, Vasily Ionin for careful reading of a draft of this paper, and the anonymous referee for useful suggestions that helped improve the exposition. The initial part of the research was done during my stay at the Byurakan Astrophysical Observatory. The work is supported by the DFG research grant AN 1545/4-1.
\section{\texorpdfstring{$c_1$}{c1}-spherical algebraic cobordism}\label{section-2}
In this section we construct the $c_1$-spherical algebraic cobordism spectrum in the stable motivic homotopy category over a scheme $S$ (see Definition \ref{definition:MWL}) and establish its basic properties. We equip it with a natural action of $\MSL$ (see Lemma \ref{prop:wall_multi}) and compute the complex Betti realization (see Proposition \ref{prop:betti_realiz_mwl}). For this purpose, we use the motivic Thom functor formalism from \cite{BH21}. An equivalent definition in terms of Thom spaces over appropriate Grassmannians is provided, see Proposition \ref{prop:wall_colim}.
\subsection{Recollection on the motivic Thom functor}
Let $\Pic(\SH)$ denote the presheaf on $\Smk$ that takes a smooth $S$-scheme $X$ to the $\Einf$-space of $\wedge$-invertible motivic spectra $\Pic(\SH(X))$. To a morphism of presheaves of spaces $\beta\colon\mathrm{B}\to \Pic(\SH)$ we associate a motivic Thom spectrum $\mathrm{M}\beta\in \SHk$ by the colimit construction \[ \mathrm{M}(\beta\colon\mathrm{B}\to \Pic(\SH)):=\colim_{\substack{f\colon X\to S\;\mathrm{smooth} \\ b\in \mathrm{B}(X)}} f_{\#}\beta(b).\] This defines a symmetric monoidal functor of $\infty$-categories $\mathrm{M}\colon\PShv(\Smk)_{/\Pic(\SH)}\to \SHk.$ Moreover, this functor inverts Nisnevich equivalences and even motivic equivalences over a motivic space; see \cite[Proposition 16.9]{BH21}.

For a scheme $X$ we denote by $\Vect(X)$ the $\infty$-groupoid of vector bundles over $X$. Taking the group completion and the Zariski localization of the presheaf $\Vect\in\PShv(\mathbf{Sch})$ we get the Thomason--Trobaugh $\K$-theory presheaf $\K:=\mathrm{L_{Zar}}(\Vect^{\mathrm{gp}})$ (see \cite[Theorems 7.6 and 8.1]{TT90}). Restricting $\K$ to $\Smk$, we can construct the motivic $\Jhom$-homomorphism \[\Jhom\colon\K\to \Pic(\SH),\ E\mapsto \Sigma^\infty\Th(E). \] This is a map of presheaves of grouplike $\Einf$-spaces. Restricting the motivic Thom functor along the $\Jhom$-homomorphism, we obtain a symmetric monoidal functor of $\infty$-categories \[\mathrm{M}\colon\PShv(\Smk)_{/\K}\to\SHk.\]
\begin{example}
    \begin{enumerate}[label=({\arabic*}),ref=2.1.(\arabic*)]
    \item \label{exmpl:virtual_thom}  Let $X$ be a smooth (ind-)scheme over $S$, let $E\ominus\struct^n$ be a virtual vector bundle over $X$, and let $[E\ominus\struct^n]\colon X\to \K$ be its class in the $\K$-theory space. Then the corresponding Thom spectrum $\mathrm{M}([E\ominus\struct^n]\colon X\to\K)$ is given by $\Sigma^{\infty-(2n,n)}\Th_X(E)$. We denote this spectrum by $\Th_X(E\ominus \struct^n)\in\SHk$ or $\Th(E\ominus \struct^n)$ if $X$ is clear from the context.
    \item   The motivic Thom spectrum of the rank zero summand in $\K$-theory $\iota\colon\Krk=\K\times_{\underline{\Z}}\{0\}\to \K$ is the Voevodsky algebraic cobordism spectrum \[\mathrm{M}(\Krk\xrightarrow{\iota}\K)\cong \MGL.\] This can be shown using a motivic description of $\Krk$ in terms of Grassmannians together with some basic properties of the motivic Thom functor; see \cite[Lemma 4.6]{BHop}.
    \item   For a scheme $X$ consider the $\infty$-groupoid of vector $\SL$-bundles $\Vect^{\SL}(X)$ (see e.g., \cite[\S 2]{Ana20}). Taking the group completion and the Zariski localization of the presheaf $\Vect^\SL\in\PShv(\mathbf{Sch})$ we get the special linear $\K$-theory presheaf $\KSL$. Restricting it to $\Smk$ and applying the Thom functor to the rank zero presheaf $\KSLrk=\KSL\times_{\underline{\Z}}\{0\}$ through the natural map $\KSLrk\to \Krk$, we obtain the special linear algebraic cobordism spectrum of Panin and Walter \cite{PW22} \[\mathrm{M}(\KSLrk\to \K)\cong \MSL.\]
    \end{enumerate}
\end{example}
\begin{remark}\label{remark:det}
    The presheaf $\KSL$ from the last example is equivalent to the fiber of the determinant map $\det\colon\K\to \Pic$; see \cite[Example 3.3.4]{EHK}. We stress that $\KSL$ is a presheaf of $\Eone$-spaces, while $\KSLrk$ is a presheaf of $\Einf$-spaces; see e.g., \cite[Example A.0.6]{EHK}.
\end{remark}
\subsection{The main construction}\label{section_2-2}
Consider the following composition of maps of presheaves on $\Smk$ \[ \phi\colon\K\times \Proj^1\to \Pic\times \Pic\to \Pic,\] where the first morphism is given by the product of $\det\colon\K\to \Pic$ and $\struct(1)\colon\Proj^1\to \Pic$ and the second one is the multiplication on $\Pic$. Denote by $\KWall$ the fiber of $\phi$, and by $\KWallrk$ the respective rank zero presheaf $\KWall\times_{\underline{\Z}} \{0\}$.
\begin{lemma}\label{lem:kspace_natural}
    Let $f\colon T\to S$ be a morphism of schemes. Then the canonical map of presheaves of spaces on $\mathbf{Sm}_T$ \[f^*(\K_S^{\mathrm{Wall}})\to \K_T^{\mathrm{Wall}}\] is a Zariski equivalence. The same holds for the rank zero presheaf $\KWallrk$.
\end{lemma}
\begin{proof}
    Applying $f^*$ to the fiber sequence $\K_S^{\mathrm{Wall}}\to \K_S\times \Proj^1_S\to \Pic_S$ we get a fiber sequence of presheaves over $T$. The map $f^*(\K_S\times \Proj^1_S)=f^*(\K_S)\times \Proj^1_T\to \K_T\times \Proj^1_T$ is a Zariski equivalence by \cite[Lemma 16.12]{BH21} and $f^*(\Pic_S)\to \Pic_T$ is a Zariski equivalence since $\Pic_S$ is Zariski equivalent to $\mathrm{B}\Gm_{,S}$; see \cite[Lemma 2.6]{NSO}. Hence, the claim is proved for $\KWall$. The case of the rank zero presheaves follows immediately.
\end{proof}
The embedding $\infty\colon S\hookrightarrow \Proj^1$ leads to the commutative diagram of presheaves \[\xymatrix{\K\times S \ar@{=}[r] \ar[rd] & \K \ar[r]^{\det} & \Pic \ar@{=}[d] \\ & \K\times \Proj^1 \ar[r]^\phi & \Pic.}\] It induces a morphism between the fibers of the horizontal arrows. Combining this with Remark \ref{remark:det}, we obtain a map $\KSL\to \KWall$. The restriction to the rank zero presheaves gives $\KSLrk\to \KWallrk$.
\begin{lemma}\label{lemma:kwall_action}
    The presheaf $\KWall$ has a natural left $\KSL$-module structure such that $\KSL\to\KWall\to\K$ are maps of $\Eone$-modules. Similarly, the rank zero presheaf $\KWallrk$ has a natural $\KSLrk$-module structure such that $\KSLrk\to \KWallrk\to \Krk$ are maps of $\Einf$-modules. 
\end{lemma}
\begin{proof}
    Let us endow $\K\times \Proj^1$ and $\Pic$ with the natural and trivial left action of $\KSL$ respectively. Then the map $\phi$ is a morphism of left $\KSL$-modules, and it follows from \cite[Corollary 4.2.3.3]{LHA} that the left $\KSL$-module structure on $\K\times \Proj^1$ lifts to the fiber $\KWall$. The map $ \KWall\to \K$ has a lift to a morphism of $\KSL$-modules via the composition $\KWall\to\K\times \Proj^1\to\K$. In turn, the map $\KSL\to\KWall$ lifts to a morphism of $\KSL$-modules according to the above diagram. The case of the rank zero presheaf $\KWallrk$ is similar with the difference that $\KSLrk$ is a presheaf of $\Einf$-spaces (see Remark \ref{remark:det}).
\end{proof}
\begin{definition}\label{definition:MWL}
The \textit{$c_1$-spherical algebraic cobordism spectrum} $\MW_S$ is the motivic Thom spectrum associated with the composition $\KWallrk\to \Krk\xrightarrow{\iota} \K$ \[ \MW_S:=\mathrm{M}(\KWallrk\to \K)\in \SHk.\] When $S$ is clear from the context we denote it simply by $\MW$. Applying the motivic Thom functor to the maps $\KSLrk\to \KWallrk\to \Krk$ of presheaves over $\K$, we obtain \[\MSL\xrightarrow{\forg} \MW\xrightarrow{\bforg} \MGL.\]
\end{definition}
\begin{proposition}\label{prop:wall_multi}
    The motivic spectrum $\MW$ has a natural $\MSL$-module structure such that the above morphisms are maps of $\MSL$-modules.
\end{proposition}
\begin{proof}
    The natural $\KSLrk$-module structure constructed in the previous lemma gives the $\KSLrk$-module structure in the slice $\infty$-category $\PShv(\Smk)_{/\K}$. Moreover, the maps $\KSLrk\to \KWallrk\to\Krk$ are compatible with the canonical morphisms to $\K$. The result follows since the motivic Thom functor is symmetric monoidal.
\end{proof}
\begin{proposition}\label{prop:wall_natural}
    The $c_1$-spherical algebraic cobordism spectrum $\MW$ is stable under base change.
\end{proposition}
\begin{proof}
    Follows from Lemma \ref{lem:kspace_natural} and \cite[Lemma 16.7 and Proposition 16.9.(1)]{BH21}. 
\end{proof}
\subsection{Description via Grassmannians}
Let $\Gr_n(\A^m)$ denote the Grassmannian of $n$-dimen\-sional vector subbundles of $\struct^m_S$ and let $\ggl_{n,m}$ be the tautological rank $n$ vector bundle over $\Gr_n(\A^m)$. Taking the colimit over the closed embeddings $\Gr_n(\A^m)\hookrightarrow \Gr_n(\A^{m+1})$ in $\PShv(\Smk)$, we get the ind-scheme $\Gr_n(\A^\infty)\in\PShv(\Smk)$ or simply $\Gr_n$. We also use the symbol $\ggl_{n}$ to denote the colimit of the corresponding tautological bundles.
\begin{definition}
    The \textit{Wall Grassmannian} is the complement of the zero section of the line bundle $\det(\ggl_{n,m})\boxtimes \struct(1)$ over $\Gr_n(\A^m)\times \Proj^1$, \[\WGr_n(\A^m):=(\det(\ggl_{n,m})\boxtimes \struct(1))^\circ\in \Smk.\] Denote by $\gwall_{n,m}$ the pullback of the tautological vector bundle $\ggl_{n,m}$ along the canonical projection $\WGr_n(\A^m)\to \Gr_n(\A^m)$. 
\end{definition}
    The closed embeddings of the Grassmannians $\Gr_n(\A^m)\hookrightarrow \Gr_n(\A^{m+1})$ induce closed embeddings of the Wall Grassmannians $\WGr_n(\A^m)\hookrightarrow \WGr_n(\A^{m+1})$. Taking the colimit over these maps we obtain the ind-scheme \[\WGr_n=\WGr_n(\A^\infty):=\colim_{m}\WGr_n(\A^m)\in \PShv(\Smk).\] We also denote by $\gwall_{n}$ the colimit of the vector bundles $\gwall_{n,m}$. In addition, there are maps $\WGr_m\to \WGr_{m+1}$ induced by $\Gr_m\to \Gr_{m+1}$.
\begin{remark}
    Let us explain the intuition behind the definition of the Wall Grassmannians. Consider the presheaf $\mathrm{BW}_n$ defined as the fiber of the following morphism 
    $$ \mathrm{BGL}_n\times \Proj^1\xrightarrow{\det\times\struct(1)^\circ} \mathrm{B}\Gm\times \mathrm{B}\Gm\to \mathrm{B}\Gm, $$
    where the second map is induced by the multiplication on $\Gm$. Then the set of maps of presheaves $X\to \mathrm{BW}_n$ for a smooth $S$-scheme $X$ is in bijection with the set of triples $(E,f,\lambda)$, where $E$ is a rank $n$ vector bundle over $X$, $f\colon X\to \Proj^1$ is a morphism of $S$-schemes and $\lambda\colon\det(E)\xrightarrow{\simeq} f^*\struct(-1)$ is an isomorphism of line bundles. In other words, this presheaf classifies rank $n$ vector bundles whose determinants come from $\Proj^1$. Note that the existence of such a lift $\det(E)\colon X\to \Proj^1$ is a strong homotopical assumption on the vector bundle $E$ (for example, it implies vanishing of $c_1(E)^2\in H^4(X,\Z(2))$).
    We also stress that this additional structure is not linear in the sense that it cannot be described as a reduction of the structure group from $\GL_n$ to a linear algebraic group $G$. In turn, the Wall Grassmannian $\WGr_n$ gives us a geometric model for the object $\mathrm{BW}_n$ in the unstable $\A^1$-homotopy category. To be more precise, one can prove that the vector bundle $\gamma^\W_n$ induces a motivic equivalence $\WGr_n\cong_{\mot}\mathrm{BW}_n$. However, we do not consider $\mathrm{BW}_n$ in the main part of the text since $\WGr_n$ is enough for all our purposes. 
\end{remark}
    Our goal is to give a description of $\MW$ in terms of the Wall Grassmannians. Consider the following commutative diagram \[\xymatrixcolsep{7pc}\xymatrix{\Gr_n\times \Proj^1 \ar[r]^-{(\det(\ggl_{n})\,\boxtimes\:\struct(1))^\circ} \ar[d]_-{[\ggl_{n}\,\ominus\: \struct^n]\times \id} & \mathrm{B}\Gm \ar[d] \\ \Krk\times \Proj^1 \ar[r]^-{\phi_{\mathrm{rk=0}}} & \Pic,}\] where the top map classifies the respective $\Gm$-torsor and $[\ggl_{n}\,\ominus\,\struct^n]$ stands for the corresponding class of the virtual vector bundle in the $\K$-theory space. It induces a map between the fibers of the horizontal morphisms \[\WGr_n\to \KWallrk.\] Since the arrow $\Gr_n\to \Krk$ agrees with the maps $\Gr_m\to \Gr_{m+1}$, we have a natural morphism \[\colim_{n} \WGr_n\to \KWallrk.\]
\begin{lemma}
    The canonical map of presheaves $\colim_{n} \WGr_n\to \KWallrk$ is a motivic equivalence.
\end{lemma}
\begin{proof}
    Suppose first that $S$ is a regular scheme. Consider the diagram \[ \xymatrix{\colim_{n} \WGr_n \ar[r] \ar[d] & \colim_{n} \Gr_n\times \Proj^1 \ar[r] \ar[d] & \mathrm{B}\Gm \ar[d] \\ \KWallrk \ar[r] & \Krk\times \Proj^1 \ar[r] & \Pic,} \] which commutes by the discussion above. By the universality of colimits in the $\infty$-topos of presheaves, horizontal lines in the diagram form fiber sequences in $\PShv(\Smk)$. Since $\mathrm{B}\Gm$ and $\Pic$ are Nisnevich sheaves with $\A^1$-invariant $\pi_0$, these fiber sequences are motivic fiber sequences (i.e., they remain fiber sequences after applying the  motivic localization $\mathrm{L}_\mot$) by \cite[Theorem 2.2.5]{AHW}. Hence, it suffices to show that the middle and the right vertical maps are motivic equivalences. The middle one follows from the discussion just before \cite[Theorem 16.13]{BH21} since the motivic localization functor commutes with finite products; the right one is an equivalence even before localization.
    
    For an arbitrary base scheme $S$ the result follows from Lemma \ref{lem:kspace_natural} by base change from $\Spec(\Z)$.
\end{proof}
\begin{proposition}\label{prop:wall_colim}
    The canonical morphism $\colim_{n} \WGr_n\to\KWallrk$ induces an equivalence \[ \colim_{n} \Th(\gwall_{n}\ominus \struct^n)\cong \MW, \] of motivic spectra over $S$.
\end{proposition}
\begin{proof}
From \cite[Proposition 16.9.(2) and Remark 16.11]{BH21} it follows that the motivic Thom functor inverts the motivic equivalence from the previous lemma. Thus, it induces \[\colim_n \mathrm{M}(\WGr_n\to \K)\cong \mathrm{M}(\colim_{n} \WGr_n\to \K)\xrightarrow{\simeq} \MW,\] and we obtain the claim by Example \ref{exmpl:virtual_thom}. 
\end{proof}
    Using the above description we can identify the complex Betti realization of $\MW$. Recall that if $S=\Spec(\C)$, then there is a complex realization functor $\BettiC\colon\SH(\C)\to \Spt$, which is symmetric monoidal and satisfies $\BettiC(\Sigma^\infty_+ X)\cong \Sigma^\infty_+ X(\mathbb{C})$ for $X\in\mathbf{Sm}_\C$. We refer to Appendix \ref{appendix_B} for a recollection on the $c_1$-spherical cobordism spectrum in topology. 
    \begin{proposition}\label{prop:betti_realiz_mwl}
    We have $\BettiC(\MW_\C)\cong \W$, where $\W$ is the $c_1$-spherical cobordism spectrum.
    \end{proposition}
    \begin{proof} 
        By construction, the complex realization of $\WGr_n$ is given by the complement to the zero section of $\det(\mathrm{EU}(n))\boxtimes \struct(1)$ over $\Gr_n(\C^\infty)\times\CP^1$. The underlying $\C^\times$-bundle is equivalent to the respective principal $S^1$-bundle $\mathrm{BW}(n)$; see Remark \ref{remark:bwn_s1bundle}. Since the complex realization commutes with colimits, we see that $\BettiC(\Th(\gwall_{n}))\cong\mathrm{TBW}(n)$ (here $\mathrm{TBW}(n)$ is the Thom space of the tautological complex vector bundle over $\mathrm{BW}(n)$) and $\BettiC(\MW_\C)\cong\W$.
    \end{proof}
    \begin{remark}
        Analogously, it can be shown that the real Betti realization (or more generally real \etale realization) of $\MW$ is given by Wall's $w_1$-spherical cobordism spectrum $\W_\R$ \cite[Chapter VIII]{Sto}. However, we do not need this here.
    \end{remark}
\section{Connection with \texorpdfstring{\textrm{MSL}}{MSL} and \texorpdfstring{\textrm{MGL}}{MGL}}\label{section-3}
In this section we discuss maps $\forg\colon\MSL\to \MW$ and $\bforg\colon\MW\to \MGL$ in detail. In particular, we compute their cofibers and identify maps in the respective cofiber sequences, see Corollary \ref{cor:cofiber} and Theorem \ref{thm:fiber_seq}. This section is the main technical part of the paper.
\subsection{The \texorpdfstring{$c_1$}{c1}-spherical algebraic cobordism spectrum as a cofiber}
Denote by $s\colon\Pic(X)\to \Vect(X)$ the morphism of groupoids that takes a line bundle over a smooth $S$-scheme $X$ to itself viewed as a vector bundle. This rule defines a section $s\colon\Pic\to \Vect$ to the determinant map: $\det\circ\,s=\id$. Applying the group completion and the Zariski localization, we get a section $s\colon\Pic\to \K$ to the determinant morphism $\det\colon\K\to \Pic$. Similarly there is a section $s_{\mathrm{rk=0}}\colon\Pic\to \Krk$ to $\det_{\mathrm{rk=0}}\colon\Krk\to \Pic$.
\begin{proposition}\label{prop:k=kslxpic}
    Let $S$ be a base scheme. Then the morphism of presheaves over $S$ \[ \KSL\times \Pic\xrightarrow{\id\times s} \KSL\times \K\xrightarrow{\mathrm{act}} \K, \] is an equivalence, where $\mathrm{act}$ is the action on the $\KSL$-module. The same holds for the rank zero presheaves $\KSLrk\times \Pic\xrightarrow{\simeq} \Krk$.
\end{proposition}
\begin{proof}
    From Remark \ref{remark:det} and the discussion above, we have a fiber sequence of presheaves with a section \[\begin{tikzcd} \KSL \ar{r} & \K \ar{r}{\det} & \Pic. \ar[shift left=1ex]{l}{s} \end{tikzcd}\] The action of $\KSL$ gives the desired equivalence since it induces an isomorphism on the homotopy groups due to the standard topological argument. The proof for the rank zero presheaves is analogous.
\end{proof}
\begin{remark}
    Consider the following motivic equivalences $\mathrm{B}\GL=\colim_n \mathrm{B}\GL_n\cong_\mot \Krk$, $\mathrm{B}\SL=\colim_n \mathrm{B}\SL_n\cong_\mot \KSLrk$, and $\mathrm{B}\Gm\cong\Pic$. Combining these descriptions with the above isomorphism we obtain a motivic equivalence $\mathrm{B}\SL\times \mathrm{B}\Gm\cong_\mot\mathrm{B}\GL$. This generalizes the splittings $\mathrm{BSU}\times \BU(1)\cong \BU$ and $\mathrm{BSO}\times \BO(1)\cong \BO$ in topology.
\end{remark}
\begin{corollary}[see also {\cite[Theorem 1.1]{Nan23}}] \label{corollary:mgl=mslxpinf}
    Let $S$ be a base scheme. Then the morphism of motivic spectra over $S$ \[\MSL\wedge \Th_{\Proj^\infty}(\struct(-1)\ominus \struct)\xrightarrow{\id\wedge \, \mathrm{in}} \MSL\wedge \MGL\xrightarrow{\mathrm{act}} \MGL,\] is an equivalence of $\MSL$-modules. Here $\mathrm{in}\colon\Th_{\Proj^\infty}(\struct(-1)\ominus \struct)\to \MGL$ is the canonical map and $\mathrm{act}$ is the action on the $\MSL$-module.
\end{corollary}
\begin{proof}
    It is clear that the desired morphism is a map of $\MSL$-modules. Let us consider the presheaf $\Pic$ as an object of $\PShv(\Smk)_{/\K}$ via the map $\Pic\xrightarrow{s_{\mathrm{rk=0}}} \Krk\xrightarrow{\iota} \K$. Then the equivalence $\KSLrk\times \Pic\cong \Krk$ from the previous lemma takes place in the slice category. Applying the motivic Thom functor we obtain that the composition \[\MSL\wedge\mathrm{M}(\Pic\xrightarrow{\iota\, \circ\, s_{\mathrm{rk=0}}}\K)\rightarrow \MSL\wedge \MGL\to \MGL\] is an isomorphism. The standard motivic equivalence $\struct(-1)\colon\Proj^\infty\to \Pic$ induces an equivalence of the Thom spectra \[\Th(\struct(-1)\ominus \struct)\xrightarrow{\simeq} \mathrm{M}(\Pic\xrightarrow{\iota\, \circ\, s_{\mathrm{rk=0}}} \K)\] by \cite[Proposition 16.9(2) and Remark 16.11]{BH21}. This concludes the proof.
\end{proof}
Recall the map $\phi$ from \S \ref{section_2-2} and consider the morphism of presheaves $([\struct(-1)],\id)\colon\Proj^1\to \K\times \Proj^1$ over a scheme $S$. The choice of the isomorphism $\struct(-1)\otimes\struct(1)\cong \struct$ gives a null-homotopy for the composition $\phi\circ([\struct(-1)],\id)$ and thus yields a lift $t$: \[\xymatrix{& \Proj^1 \ar[d]^-{([\struct(-1)],\id)} \ar@{-->}[ld]_-t & \\ \KWall \ar[r] & \K\times \Proj^1 \ar[r]^-{\phi} & \Pic.}\] It follows from the diagram that $t$ is a section of the projection $\KWall\to \Proj^1$. Taking the restriction of the above picture to the rank zero presheaves, we obtain a section $t_{\mathrm{rk=0}}\colon\Proj^1\to \KWallrk$ to the map $\KWallrk\to \Proj^1$.
\begin{lemma}\label{lemma:KslP1_KWall}
    Let $S$ be a base scheme. Then the morphism of presheaves over $S$ \[\KSL\times \Proj^1\xrightarrow{\id\times \,t} \KSL\times \KWall\xrightarrow{\mathrm{act}} \KWall\] is an equivalence, where $\mathrm{act}$ is the action constructed in Lemma \ref{lemma:kwall_action}. The same holds for the rank zero presheaves $\KSLrk\times \Proj^1\xrightarrow{\simeq} \KWallrk$.
\end{lemma}
\begin{proof}
    Consider the commutative diagram in which rows are fiber sequences in $\PShv(\Smk)$ \[\xymatrix{\fib(\KWall\to\Proj^1) \ar[r] \ar[d] & \KWall \ar[r] \ar[d] & \Proj^1 \ar[d]^{\struct(-1)} \\ \KSL \ar[r] & \K \ar[r]^-{\det} & \Pic.}\] By definition of $\KWall$ the right square is a pullback square. Thus, the left vertical arrow is an equivalence and there is a fiber sequence of presheaves with a section \[\begin{tikzcd}
    \KSL \arrow{r} & \KWall \ar{r} & \Proj^1. \arrow[shift left=1ex]{l}{t} \end{tikzcd}\] The rest of the proof is the same as in Proposition \ref{prop:k=kslxpic}. The argument for the rank zero presheaves is similar.
\end{proof}
\begin{remark}
    The morphism $t_{\mathrm{rk=0}}$ constructed above induces a map of the respective Thom spectra $\Th(t_{\mathrm{rk=0}})\colon\sph/\eta\cong\Th_{\Proj^1}(\struct(-1)\ominus\struct)\to \MW$ (see the diagram \ref{equation:diagram} for the first equivalence). Actually, it can be shown that the Thom space $\Th_{\Proj^1}(\struct(-1))$ is motivically equivalent to the first space $\Th(\gwall_{1})$ of the $\T$-spectrum $\MW$. However, we do not need this here.
\end{remark}
\begin{theorem}\label{thm:msl/eta=mwl}
    Let $S$ be a base scheme. Then the morphism of motivic spectra over $S$ \[\MSL/\eta=\MSL\wedge \sph/\eta \xrightarrow{\id\wedge \, \Th(t_\mathrm{rk=0})} \MSL\wedge \MW\xrightarrow{\mathrm{act}}\MW\] is an equivalence of $\MSL$-modules. Here $\mathrm{act}$ is the action constructed in Proposition \ref{prop:wall_multi}.
\end{theorem}
\begin{proof}
    The proof is the same as in Corollary \ref{corollary:mgl=mslxpinf}, using the previous lemma instead of Proposition \ref{prop:k=kslxpic}.
\end{proof}
\begin{lemma}
    The morphism of presheaves $\KSL=\KSL\times S\xrightarrow{\id\times \infty} \KSL\times \Proj^1\xrightarrow{\simeq} \KWall$ is homotopic to the canonical map $\KSL\to \KWall$. The same holds for the rank zero presheaves.
\end{lemma}
\begin{proof}
    We need to show that the following diagram commutes up to homotopy \[ \xymatrix{ \KSL \ar[r] \ar[d] \ar@{-->}[rd] & \KSL\times \Proj^1 \ar[d]^-{\id\times \, t} \\ \KWall & \KSL\times \KWall. \ar[l]_-{\mathrm{act}} } \] Define the dashed arrow $\KSL=\KSL\times S\to \KSL\times \KWall$ as the product of $\KSL\xrightarrow{\id}\KSL$ and $S\xrightarrow{[\struct]} \KSL\to \KWall$. We need to check that both triangles commute up to a homotopy. The commutativity of the bottom one follows from Lemma \ref{lemma:kwall_action}. To prove the commutativity of the top triangle we need to show that the maps $$S\xrightarrow{\infty}\Proj^1\xrightarrow{t} \KWall\ \ \text{and}\ \ S\xrightarrow{[\struct]} \KSL\to\KWall$$ are homotopic. The second composite is determined up to homotopy by the morphism $([\struct],\infty)\colon S\to \K\times \Proj^1$ (see the pullback square in the proof of Lemma \ref{lemma:KslP1_KWall}) essentially by construction of $\KSL\to \KWall$. Similarly, the first composite is determined up to homotopy by $([\struct(-1)\vert_{S}],\infty)\colon S\to \K\times\Proj^1$. Hence, choosing an isomorphism $\struct(-1)\vert_{S}\cong\struct$ we obtain the desired homotopy. The argument for the rank zero presheaves is similar.
\end{proof}
\begin{corollary}\label{cor:cofiber}
    Multiplication with the motivic Hopf element induces the cofiber sequence in $\SH(S)$ \[\Sigma^{1,1}\MSL\xrightarrow{\eta} \MSL\xrightarrow{\forg} \MW.\]
\end{corollary}
\begin{proof}
    By Theorem \ref{thm:msl/eta=mwl}, it is enough to show that the composition $\MSL\to \MSL/\eta\xrightarrow{\simeq} \MW$ is homotopic to the morphism $\forg\colon\MSL\to \MW$. This is true by the previous lemma.
\end{proof}
\subsection{The \texorpdfstring{$c_1$}{c1}-spherical algebraic cobordism spectrum as a fiber}
The equivalences from the previous subsection (see Corollary \ref{corollary:mgl=mslxpinf} and Theorem \ref{thm:msl/eta=mwl}) are compatible in the sense that the following diagram commutes \[ \xymatrixcolsep{4pc}\xymatrix{\MSL\wedge \Th_{\Proj^1}(\struct(-1)\ominus \struct) \ar[r]^{\id\wedge \Th(\mathrm{in})} \ar[d]^\simeq & \MSL\wedge \Th_{\Proj^\infty}(\struct(-1)\ominus \struct) \ar[d]^\simeq \\ \MW \ar[r]^{\bforg} & \MGL,} \] where $\Th(\mathrm{in})$ is the map on Thom spaces induced by the embedding $\mathrm{in}\colon\Proj^1\hookrightarrow\Proj^\infty$ and $\bforg$ is the morphism from Definition \ref{definition:MWL}. Thus, the cofiber of $\bforg$ is equivalent to the cofiber of the top arrow. To compute it, we use motivic equivalences $\Th_{\Proj^1}(\struct(-1))\cong_\mot\Proj^2$ and $\Th_{\Proj^\infty}(\struct(-1))\cong_\mot\Proj^\infty$. These identifications follow from the diagram \begin{equation}\label{equation:diagram}
\begin{gathered}\xymatrix{\A^{n+1}\setminus\{0\} \ar[r] & \Proj^{n} \ar[r] & \Proj^{n+1} \\ \struct(-1)^\circ \ar[u]_-{\simeq_\mot} \ar@{^{(}->}[r] & \struct(-1) \ar[u]_-{\simeq_\mot} \ar[r] & \Th_{\Proj^n}(\struct(-1)), \ar[u] }\end{gathered}\end{equation} (see e.g., \cite[Lemma 6.2.1]{Mor03}). Therefore, $\cofib(\bforg)$ is isomorphic to the cofiber of $$\MSL\wedge \Sigma^{\infty-(2,1)}\Proj^2\to \MSL\wedge \Sigma^{\infty-(2,1)}\Proj^\infty,$$ which is given by $\MSL\wedge \Sigma^{\infty-(2,1)}\Proj^\infty/\Proj^2$.
\begin{lemma}\label{lemma:purity_proj_spaces}
    Let $S$ be a base scheme. Then there is a motivic equivalence $\Proj^\infty/\Proj^2\cong_\mot \Th_{\Proj^{\infty}}(\struct(1)^3)$. Furthermore, the map $$\MGL^{*,*}(\Th_{\Proj^{\infty}}(\struct(1)^3))\to\widetilde{\MGL}{}^{*,*}(\Proj^\infty)$$ obtained by pullback along the morphism $\Proj^\infty\to\Proj^\infty/\Proj^2\cong_\mot\Th_{\Proj^{\infty}}(\struct(1)^3)$ sends the Thom class $\thc(\struct(1)^3)$ to $c_1(\struct(1))^3$. Here the right hand side is $\MGL$-cohomology of $\Proj^\infty$ pointed at any point of the image $\Proj^2\hookrightarrow\Proj^\infty$.
\end{lemma}
\begin{proof}
    For $n>2$, consider the embedding of $\Proj^{n-3}$ into $\Proj^{n}$ as the linear subspace of points whose first three coordinates are zero. Then, the map $\Proj^2\hookrightarrow \Proj^n$ is the zero section of the vector bundle $\Proj^n\setminus \Proj^{n-3}\to \Proj^2$ given by projection onto the first three coordinates. Hence, we have motivic equivalences \[\Proj^n/\Proj^2\cong_\mot \Proj^n/(\Proj^n\setminus \Proj^{n-3})\cong_\mot \Th_{\Proj^{n-3}}(\struct(1)^3),\] where the last isomorphism is given by the homotopy purity; see \cite[Theorem 2.23]{MV99}. Consider the morphism \[\MGL^{*,*}(\Proj^{n-3})\xrightarrow{\simeq}\MGL^{*+6,*+3}(\Th_{\Proj^{n-3}}(\struct(1)^3))\to\widetilde{\MGL}{}^{*+6,*+3}(\Proj^n), \] where the left map is the Thom isomorphism in $\MGL$-cohomology and the right one is the pullback along $\Proj^n\to\Th_{\Proj^{n-3}}(\struct(1)^3)$. This composition is the pushforward along the closed embedding $\Proj^{n-3}\to\Proj^n$ (see \cite[\S2.2]{Pan09}), which is the multiplication by $c_1(\struct(1))^3$. Consequently, the second homomorphism sends the Thom class to $c_1(\struct(1))^3$. Taking the colimit over $n$, we obtain the claim.
\end{proof}
\begin{lemma}\label{lemma:cofib_step1}
    Let $S$ be a base scheme. Then there is an equivalence in $\SHk$
    \[\cofib(\bforg\colon\MW\to \MGL)\cong \MSL\wedge \Th_{\Proj^{\infty}}(\struct(1)^3\ominus \struct).\]
\end{lemma}
\begin{proof}
    Follows immediately from the previous lemma and the above discussion. 
\end{proof}
For the next step in the computation of the cofiber we prove a slight improvement of the Thom isomorphism. We state it for an arbitrary $\SL$-oriented homotopy commutative ring spectrum $\EE$; see \cite[\S 3]{Ana15} and \cite[\S 5]{PW22} for further information. The reader can assume that $\EE=\MSL$.

Let $\EE\in \SHk$ be an $\SL$-oriented homotopy commutative ring spectrum over a base scheme $S$ and let $(E,\lambda)$ be a rank $n$ special linear vector bundle over a smooth $S$-scheme $X\in \Smk$. Recall that it means that $E$ is a rank $n$ vector bundle over $X$ and $\lambda\colon\det(E)\xrightarrow{\simeq} \struct_X$ is a trivialization of the determinant of $E$. Also let $F$ be a vector bundle over $X$. Consider the composition \begin{gather*} \EE\wedge \Th_X(E\oplus F)\xrightarrow{\id\wedge \Th(\Delta)} \EE\wedge \Th_{X\times_S X}(E\boxplus F)\cong \EE\wedge \Th_X(E)\wedge\Th_X(F)\xrightarrow{\id\wedge \thc(E,\lambda)\wedge \id} \\ \xrightarrow{\id\wedge \thc(E,\lambda)\wedge \id} \EE\wedge\Sigma^{2n,n}\EE\wedge\Th_X(F)\xrightarrow{m_\EE\wedge \id} \EE\wedge\Sigma^{2n,n}\Th_X(F), \end{gather*} where $\Th(\Delta)$ is the map induced by the diagonal morphism $\Delta\colon X\to X\times_S X$, the element $\thc(E,\lambda)\in \EE^{2n,n}(\Th_X(E))$ is the Thom class of the special linear vector bundle $(E,\lambda)$, and $m_\EE\colon\EE\wedge \EE\to \EE$ is the multiplication map. Note, that if $F=0$ then the above composition is the standard Thom isomorphism; see \cite[Lemma 4.11]{BHop}.  For a general $F$ it is an equivalence as well.
\begin{lemma}\label{lemma:thom_iso_rev}
    Let $\EE\in \SHk$ be an $\SL$-oriented homotopy commutative ring spectrum, let $(E,\lambda)$ be a special linear vector bundle over $X\in \Smk$ and let $F$ be a vector bundle over $X$. Then the map $\EE\wedge \Th(E\oplus F)\to \EE\wedge \Sigma^{2n,n}\Th(F)$ constructed above is an equivalence in $\SHk$.
\end{lemma}
\begin{proof}
     A cohomological version of this fact is treated in \cite[Lemma 3.6]{Ana16}. Below we repeat the proof of \textit{loc. cit.} at the level of spectra. Denote by $p\colon F^\circ\to X$ the projection from the complement of the zero section of $F$ and consider the following diagram
    \[\xymatrix{ \EE\wedge \Th_{F^\circ}(p^*E) \ar[r] \ar[d]^-\simeq & \EE\wedge \Th_X(E) \ar[r] \ar[d]^-\simeq & \EE\wedge\Th_X(E\oplus F) \ar[d] \\ \EE\wedge\Sigma^{\infty+(2n,n)}_+ F^\circ \ar[r] & \EE\wedge \Sigma^{\infty+(2n,n)}_+ X \ar[r] & \EE\wedge\Sigma^{2n,n}\Th_X(F), } \]
    where the middle and left vertical morphisms are the Thom isomorphisms for the special linear vector bundles $(E,\lambda)$ and $(p^*E,p^*\lambda)$ respectively (see e.g., \cite[Lemma 4.11]{BHop}), the right vertical morphism is our desired map, the bottom horizontal morphisms are induced by the definition of $\Th_X(F)$, the top left horizontal arrow is induced by $p$, and the top right horizontal arrow is induced by the inclusion $E=E\oplus 0\hookrightarrow E\oplus F$. It is easy to see that the rows are cofiber sequences in $\SHk$. Hence, it suffices to prove that the diagram commutes (up to homotopy). The left square commutes by the functoriality of the Thom isomorphism, while the commutativity of the right one reduces to 
    \[ \xymatrix{\Th_X(E) \ar[r] \ar[d]^{\Th(\Delta)} & \Th_X(E\oplus F) \ar[d]^{\Th(\Delta)} \\ \Th_{X\times_S X}(E\boxplus 0) \ar[r] \ar[d]^-\simeq & \Th_{X\times_S X}(E\boxplus F) \ar[d]^-\simeq \\ \Th_X(E)\wedge X_+ \ar[r] & \Th_X(E)\wedge\Th_X(F), \ } \]
    where the horizontal maps are induced by $0\hookrightarrow F$. This diagram commutes up to homotopy already in the $\infty$-category of pointed Nisnevich sheaves (but not in $\PShv(\Smk)$ since the vertical equivalences are not defined there).
\end{proof}
\begin{remark}
    Of course, a similar statement is true more generally for a vector $G$-bundle $E$ and $G$-oriented homotopy commutative ring spectrum $\EE\in \SHk$ for any $G=\GL,\,\SL,\,\mathrm{SL^c},\,\Sp$; see \cite[\S 3]{Ana20} for a discussion of these notions.
\end{remark}
\begin{proposition}\label{prop:cofib}
    Let $S$ be a base scheme. Then there is an equivalence in $\SHk$
    \[\cofib(\bforg\colon\MW\to \MGL)\cong \Sigma^{4,2}\MGL.\]
\end{proposition}
\begin{proof}
    By Lemma \ref{lemma:cofib_step1}, there is an equivalence $\cofib(\bforg)\cong \MSL\wedge \Th(\struct(1)^3\ominus \struct)$. Passage to the dual line bundles via \cite[Lemma 4.1]{Ana20} yields an isomorphism \[\cofib(\bforg)\cong\MSL\wedge \Th(\struct(1)\oplus\struct(-1)^2\ominus \struct).\] Applying the above lemma with $\EE=\MSL$, $E=\struct(1)\oplus\struct(-1)$ and $F=\struct(-1)$, we obtain an equivalence \[\cofib(\bforg)\cong \MSL\wedge \Th(\struct(-1)\oplus \struct)\cong\Sigma^{4,2}\MSL\wedge\Th(\struct(-1)\ominus\struct),\] and the result follows from Corollary \ref{corollary:mgl=mslxpinf}.
\end{proof}
The corresponding morphism $\MGL\to \Sigma^{4,2}\MGL$ is given by the chain of arrows \begin{equation}\label{chain} \begin{gathered} \MGL\xleftarrow[(6)]{\simeq} \MSL\wedge \Th(\struct(-1)\ominus \struct)\xleftarrow[(5)]{\simeq} \MSL\wedge \Sigma^{\infty-(2,1)} \Proj^\infty\xrightarrow[(4)]{} \MSL\wedge \Th(\struct(1)^3\ominus \struct) \xrightarrow[(3)]{\simeq} \\ \xrightarrow[(3)]{\simeq} \MSL\wedge \Th(\struct(1)\oplus \struct(-1)^2\ominus \struct)\xrightarrow[(2)]{\simeq} \MSL\wedge \Th(\struct(-1)\oplus \struct)\xrightarrow[(1)]{\simeq} \Sigma^{4,2}\MGL, \end{gathered} \end{equation}
where (1) and (6) come from Corollary \ref{corollary:mgl=mslxpinf}, (2) and (3) come from the (proof of the) previous proposition, (4) comes from Lemma \ref{lemma:purity_proj_spaces}, and (5) is discussed at the beginning of this subsection. All these maps, except for (4), are equivalences based on what we proved above. Notice, that all morphisms here are maps of $\MSL$-modules. We denote by $\Phi$ the motivic equivalence $\mathrm{B}\SL\times \Proj^\infty\cong_\mot \mathrm{B}\GL$ and by
$\Th(\Phi)$ the induced equivalence of the Thom spectra from Corollary \ref{corollary:mgl=mslxpinf}.

Consider the stable $\infty$-category of $\MSL$-modules $\mathrm{Mod}_{\MSL}(\SHk)$. By the usual free-forgetful adjunction \[\MSL\wedge-\colon\SHk\rightleftarrows \mathrm{Mod}_{\MSL}(\SHk)\colon \mathrm{Forg}\] there is an isomorphism $[\MSL\wedge X,Y]_{\MSL}\cong[X,Y]$ for $X\in \SHk$ and $Y\in \mathrm{Mod}_{\MSL}(\SHk)$. It is given by $$(f\colon\MSL\wedge X\to Y)\mapsto(X=\sph\wedge X\xrightarrow{1\wedge\, \id} \MSL\wedge X\xrightarrow{f} Y)$$ and $$(g\colon X\to Y)\mapsto(\MSL\wedge X\xrightarrow{\id\wedge\, g} \MSL\wedge Y\xrightarrow{\mathrm{act}} Y).$$
\begin{theorem}\label{thm:fiber_seq}
    Let $S$ be a base scheme. Then the forgetful map  induces the cofiber sequence in $\SHk$ \[ \MW\xrightarrow{\bforg} \MGL\xrightarrow{\Delta} \Sigma^{4,2} \MGL, \] where the operation $\Delta$ corresponds to the characteristic class $c_1(\det\ggl)\cdot  c_1(\det\ggl^\vee)$ under the Thom isomorphism $\MGL^{4,2}(\MGL)\cong \MGL^{4,2}(\mathrm{B}\GL)$.
\end{theorem}
\begin{proof}
    By the previous proposition there is a cofiber sequence $\MW\xrightarrow{\bforg} \MGL\to \Sigma^{4,2}\MGL$, where the second map is given by the chain of morphisms \eqref{chain}. Hence, it remains to verify that this composition is homotopic to $\Delta$. The idea is to compute the pullback of $\id\in [\Sigma^{4,2}\MGL,\Sigma^{4,2}\MGL]_\MSL$ along this chain. Let us explain why we do that in the stable $\infty$-category of $\MSL$-modules. The maps (3)--(5) are morphisms of free $\MSL$-modules, which are constant on $\MSL$. Combined with the free-forgetful adjunction (see above), this observation allows us to ignore the additional $\MSL$ factor in these steps and work with the $\MGL$-cohomology of the corresponding Thom spaces over $\Proj^\infty$ directly. We stress that below all Thom classes are Thom classes in $\MGL$-cohomology for the canonical orientation of $\MGL$. \begin{enumerate} \item The pullback of $\id$ along $\Sigma^{4,2}\Th(\Phi)$ is equal to $\Sigma^{4,2}\Th(\Phi)$. It can be viewed as the Thom class of the Thom spectrum $\MSL\wedge \Th(\struct(-1)\oplus \struct)$ in $\MGL$-cohomology. \item A straightforward computation shows that the pullback of $\Sigma^{4,2}\Th(\Phi)$ along (2) is given by the respective Thom class. It corresponds to $\thc(\struct(1)\oplus \struct(-1)^2\ominus \struct)\in [\Th(\struct(1)\oplus \struct(-1)^2\ominus \struct),\Sigma^{4,2}\MGL]$ under the bijection \[[\MSL\wedge\Th(\struct(1)\oplus \struct(-1)^2\ominus \struct),\Sigma^{4,2}\MGL]_\MSL\cong [\Th(\struct(1)\oplus \struct(-1)^2\ominus \struct),\Sigma^{4,2}\MGL].\] \end{enumerate} 
    As already mentioned at the beginning, we can omit the smash product with $\MSL$ in the next three steps thanks to the free-forgetful adjunction. \begin{enumerate}\setcounter{enumi}{2} \item By the construction of equivalence $\Th(\struct(1)^3)\cong \Th(\struct(1)\oplus \struct(-1)^2)$ (see \cite[Lemma 4.1]{Ana20}) the following diagram commutes \[\xymatrix{\Th(\struct(1)^3) \ar[rr]^-\simeq & & \Th(\struct(1)\oplus \struct(-1)^2) \\ & \Proj^\infty_+, \ar[lu] \ar[ru] & }\] where the diagonal maps are induced by the zero sections. Denote by $\chi(x)$ the formal inverse to $x$ with respect to the formal group law of $\MGL^{*,*}(S)$. It follows from the diagram that the induced map on the $\MGL$-cohomology sends the Thom class $\thc(\struct(1)\oplus \struct(-1)^2)$ to $g(h)^2\cdot\thc(\struct(1)^3)$, where $g(h)\in \MGL^{0,0}(\Proj^\infty)$ is a power series in $h=c_1(\struct(1))$ defined by the formula $g(h)=\chi(h)/h$. Hence, the desired element is equal to $g(h)^2\cdot\thc(\struct(1)^3\ominus\struct)$. \item The pullback along (4) sends the Thom class $\thc(\struct(1)^3\ominus \struct)$ to $\Sigma^{-2,-1}h^3$ by Lemma \ref{lemma:purity_proj_spaces}. Therefore, we have an element $\Sigma^{-2,-1}h\cdot\chi(h)^2\in \widetilde{\MGL}{}^{4,2}(\Proj^\infty)$. \item The pullback along the equivalence $\Proj^\infty\xrightarrow{\simeq}\Th(\struct(-1))$ is given by $\thc(\struct(-1))\mapsto e(\struct(-1))=\chi(h)$. It follows that our composition corresponds to $h\cdot\chi(h)\cdot\thc(\struct(-1)\ominus\struct)$. \end{enumerate} Now we have the commutative diagram \[\xymatrix{\MGL^{4,2}(\Proj^\infty) \ar[r]^-\simeq \ar[d] & \MGL^{4,2}(\Th(\struct(-1)\ominus \struct)) \ar[d] \\ \MGL^{4,2}(\mathrm{B}\SL\times \Proj^\infty) \ar[r]^-\simeq & \MGL^{4,2}(\MSL\wedge \Th(\struct(-1)\ominus \struct)) \\ \MGL^{4,2}(\mathrm{B}\GL) \ar[u]^-\simeq_-{\Phi^*} \ar[r]^-\simeq & \MGL^{4,2}(\MGL), \ar[u]^-\simeq_-{\Th(\Phi)^*} }\] where the horizontal arrows are given by the Thom isomorphisms and the top vertical maps are induced by the multiplication with the canonical elements of $[\mathrm{B}\SL,\MGL]$ and $[\MSL,\MGL]$ respectively. Thus, the resulting characteristic class is equal to $h\cdot\chi(h)\in \MGL^{4,2}(\mathrm{B}\SL\times \Proj^\infty)$. To conclude the proof, it remains to note that $c_1(\det\ggl)$ goes to $h$ under the isomorphism $\Phi^*\colon\MGL^{4,2}(\mathrm{B}\GL)\xrightarrow{\simeq}\MGL^{4,2}(\mathrm{B}\SL\times \Proj^\infty)$. This follows from the construction of $\Phi$ and the functoriality of the Chern classes.
\end{proof}
\section{Cohomological operations and homotopy groups of \textrm{MWL}}\label{section-4}
Let $S$ be a base scheme. Then Theorem \ref{thm:fiber_seq} leads to the exact sequences \begin{equation}\label{equation:fiber} \cdots\to\pi_{i-3,j-2}(\MGL)\to \pi_{i,j}(\MW)\xrightarrow{\bforg_*} \pi_{i,j}(\MGL)\xrightarrow{\Delta_*} \pi_{i-4,j-2}(\MGL)\to \cdots.\end{equation} In this section we use them to compute some homotopy groups of $\MW$ over specific bases, see Proposition \ref{prop:wall_geom_part}. Afterwards, we introduce a motivic version of the Conner--Floyd homology and get a partial computation of these groups, see Theorem \ref{theorem:conner_floyd_hom}.
\subsection{Homotopy groups of \texorpdfstring{$\MW$}{MWL}}
Given a cohomological operation $\varphi\in \MGL^{i,j}(\MGL)$, it induces an action on the coefficient ring $\varphi_*\colon\pi_{*,*}(\MGL)\to \pi_{*-i,*-j}(\MGL),\ f\mapsto\varphi\circ f$. This gives $\pi_{*,*}(\MGL)$ the structure of a left $\MGL^{*,*}(\MGL)$-module. 
The morphism $\mathrm{in}\colon\Th(\struct(-1)\ominus\struct)\to \MGL$ together with the standard motivic equivalence $\Th(\struct(-1))\cong_{\mot}\Proj^\infty$ endow $\MGL$ with the canonical orientation, which induces a graded formal group law $\mathbb{L}\to \pi_{2*,*}(\MGL)$. Similarly, the canonical complex orientation of $\MU$ induces a homomorphism $\mathbb{L}\to \pi_*(\MU)$. 
Hence, we have the following span
\begin{align}\label{span_algebroids} \xymatrix{& \mathbb{L} \ar[rd] \ar[ld]_-\simeq & \\ \pi_{2*}(\MU) & & \pi_{2*,*}(\MGL). } \end{align} The left arrow is an isomorphism due to a result of Quillen (see \cite{Quillen}) and we will usually identify these rings. 
The right morphism becomes an isomorphism after inverting the residue characteristic due to results of Hopkins--Morel, Hoyois, and Spitzweck:
\begin{proposition}\label{prop:hopf_algebroids}
    Let $S$ be a spectrum of a local Dedekind domain $k$ and let $e$ be the exponential characteristic of the residue field of $k$. Then the homomorphism \[ \pi_{2*}(\MU)[\einv]\to \pi_{2*,*}(\MGL)[\einv]\] classifying the formal group law of $\MGL[\einv]$ is an isomorphism. 
\end{proposition}
\begin{proof}
    If $k$ is a field, the result is proven in \cite[Proposition 8.2]{Hoy15}. The case of discrete valuation rings is considered in \cite[Theorem 6.7]{SpiMGL}. 
\end{proof}
\begin{definition}
    We say that characteristic classes $c\in \MGL^{2*,*}(\mathrm{B}\GL)[\einv]$ and $c'\in \MU^{2*}(\BU)[\einv]$ are \textit{the same} if they both correspond to the same element of $\mathbb{L}[\einv][\![c_1,c_2,\dots]\!]$ under the isomorphisms
    \[\xymatrix{\MU^{2*}(\BU)[\einv] & \mathbb{L}[\einv][\![c_1,c_2,\dots]\!] \ar[r]^-\simeq \ar[l]_-\simeq & \MGL^{2*,*}(\mathrm{B}\GL)[\einv], }\]
    which are induced by the span \eqref{span_algebroids} and the canonical orientations of $\MU$ and $\MGL$ (see \cite[Proposition 6.2]{NSOLandw} for the computation of $\MGL^{*,*}(\mathrm{BGL})$).
\end{definition}
We begin with the next lemma. We stress that if $f\colon k\to k'$ is a homomorphism of local Dedekind domains then the exponential characteristic of the residue field of $k'$ is either equal to the exponential characteristic of the residue field of $k$ or equal to $1$.
\begin{lemma}\label{lemma:charclasses_and_operations}
    Let $k$ and $e$ be as above. Suppose that $\varphi\in \MGL^{2n,n}(\MGL)[\einv]$ and $\varphi'\in \MU^{2n}(\MU)[\einv]$ are cohomological operations that correspond to the same characteristic classes under the Thom isomorphisms. Then the following diagram commutes \[\xymatrix{\pi_{2*}(\MU)[\einv] \ar[r]^-{\varphi'_*} \ar[d]_-\simeq &  \pi_{2*-2n}(\MU)[\einv] \ar[d]^-\simeq \\ \pi_{2*,*}(\MGL)[\einv] \ar[r]^-{\varphi_*} &  \pi_{2*-2n,*-n}(\MGL)[\einv].}\]
\end{lemma}
\begin{proof}
    Let $f\colon k\to k'$ be a homomorphism of local Dedekind domains. Denote the operation $\varphi$ over $k$ (respectively $k'$) by $\varphi_k$ (respectively $\varphi_{k'}$) and consider the following diagram 
    \[\xymatrix{ \pi_{2*}(\MU)[\einv] \ar[r]^-{\varphi'_*} \ar[d]^-\simeq & \pi_{2*-2n}(\MU)[\einv] \ar[d]^-\simeq \\ \pi_{2*,*}(\MGL_k)[\einv] \ar[r]^-{(\varphi_k)_*} \ar[d]^-\simeq & \pi_{2*-2n,*-n}(\MGL_k)[\einv] \ar[d]^-\simeq \\ \pi_{2*,*}(\MGL_{k'})[\einv] \ar[r]^-{(\varphi_{k'})_*} & \pi_{2*-2n,*-n}(\MGL_{k'})[\einv], }\]
    where the vertical morphisms in the bottom square are given by the base change maps. The bottom square clearly commutes, and we obtain that the result holds over $k$ if and only if it holds over $k'$ (for one implication, we use the injectivity of the homomorphism $\pi_{2*,*}(\MGL_k)[\einv]\to \pi_{2*,*}(\MGL_{k'})[\einv]$).

    Let us treat $k=\C$. The complex Betti realization sends the canonical orientation of $\MGL_\C$ to the canonical complex orientation of $\MU$ essentially by the constructions of these spectra. It follows that the realization sends Chern classes to Chern classes, the operation $\varphi_\C$ to the operation $\varphi'$, and, hence, the action of $\varphi_\C$ on the geometric diagonal $\pi_{2*,*}(\MGL_\C)$ to the action of $\varphi'$ on $\pi_*(\MU)$. This implies the result over the complex numbers.
        
    Now we prove the remaining cases by various base changes using the above observation. First, we can extend the result from $\C$ to $\mathbb{Q}$ and hence to an arbitrary field of characteristic zero. Second, we get the desired statement for a discrete valuation ring of mixed characteristic passing to the fraction field. For the next step, assume that $k$ is a field of positive characteristic $p=e>1$. Then there exists a discrete valuation ring of mixed characteristic $R$ such that its residue field is $k$. If $k$ is a perfect field, we can choose $R$ to be the ring of $p$-adic Witt vectors $\mathbb{W}_{p^\infty}(k)$ \cite[Theorems 6.19, 6.20]{HAZEWINKEL}. In general, we can take the Cohen ring of $k$ for $R$ (see \S 6.23 of \textit{loc. cit.}). Hence, base change along $R\to k$ solves this step. The situation with an equicharacteristic discrete valuation ring is similar to the mixed characteristic case.
    %
\end{proof}
\begin{remark}
    Roughly speaking, the above lemma says that the left action of $\MGL^{2*,*}(\MGL)[\einv]$ on the $\Proj^1$-diagonal $\pi_{2*,*}(\MGL)[\einv]$ is the same as in topology. We usually denote cohomological operations that correspond to the same characteristic classes by the same letter. To distinguish them, we label the base ring near the motivic operation as in the proof of the previous lemma.
\end{remark}
\begin{lemma}\label{lemma:vanishing_mgl}
    Suppose that $k$ is a local Dedekind domain and $e$ is the exponential characteristic of the residue field. Then we have $\pi_{i,j}(\MGL)[\einv]=0$ for $i<2j$ or $i<j$.
\end{lemma}
\begin{proof}
    This is a combination of the results of the paper \cite{SpiMGL} (see also \cite[Theorem 2.1]{LYZ} for the case of a field). By \cite[Corollary 4.6 and Proposition 7.1]{SpiMGL} the homotopy groups of $\MGL[\einv]$ vanish for $i+1<j$. If $i+1=j$ the desired group is isomorphic to $\pi_{i,i+1}(\HZ)[\einv]$ by the discussion just after Proposition 7.8 of \textit{loc. cit.}, which is trivial since $k$ is local (see e.g., \cite[Corollary 4.4]{Geisser}). Here $\HZ$ is Spitzweck's motivic cohomology spectrum \cite{SpiHZ}. The claim for $i<2j$ follows from \cite[Proposition 7.8 and below]{SpiMGL}.
\end{proof}
\begin{proposition}\label{prop:vanishing}
    Let $k$ and $e$ be as above. Then we have $\pi_{i,j}(\MW)[\einv]=0$ for $i<2j$ or $i<j$.
\end{proposition}
\begin{proof}
    We implicitly invert $e$ throughout this proof. Consider the following part of the exact sequence \eqref{equation:fiber}
    $$ \pi_{i+1,j}(\MGL)\xrightarrow{(\Delta_k)_*} \pi_{i-3,j-2}(\MGL)\to \pi_{i,j}(\MW)\xrightarrow{\bforg_*} \pi_{i,j}(\MGL). $$ Assume that $i<j$. Then $i-3<j-2$ and we obtain the vanishing of the second and fourth groups in the exact sequence by the previous lemma. Hence, we have $\pi_{i,j}(\MW)=0$ for $i<j$. If $i<2j-1$ then $i-3<2j-4$ and the same argument shows $\pi_{i,j}(\MW)=0$ for such $i$ and $j$. Therefore, it remains to consider $i=2j-1$. In this case the right group in the above exact sequence is trivial by the previous lemma again since $i<2j$. Consequently, $\pi_{i,j}(\MW)$ is the cokernel of the homomorphism \[ (\Delta_k)_*\colon\pi_{2j,j}(\MGL)\to\pi_{2j-4,j-2}(\MGL). \] According to Lemma \ref{lemma:charclasses_and_operations} this morphism can be identified with $\Delta_*\colon\pi_{2j}(\MU)\to\pi_{2j-4}(\MU),$ which is surjective by Proposition \ref{appendix:split_cofib}.
\end{proof}
We denote by $\slice_q(-)$ the $q$-th slice functor; see \cite[\S2]{RO16} for an overview of the slice filtration. By \cite[Theorem 3.1]{SpiMGL}, the slices of $\MGL[\einv]$ are given by $\slice_q(\MGL)[\einv]\cong\Sigma^{2q,q}\mathrm{H}(\pi_{2q}(\MU)[\einv])$, where $\mathrm{H}A\in\SH(k)$ is the motivic cohomology spectrum with $A$-coefficients.
\begin{lemma}
    Let $k$ and $e$ be as above. The homomorphism $\pi_{2n+1,n}(\MGL)[\einv]\to \pi_{2n-3,n-2}(\MGL)[\einv]$ induced by $\Delta$ is surjective for $n\in\Z$.
\end{lemma}
\begin{proof}
    We implicitly invert $e$ below. First note that the group $\pi_{-3,-2}(\HZ)=H^3(k,\Z(2))$ is trivial since $k$ is local (see \cite[Corollary 4.4]{Geisser}). Therefore, by the proof of \cite[Proposition 7.7]{SpiMGL} there are isomorphisms \[\pi_{2n+1,n}(\MGL)\cong \pi_{2n+1,n}(\slice_{n+1}(\MGL))\cong \pi_{-1,-1}(\mathrm{H}\pi_{2n+2}(\MU))\cong \pi_{2n+2}(\MU)\otimes k^*.\] Hence, the desired homomorphism is obtained by applying $\pi_{2n+1,n}$ to $\slice_{n+1}(\Delta_k)$. The map on the $(n+1)$-th slices is given by the $(2n+2,n+1)$-suspension of $\mathrm{\mathrm{H}\pi_{2n+2}(\MU)\to\mathrm{H}\pi_{2n-2}(\MU)}$, which is induced by $\Delta_*\colon\pi_{2n+2}(\MU)\to\pi_{2n-2}(\MU)$ (to see this combine \cite[Lemma A.3]{RSO19}, Lemma \ref{lemma:charclasses_and_operations}, and the proof of \cite[Theorem 6.7]{SpiMGL}). The result follows from surjectivity of $\Delta_*$; see Proposition \ref{appendix:split_cofib}.
\end{proof}
\begin{proposition}\label{prop:wall_geom_part}
    Let $k$ and $e$ be as above. The isomorphism $\pi_{2*}(\MU)[\einv]\xrightarrow{\simeq} \pi_{2*,*}(\MGL)[\einv]$ classifying the formal group law of $\MGL[\einv]$ restricts to the isomorphism $\pi_{2*}(\W)[\einv]\xrightarrow{\simeq} \pi_{2*,*}(\MW)[\einv]$.
\end{proposition}
\begin{proof}    
    By the previous lemma, exact sequence \eqref{equation:fiber}, Lemma \ref{lemma:charclasses_and_operations} and Proposition \ref{appendix:split_cofib}, we have the commutative diagram with exact rows \[\xymatrixcolsep{3pc}\xymatrix{0\to \pi_{2*}(\W)[\einv] \ar[r] \ar@{-->}[d] & \pi_{2*}(\MU)[\einv] \ar[r]^-{\Delta_*[\einv]} \ar[d]^\simeq & \pi_{2*-4}(\MU)[\einv] \ar[d]^\simeq \\ 0\to \pi_{2*,*}(\MW)[\einv] \ar[r] & \pi_{2*,*}(\MGL)[\einv] \ar[r]^-{(\Delta_k)_*[\einv]} & \pi_{2*-4,*-2}(\MGL)[\einv].} \] The middle and the right vertical maps are isomorphisms by Proposition \ref{prop:hopf_algebroids}. Thus, the left vertical map is an isomorphism.
\end{proof}
\subsection{The algebraic Conner--Floyd homology}\label{section:4_2}
Consider the $\eta$-Bockstein morphism \[\delta:=\Sigma^{2,1}\forg\,\circ\,d\colon\MW\to\Sigma^{2,1}\MW,\] where $d$ is the boundary map in the cofiber sequence from Corollary \ref{cor:cofiber}. By construction, $\delta^2=(\Sigma^{2,1}\delta)\circ\delta$ is nullhomotopic, and there are chain complexes of abelian groups \[\cdots\to\pi_{i+2,j+1}(\MW)\xrightarrow{\delta_*}\pi_{i,j}(\MW)\xrightarrow{\delta_*}\pi_{i-2,j-1}(\MW)\to\cdots.\] Denote by $\Hml_{i,j}(\MW,\delta)$ (respectively $\Cyc_{i,j}(\MW,\delta)$, $\Bnd_{i,j}(\MW,\delta)$) their homology (respectively cycles, boundaries). These homology groups are a motivic version of the Conner--Floyd homology \cite{CF} (see also a brief overview in Appendix \ref{appendix_B}).

\begin{lemma}
    Let $f\colon k\to k'$ be a homomorphism of local Dedekind domains and let $e$ be the exponential characteristic of the residue field of $k$. Then the base change along $f$ induces an isomorphism \[\MGL^{2*,*}(\MW_k)[\einv]\cong \MGL^{2*,*}(\MW_{k'})[\einv].\]
\end{lemma}
\begin{proof}
    Consider the cofiber sequence in $\SH(k)$ \[ \Sigma^\infty_+ \WGr_n\to \Sigma^\infty_+ \Gr_n\times \Proj^1\to \Sigma^\infty \Th(\det(\ggl_{n})\boxtimes \struct(1)). \] Taking the associated long exact sequence of the $\MGL$-cohomology groups, we obtain the commutative diagram with an exact row \[ \xymatrixcolsep{1pc}\xymatrix{\cdots \ar[r] & \MGL^{i,j}(\Th(\det(\ggl_{n})\boxtimes \struct(1))) \ar[r] &  \MGL^{i,j}(\Gr_n\times \Proj^1) \ar[r] & \MGL^{i,j}(\WGr_n) \ar[r] & \cdots \\ & \MGL^{i-2,j-1}(\Gr_n\times \Proj^1), \ar[u]^\simeq \ar@{-->}[ur] & & & } \] where the vertical map is the Thom isomorphism and the diagonal arrow is the multiplication by $c_1(\det(\ggl_{n})\boxtimes \struct(1))$. From the projective bundle formula and the computation of the $\MGL$-cohomology of $\Gr_n$ \cite[Proposition 6.2]{NSOLandw}, we see that the diagonal morphism is injective. Therefore, the ring $\MGL^{*,*}(\WGr_n)$ is the quotient of $\MGL^{*,*}(\Gr_n\times \Proj^1)$ by the ideal generated by $c_1(\det(\ggl_{n})\boxtimes \struct(1))$. Restricting to the $\Proj^1$-diagonal, we have that the base change along $f$ yields an isomorphism \[\MGL^{2*,*}(\WGr_{n,k})[\einv]\cong \MGL^{2*,*}(\WGr_{n,k'})[\einv].\] Applying the Thom isomorphisms, we get the result for $\Th(\gwall_{n}\ominus \struct^n)$. Finally, the claim follows from the Milnor exact sequence and Proposition \ref{prop:wall_colim}.
\end{proof}
\begin{remark}
    A detailed analysis of the above proof gives a complete description of the $\EE$-cohomology of the Wall Grassmannians $\WGr_n$ and the $c_1$-spherical algebraic cobordism spectrum $\MW$ for an arbitrary oriented homotopy commutative ring spectrum $\EE$.
\end{remark}
\begin{lemma}\label{lemma:boundary_map}
    Let $k$ and $e$ be as above. Suppose that $\partial\in \MGL^{2,1}(\MGL)$ is a cohomological operation that corresponds to the characteristic class $c_1(\det\ggl^\vee)$ under the Thom isomorphism. Then the following diagram commutes up to homotopy after inverting $e$ \[\xymatrix{\MW \ar[r]^-{\delta} \ar[d]_-{\bforg} & \Sigma^{2,1}\MW \ar[d]^-{\Sigma^{2,1}\bforg} \\ \MGL \ar[r]^-{-\partial} & \Sigma^{2,1}\MGL. }\]
\end{lemma}
\begin{proof}
    We implicitly invert $e$ throughout this proof. First, assume that $f\colon k\to k'$ is a homomorphism of local Dedekind domains. Then there is an isomorphism $f^*\colon\MGL^{2*,*}(\MW_k)\xrightarrow{\simeq} \MGL^{2*,*}(\MW_{k'})$ by the previous lemma. Moreover, it sends the above diagram over $k$ to the diagram over $k'$. Therefore, we see that the statement holds over $k$ if and only if it holds over $k'$.

    Let us treat $k=\C$. In this case the complex Betti realization $\BettiC$ induces an isomorphism $\MGL^{2,1}(\MW_\C)$ $\cong\MU^2(\W)$ and the result follows from the corresponding topological counterpart; see Lemma \ref{appendix:lemma_boundary}. The remaining cases follow from the various base changes similarly to the proof of Lemma \ref{lemma:charclasses_and_operations}.
    %
\end{proof}
\begin{theorem}\label{theorem:conner_floyd_hom}
    Suppose that $k$ is a local Dedekind domain and $e$ is the exponential characteristic of the residue field of $k$. Then the equivalence from Proposition \ref{prop:wall_geom_part} induces isomorphisms \begin{enumerate} \item $\Cyc_{2*}(\W,\delta)[\einv]\cong \Cyc_{2*,*}(\MW,\delta_k)[\einv]$, \item $\Bnd_{2*}(\W,\delta)[\einv]\cong \Bnd_{2*,*}(\MW,\delta_k)[\einv]$, \item $\Hml_{2*}(\W,\delta)[\einv]\cong \Hml_{2*,*}(\MW,\delta_k)[\einv]$. \end{enumerate} Given a homomorphism of local Dedekind domains $f\colon k\to k'$, the induced base change maps between the groups of cycles (respectively boundaries, homology) are isomorphisms. Moreover, the group of cycles $\Cyc_{2*,*}(\MW,\delta_k)[\einv]$ is a subring of $\pi_{2*,*}(\MGL)[\einv]$, and the first equivalence is an isomorphism of rings.
\end{theorem}
\begin{proof}
    We implicitly invert $e$ below. A straightforward verification using the previous lemma and Lemma \ref{lemma:charclasses_and_operations} shows that the following diagram commutes \begin{center}\begin{tikzcd}[column sep={8em,between origins},row sep=2em] & \pi_{2*}(\W) \arrow[rr,rightarrowtail] \arrow[dl,swap,"\delta_*"] \arrow[dd,swap,"\simeq" near start] && \pi_{2*}(\MU) \arrow[dd,"\simeq"] \arrow[dl,"-\partial_*"] \\ \pi_{2*-2}(\W) \arrow[rr,rightarrowtail,crossing over] \arrow[dd,swap,"\simeq"] && \pi_{2*-2}(\MU) \\ & \pi_{2*,*}(\MW) \arrow[rr,rightarrowtail] \arrow[dl,swap,"(\delta_k)_*"] && \pi_{2*,*}(\MGL) \arrow[dl,"-(\partial_k)_*"] \\ \pi_{2*-2,*-1}(\MW) \arrow[rr,rightarrowtail] && \pi_{2*-2,*-1}(\MGL). \arrow[uu,leftarrow,crossing over,swap,"\simeq" near end] \end{tikzcd} \end{center} Claims (1)--(3) follow immediately. All these groups are stable under base change since $\pi_{2*,*}(\MGL)$ is stable under base change. Consider two homogeneous elements $a$ and $b$ of $\Cyc_{2*,*}(\MW,\delta_k)$. We need to show that the product $a\cdot b$ in the ring $\pi_{2*,*}(\MGL)$ comes from $\Cyc_{2*,*}(\MW,\delta_k)$. Lifting them to (unique) elements of $\pi_{2*}(\W)$ and using the above diagram, we reduce to the topological case. Hence, it remains to show that if $a,b\in\Cyc_{2*}(\W,\delta)$, then $\Delta_*(a\cdot b)=0$ and $\partial_*(a\cdot b)=0$. This follows from the formulas stated in Lemma \ref{appendix:lemma_ker_delta_top}. 
\end{proof}
\section{Lift of the \texorpdfstring{$\eta$}{eta}-periodic computation and Pontryagin numbers} \label{section-5}
In this section, we use the previous results to lift the computation of the homotopy groups of $\MSL[\eta^{-1}]$ to the geometric diagonal, see Corollary \ref{cor:modulo_eta_tors}. Then we introduce Pontryagin characteristic numbers and prove that in the case of the hermitian $\K$-theory they determine some homotopy groups of $\MSL$, see Corollary \ref{cor:pontryagin_8n+1}.
\subsection{Description modulo \texorpdfstring{$\eta$}{eta}-torsion}
By Corollary \ref{cor:cofiber}, we have the exact sequences \begin{equation}\label{equation:cofib}\cdots\to\pi_{i-1,j-1}(\MSL)\xrightarrow{\eta}\pi_{i,j}(\MSL)\xrightarrow{\forg_*}\pi_{i,j}(\MW)\xrightarrow{d_*}\pi_{i-2,j-1}(\MSL)\to\cdots,\end{equation}
where $\eta$ is the multiplication by the motivic Hopf element $\eta\in \pi_{1,1}(\MSL)$. Similar exact sequences exist on the level of Nisnevich sheaves.
\begin{theorem}\label{theorem:eta_stab}
    Suppose that $k$ is a local Dedekind domain and $e$ is the exponential characteristic of the residue field of $k$. For $n\in\Z$ the multiplication by the motivic Hopf element \[\eta\colon\pi_{2n+m,n+m}(\MSL)[\einv]\to\pi_{2n+m+1,n+m+1}(\MSL)[\einv]\] is an epimorphism if $m=0$, and an isomorphism if $m>0$.
\end{theorem}
\begin{proof}
    For $(i,j)=(2n+m,n+m)$ the long exact sequence \eqref{equation:cofib} has the form
    $$ \pi_{2n+m+2,n+m+1}(\MW)\to\pi_{2n+m,n+m}(\MSL)\xrightarrow{\eta}\pi_{2n+m+1,n+m+1}(\MSL)\to \pi_{2n+m+1,n+m+1}(\MW).$$
    The right group is trivial by Proposition \ref{prop:vanishing} for any $m\geq 0$ since $2n+m+1<2(n+m+1)$. In turn, for $m>0$ we have $2n+m+2<2(n+m+1)$ and the left group vanishes by the same proposition. 
\end{proof}
\begin{definition}\label{definition_eta_periodic}
    We say that the motivic spectrum $\EE$ \textit{is $\eta$-periodic} if the morphism $\eta\colon\Sigma^{1,1}\EE\to\EE$ is an equivalence. If $\EE$ is $\eta$-periodic then the homotopy groups $\pi_{*,*}(\EE)$ (or more generally the (co)homology theory represented by $\EE$) are $(1,1)$-periodic. In this situation we put $\pi_n(\EE):=\pi_{n,0}(\EE)$ and usually identify $\pi_{i,j}(\EE)$ with $\pi_{i-j}(\EE)$ via the multiplication by the appropriate power of $\eta$. We use the same convention for the (co)homology theory represented by $\EE$.
\end{definition}
\begin{corollary}\label{cor:modulo_eta_tors}
    Let $k$ and $e$ be as above, and suppose that $e\neq 2$. Then there is an isomorphism of rings \[\bigslant{\pi_{2*,*}(\MSL)}{{}_{\eta}\pi_{2*,*}(\MSL)}[\einv]\cong \W(k)[\einv][y_4,y_8,\dots],\ \text{where}\ \mathrm{deg}(y_i)=(2i,i). \] 
\end{corollary}
\begin{proof}
    We implicitly invert $e$ throughout this proof. By the previous theorem, the iterated multiplication by the motivic Hopf element $\eta$ gives \[\pi_{2n,n}(\MSL)\twoheadrightarrow\pi_{2n+1,n+1}(\MSL)\xrightarrow{\simeq}\pi_{2n+2,n+2}(\MSL)\xrightarrow{\simeq}\cdots.\] The colimit of this sequence coincides with the group $\pi_n(\MSL[\eta^{-1}])$ and the kernel of the first map is ${}_{\eta}\pi_{2n,n}(\MSL)$. Therefore, we get an isomorphism of rings \[\bigslant{\pi_{2*,*}(\MSL)}{{}_{\eta}\pi_{2*,*}(\MSL)}\xrightarrow{\simeq} \pi_*(\MSL[\eta^{-1}]).\] The right hand side, in turn, is isomorphic to the desired polynomial ring over $\W(k)$ by \cite[Corollary 1.3(3)]{BHop} and \cite[Proposition 5.6(2)]{BacDVR}.
\end{proof}
\begin{remark}
    Note that we use different conventions for the numbering and the grading of the variables $y_i$ than in the papers \cite{BHop}, \cite{BacDVR}.
\end{remark}
We add the following reformulation, which is more convenient for the further exposition.
\begin{corollary}\label{cor:eta_stab_rev}
    Let $k$ and $e\neq 2$ be as above. Then the following holds \begin{equation*} \pi_{2n+m,n+m}(\MSL)[\einv]\cong\begin{cases} \W(k)[\einv]^{p(\frac{n}{4})}, & \text{if}\ m>0\ \textrm{and}\ n\equiv 0\pmod*{4}, \\ 0, & \text{if}\ m>0\ \textrm{and}\ n\not\equiv 0\pmod*{4}, \end{cases} \end{equation*} where $p(-)$ is the partition function.
\end{corollary}
Also there is the following connectivity statement. Since this result is not used in the sequel, we only sketch its proof. If $k$ is a field, the stronger property $\underline{\pi}_{i,j}(\MSL)=0$ for $i<j$ holds without inverting the characteristic, as can be seen from the connectivity of $\MSL$ with respect to the homotopy $t$-structure.
\begin{proposition}
    Let $k$ be a discrete valuation ring and let $e\neq 2$ be the exponential characteristic of its residue field. Then we have $\underline{\pi}_{i,j}(\MSL)[\einv]=0$ for $i+1<j$. In other words, $\MSL[\einv]\in\SH(k)_{h\geq{-1}}$ in terms of \cite[\S 4]{SpiMGL}.
\end{proposition}
\begin{proof}
    We implicitly invert $e$ below. The homotopy sheaves versions of the exact sequence \eqref{equation:cofib} and Proposition \ref{prop:vanishing} say that if $i+1<j$ then the canonical map $\MSL\to \MSL[\eta^{-1}]$ induces an isomorphism $\underline{\pi}_{i,j}(\MSL)\xrightarrow{\simeq} \underline{\pi}_{i-j}(\MSL[\eta^{-1}])$. The right hand side vanishes in the negative degrees by \cite[Proposition 5.6(2)]{BacDVR}.
\end{proof}
\begin{remark}
    It seems that the exact bound should be $\MSL[\einv]\in \SH(k)_{h\geq0}$. However, it is unclear that $\underline{\pi}_{*,*+1}(\MSL)[\einv]$ is trivial. Consider the exact sequence (we omit inverting of $e$ below) \[\underline{\pi}_{n+2,n+2}(\MSL)\xrightarrow{\forg_*}\underline{\pi}_{n+2,n+2}(\MW)\xrightarrow{d_*}\underline{\pi}_{n,n+1}(\MSL)\xrightarrow{\eta}\underline{\pi}_{n+1,n+2}(\MSL)\to0.\] If we are able to prove, that the last map is a monomorphism for every $n$, then the result follows from the triviality of $\underline{\pi}_{-1}(\MSL[\eta^{-1}])$. For that, we need to prove that the first map in the exact sequence is surjective, which is clear only for $n\geq -1$.
\end{remark}
\subsection{Pontryagin characteristic numbers}\label{subsection:pontryagin_char}
Recall that for any $n\in \NN$ there are natural homomorphisms $\GL_n\to \Sp_{2n}$, $M\mapsto \mathrm{diag}(M,(M^{-1})^t)$ (see e.g., \cite[\S 5.2]{HW19}). These morphisms are compatible with stabilization and induce a symplectification morphism on the stable classifying spaces $\mathrm{B}\GL\to \mathrm{B}\Sp$. Composing this arrow with the canonical map $\mathrm{B}\SL\to \mathrm{B}\GL$, we obtain a symplectification morphism $\mathrm{B}\SL\to \mathrm{B}\Sp$.
\begin{definition}
    Let $\EE$ be an $\SL$-oriented homotopy commutative ring spectrum. The \textit{Pontryagin class} $p_n$ is the image of the Borel class $b_n\in \EE^{4n,2n}(\mathrm{B}\Sp)$ (see \cite[Definition 14.1]{PW19} and \cite[Theorem 9.1]{PW22}) under $\EE^{4n,2n}(\mathrm{B}\Sp)\to\EE^{4n,2n}(\mathrm{B}\SL).$ More generally, for a partition $\omega=(\omega_1,\omega_2,\dots,\omega_k)$ define a characteristic class $p_\omega\in \EE^{4|\omega|,2|\omega|}(\mathrm{B}\SL)$ as the product $p_{\omega_1}\dots p_{\omega_k}$, where $|\omega|=\omega_1+\dots+\omega_k$.
\end{definition}
\begin{definition}\label{def:char_numbers}
    Let $\EE\in \SH(k)$ be an $\SL$-oriented homotopy commutative ring spectrum. For a partition $\omega$ the \textit{Pontryagin characteristic number} of $\alpha\in \pi_{i,j}(\MSL)$ is the Kronecker pairing $$\la{p_{\omega},\eta_R(\alpha)\cap \thc\ra}\in \pi_{{i-4|\omega|,j-2|\omega|}}(\EE).$$ Here cap product with $\thc$ is the Thom isomorphism $\EE\wedge\MSL\xrightarrow{\simeq} \EE\wedge\Sigma^\infty_+\mathrm{B}\SL$ (see Lemma \ref{lemma:thom_iso_rev}), and $\eta_R$ is the right unit map $\MSL=\sph\wedge\MSL\to \EE\wedge\MSL$. Diagrammatically, it is given by \[\Sigma^{i,j}\sph\xrightarrow{\alpha} \MSL\xrightarrow{\eta_R} \EE\wedge\MSL\xrightarrow{\simeq} \EE\wedge\Sigma^\infty_+\mathrm{B}\SL\xrightarrow{\id\wedge\,p_{\omega}}\EE\wedge \Sigma^{4|\omega|,2|\omega|}\EE\xrightarrow{m_\EE}\Sigma^{4|\omega|,2|\omega|}\EE.\] This construction defines a homomorphism of groups $p_\omega\colon\pi_{i,j}(\MSL)\to \pi_{i-4|\omega|,j-2|\omega|}(\EE)$ for $i,j\in \Z$.
\end{definition}

Now, let $\EE$ be the hermitian $\K$-theory spectrum $\KQ$ or the Witt spectrum $\KW$, and suppose that $\omega$ is an even partition, i.e., $|\omega|=2n$ (see Appendix \ref{appendix_A} for a recollection on the hermitian $\K$-theory). Since these spectra are $(8,4)$-periodic, we have the homomorphisms $p_{\omega}\colon\pi_{i,j}(\MSL)\to \pi_{i,j}(\EE)$ for $i,j\in \Z$. Equivalently, we can first shift the corresponding characteristic class $p_{\omega}\in \EE^{0,0}(\mathrm{B}\SL)$ and then repeat the above definition for it.
\begin{lemma}
    Let $k$ be a local Dedekind domain with $1/2\in k$. Then the map $\eta_R\colon\pi_{*}(\MSL[\eta^{-1}])\to \KW_*(\MSL)$ is injective.
\end{lemma}
\begin{proof}
    Consider the following commutative diagram \[\xymatrix{\pi_{*}(\MSL[\eta^{-1}]) \ar[r] \ar[d] & \mathrm{kw}_*(\MSL) \ar[r] \ar[d] & \KW_*(\MSL) \\ \pi_{*}(\MSL[\eta^{-1}])_{(2)} \ar[r] & \mathrm{kw}_*(\MSL)_{(2)}, } \] where $\mathrm{kw}=\KW_{\geq 0}$. First note that the localization $\W(k)\hookrightarrow \W(k)_{(2)}$ is injective; see \cite[Chapter VI, Theorem 2.2 and Chapter II, Theorem 6.4(i)]{Sch}. Therefore, the left vertical arrow is injective by $\pi_*(\MSL[\eta^{-1}])\cong\W(k)[y_4,y_8,\dots]$. On the other hand, the lower horizontal map is also injective; see the proofs of \cite[Theorem 8.8]{BHop} and \cite[Proposition 5.6(2)]{BacDVR}. Thus, the left upper horizontal arrow is injective. The second upper horizontal morphism is injective by \cite[Theorem 4.1(2)]{BHop}.
\end{proof}
Since the unit $\MSL\to \MSL\wedge\KW$ factors through $\MSL[\eta^{-1}]$, the homomorphism $p_\omega$ is given by the composition \[\pi_{i,j}(\MSL)\to \pi_{i-j}(\MSL[\eta^{-1}])\to \pi_{i-j}(\KW).\] We denote the second map by $p_{\omega}[\eta^{-1}]$. Now we reformulate the previous lemma in terms of the $\KW$-characteristic numbers.
\begin{proposition}\label{prop:pontr_control_eta-period}
    Let $k$ be as above. Then the homotopy groups of $\MSL[\eta^{-1}]$ are determined by the characteristic numbers $p_\omega[\eta^{-1}]$, where $\omega$ runs through the partitions of the form $\omega=(2\omega_1,\dots,2\omega_m)$, i.e., the following homomorphisms are injective $$(p_\omega)\colon\pi_{n}(\MSL[\eta^{-1}])\to \prod_{\omega=(2\omega_1,\dots,2\omega_m)}\pi_n(\KW).$$
\end{proposition}
\begin{proof}
    Consider $\alpha\in \pi_*(\MSL[\eta^{-1}])$. Since $\KW$ is $\eta$-periodic and $\SL$-oriented there are unique elements $\alpha_{\omega'}\in \pi_*(\KW)$ such that $\eta_R(\alpha)\cap \thc=\sum_{\omega'} \alpha_{\omega'} e^{\omega'}$ (see \cite[Theorem 4.1(2)]{BHop}), where the sum is taken over the partitions of the form $\omega=(2\omega_1,\dots,2\omega_m)$ and $e^{\omega}=\prod_{i=1}^m e_{2\omega_i}$. We have \[p_\omega[\eta^{-1}](\alpha)=\sum_{\omega'} \alpha_{\omega'} \la{p_\omega,e^{\omega'}}\ra\in \pi_*(\KW).\] Now assume that $p_\omega[\eta^{-1}](\alpha)=0$ for all $\omega$ of the form $(2\omega_1,\dots,2\omega_m)$. It follows that the Kronecker pairing of $\eta_R(\alpha)\cap \thc$ with an arbitrary element of $\KW^*(\mathrm{B}\SL)\cong \KW^*(k)[\![p_2,p_4,\dots]\!]$ is trivial; see \cite[Theorem 10]{Ana15} for this isomorphism (we stress that in \textit{loc. cit.} the author uses a different convention for the Pontryagin classes). By duality, for every partition $\omega'$ there exists a power series in Pontryagin classes that is dual to the generator $e^{\omega'}$. Hence, all coefficients $\alpha_{\omega'}$ are trivial and $\eta_R(\alpha)=0$. 
\end{proof}
\begin{corollary}\label{cor:pontryagin_8n+1}
    Let $k$ be as above and let $e\neq 2$ be the exponential characteristic of the residue field of $k$. Then the homotopy groups $\pi_{8n+1,4n+1}(\MSL)[\einv]$ are determined by the $\KQ$-characteristic numbers $p_\omega[\einv]$ (in the same sense as in the previous proposition), where $\omega$ runs through the partitions of the form $\omega=(2\omega_1,\dots,2\omega_m)$.
\end{corollary}
\begin{proof}
    Consider the following commutative diagram \[\xymatrix{\pi_{8n+1,4n+1}(\MSL) \ar[r]^{p_\omega} \ar[d] & \pi_{8n+1,4n+1}(\KQ) \ar[d]^\simeq \\ \pi_{4n}(\MSL[\eta^{-1}]) \ar[r]^{p_\omega[\eta^{-1}]} & \pi_{4n}(\KW).}\] The right vertical map is an isomorphism by the Wood cofiber sequence \eqref{wood_cofib} (see Lemma \ref{appendix_lemma_witt}) and the left vertical arrow becomes an isomorphism after inverting $e$ by Theorem \ref{theorem:eta_stab}. The claim follows from the previous proposition.
\end{proof}
\section{Additive structure and pullback description}\label{section-6}
Suppose that $k$ is a local Dedekind domain, $e$ is the exponential characteristic of the closed fiber, and $n$ is an integer. Then the exact sequence \eqref{equation:cofib} induces a short exact sequence \begin{equation}\label{equation:cofib_cut} 0\to \eta\cdot\pi_{2n-1,n-1}(\MSL)\to \pi_{2n,n}(\MSL)\to \Ker(\pi_{2n,n}(d))\to 0.\end{equation} This section is devoted to the analysis of this extension. In the first subsection, we compute the left term using Pontryagin numbers (see Proposition \ref{prop:trickysubgroup_general}) and deduce immediate consequences about the torsion subgroup of $\pi_{2*,*}(\MSL)$. Then we investigate the right term (see Proposition \ref{prop:ker_of_d}) and compute the image of $\pi_{2*,*}(\MSL)[\einv]\to \pi_{2*,*}(\MGL)[\einv]$. Finally, in the last subsection we prove Theorem \ref{theorem:D}.
\subsection{The first term and the torsion subgroup}
\begin{lemma}\label{lemma:derived_exact}
     There is an exact sequence of abelian groups \[\cdots\to \Hml_{i+2,j+1}(\MW,\delta)\to\eta\cdot\pi_{i-1,j-1}(\MSL)\xrightarrow{\eta}\eta\cdot\pi_{i,j}(\MSL)\to \Hml_{i,j}(\MW,\delta)\to\cdots,\] where $\eta\cdot\pi_{i,j}(\MSL)$ denotes the image of the multiplication by the motivic Hopf element \[\mathrm{Im}(\eta\colon\pi_{i,j}(\MSL)\to\pi_{i+1,j+1}(\MSL)).\]
\end{lemma}
\begin{proof}
    This exact sequence is the derived couple of the exact couple \eqref{equation:cofib}.
\end{proof}
\begin{lemma}\label{lemma:trickysubgroup_quad_closed}
    Let $L$ be a quadratically closed field of exponential characteristic $e\neq 2$. Then the following holds \begin{equation*} \eta\cdot\pi_{2n-1,n-1}(\MSL_L)[\einv]\cong\begin{cases}  0, & \text{if}\ n\not\equiv 1\pmod*{4}, \\ (\Z/2)^{p(\frac{n-1}{4})}, & \text{if}\ n\equiv 1\pmod*{4}. \end{cases}
    \end{equation*}
\end{lemma}
\begin{proof}
    We implicitly invert $e$ throughout this proof. We have the exact sequence \[\eta\cdot\pi_{2n+2,n+1}(\MSL)\to \Hml_{2n+2,n+1}(\MW,\delta)\to \eta\cdot\pi_{2n-1,n-1}(\MSL)\xrightarrow{\eta} \eta\cdot\pi_{2n,n}(\MSL) \] from the previous lemma. By Corollary \ref{cor:eta_stab_rev} the group $\eta\cdot\pi_{2n,n}(\MSL)$ is trivial if $4$ does not divide $n$, and isomorphic to $(\Z/2)^{p(\frac{n}{4})}$ otherwise. We also know (by Theorem \ref{theorem:conner_floyd_hom} and Proposition \ref{prop:homology_top}) that $\Hml_{2n,n}(\MW,\delta)$ is trivial if $n$ is odd, and isomorphic to $(\Z/2)^{p(k)}$ if $n=4k$ or $n=4k+2$. Combining these computations with the above exact sequence, we obtain the result for $n\equiv 1,2\pmod*{4}$. 
    
    To treat the last two cases, consider the following exact sequence \begin{equation}\label{equation:cofib_hardplace} 0\to\eta\cdot\pi_{8n-1,4n-1}(\MSL)\xrightarrow{\eta} \eta\cdot\pi_{8n,4n}(\MSL)\to \Hml_{8n,4n}(\MW,\delta)\to \eta\cdot\pi_{8n-3,4n-2}(\MSL)\to 0.\end{equation} The groups in the middle have the same order. In particular, if we prove that the left group vanishes, then we obtain the isomorphism in the middle. Hence, it is enough to verify that $\eta^2\cdot\pi_{8n-1,4n-1}(\MSL)=0$. 
    
    First we claim that this group is annihilated by $\etatop$. Indeed, $\eta^2\cdot\etatop=12\cdot\nu$ in $\pi_{3,2}(\sph)$ (see \cite[Remark 5.8]{RSO19}), and $\nu=0$ in $\pi_{3,2}(\MSL)$ (see \cite[Lemma 5.3]{Ana21}), where $\nu$ is the second motivic Hopf element \cite[Definition 4.7]{DI13}. Second, we show that the homomorphism \[\etatop\colon\pi_{8n+1,4n+1}(\MSL)\to \pi_{8n+2,4n+1}(\MSL)\] is injective. For that consider the following commutative diagram \[\xymatrix{\pi_{8n+1,4n+1}(\MSL) \ar[r]^{\etatop} \ar[d]_{(p_\omega)} & \pi_{8n+2,4n+1}(\MSL) \ar[d]^{(p_\omega)} \\ \prod_\omega\pi_{8n+1,4n+1}(\KQ) \ar[r]^{\etatop}_{\simeq} & \prod_\omega\pi_{8n+2,4n+1}(\KQ), }\] where the products are taken along all partitions of the form $\omega=(2\omega_1,\dots,2\omega_m)$. The bottom arrow is an isomorphism by Corollary \ref{cor:etatop_iso} and the left map is injective by Corollary \ref{cor:pontryagin_8n+1}. Thus the top homomorphism is injective. In order to prove the desired vanishing it remains to combine these two facts: according to the first claim, the group $\eta^2\cdot\pi_{8n-1,4n-1}(\MSL)$ lies in the kernel of the multiplication by $\etatop$, and according to the second claim, this kernel is trivial.
\end{proof}
\begin{proposition}\label{prop:trickysubgroup_general}
    Let $k$ be a local Dedekind domain and let $e$ be the exponential characteristic of the residue field. Assume that $e\neq 2$. Then the following holds \begin{equation*} \eta\cdot\pi_{2n-1,n-1}(\MSL)[\einv]\cong\begin{cases} \mathrm{I}(k)[\einv]^{p(\frac{n}{4})}, & \text{if}\ n\equiv 0\pmod*{4}, \\ (\Z/2)^{p(\frac{n-1}{4})}, & \text{if}\ n\equiv 1\pmod*{4}, \\ 0, & \text{if}\ n\equiv 2,3\pmod*{4}. \end{cases} \end{equation*} Moreover, if $f\colon k\to k'$ is a homomorphism of local Dedekind domains in which $2$ is invertible, then the base change maps $\eta\cdot\pi_{2n-1,n-1}(\MSL_k)[\einv]\to \eta\cdot\pi_{2n-1,n-1}(\MSL_{k'})[\einv]$ are isomorphisms for $n\not\equiv 0\pmod*{4}$.
\end{proposition}
\begin{proof}
    We implicitly invert $e$ below. For $n\equiv 1,2\pmod*{4}$ the argument is absolutely the same as in the previous lemma. To compute the remaining groups, we choose a ring homomorphism $k\to L$, where $L$ is a quadratically closed field (e.g., the algebraic closure of the fraction field). Consider the commutative diagram \[\xymatrix{\W(k)^{p(n)}\cong\eta\cdot\pi_{8n,4n}(\MSL_k) \ar[r] \ar@{->>}[d]_-{\oplus \overline{\mathrm{rk}}} & \Hml_{8n,4n}(\MW_k,\delta) \ar[d]^\simeq \\  (\Z/2)^{p(n)}\cong\eta\cdot\pi_{8n,4n}(\MSL_L) \ar[r]^-\simeq & \Hml_{8n,4n}(\MW_L,\delta). }\] The bottom arrow is an isomorphism by the proof of the previous lemma. It follows that the top map is surjective with kernel $\mathrm{I}(k)^{p(n)}$ and the claim follows from the exact sequence \eqref{equation:cofib_hardplace}. The base change map for $\eta\cdot\pi_{2n-1,n-1}(\MSL)[\einv]$ is an isomorphism for $n\equiv 1\pmod*{4}$ since the above proof identifies the group with $\Hml_{2n+2,n+1}(\MW,\delta)$, and the latter satisfies the property by Theorem \ref{theorem:conner_floyd_hom}.
\end{proof}
As an immediate consequence we obtain information about the torsion subgroup of $\pi_{2*,*}(\MSL)$.
\begin{corollary}\label{cor:2_torsion}
    Let $k$ and $e\neq 2$ be as above. Then the geometric part of the special linear algebraic cobordism $\pi_{2*,*}(\MSL)$ contains no $l$-torsion for primes $l\neq 2,e$. Furthermore, the $2$-primary torsion subgroup of $\pi_{2n,n}(\MSL)$ is given by ${}_{2^\infty}\mathrm{I}(k)^{p(\frac{n}{4})}$ if $n\equiv 0\pmod*{4}$, by $(\Z/2)^{p(\frac{n-1}{4})}$ if $n\equiv 1\pmod*{4}$, and trivial otherwise.
    
    If in addition the fraction field of $k$ has finite virtual $2$-cohomological dimension $\mathrm{vcd}_2(\mathrm{Frac}(k))<\infty$, then the $2$-primary torsion of $\pi_{2*,*}(\MSL)$ is of bounded order.
\end{corollary}
\begin{proof}
    Consider a prime $l\neq e$. Then the $l$-primary torsion subgroup of $\pi_{2n,n}(\MSL)$ is equal to the $l$-primary torsion subgroup of $\pi_{2n,n}(\MSL)[\einv]$.  Since the group $\Ker(\pi_{2n,n}(d))[\einv]$ is torsion free (it is a subgroup of $\pi_{2n,n}(\MGL)[\einv]$), it follows that \[{}_{l^\infty}\pi_{2n,n}(\MSL)[\einv]={}_{l^\infty}\eta\cdot\pi_{2n-1,n-1}(\MSL)[\einv].\] It remains to combine the previous proposition with \cite[Lemmas 2.9 and 2.10]{BHop}.
\end{proof}
\begin{corollary}\label{cor:2tors_multiples}
    Let $k$ and $e\neq 2$ be as above. Then the $2$-torsion elements of $\pi_{8n+2,4n+1}(\MSL)[\einv]$ are multiples of $\eta\cdot\etatop=:\eta\etatop\in \pi_{2,1}(\MSL)$. 
\end{corollary}
\begin{proof}
    By Theorem \ref{theorem:eta_stab} the map $\eta\colon\pi_{8n,4n}(\MSL)[\einv]\to \pi_{8n+1,4n+1}(\MSL)[\einv]$ is surjective. Hence, it is enough to show that $\etatop\colon\pi_{8n+1,4n+1}(\MSL)[\einv]\to \pi_{8n+2,4n+1}(\MSL)[\einv]$ is surjective onto the $2$-torsion subgroup. By the proof of the previous corollary this subgroup is identified with $\eta\cdot\pi_{8n+1,4n}(\MSL)[\einv]$. Consider the commutative diagram \[\xymatrix{\W(k)[\einv]^{p(n)}\cong\eta\cdot\pi_{8n,4n}(\MSL)[\einv] \ar[r]^-{\etatop} \ar@{=}[d] & \eta\cdot\pi_{8n+1,4n}(\MSL)[\einv]\cong(\Z/2)^{p(n)} \ar@{^{(}->}[d] \\ \pi_{8n+1,4n+1}(\MSL)[\einv] \ar[r]^-{\etatop} & \pi_{8n+2,4n+1}(\MSL)[\einv]. }\] We need to show that the top map is surjective. If $k$ is a quadratically closed field then it is a monomorphism of groups of the same order by the proof of Lemma \ref{lemma:trickysubgroup_quad_closed}. The general case follows by base change.
\end{proof}
\begin{remark}
    In topology the $2$-primary torsion subgroup of $\pi_{2n}(\MSU)$ is nontrivial only if $n$ is congruent to $1$ modulo $4$. In this case all torsion elements are multiples of $\etatop^2$; see Theorem \ref{thm:torsion_top}. The motivic picture is similar if we consider elements that are simultaneously $\eta$-torsion and $2$-torsion, with the difference that one needs to look at the product $\eta\etatop$ of the different Hopf elements.
\end{remark}
\subsection{The remaining part}
Recall, that $d\colon \MW\to \Sigma^{2,1}\MSL$ is the boundary morphism in the cofiber sequence from Corollary \ref{cor:cofiber} and $\pi_{2n,n}(d)$ is the induced map $\pi_{2n,n}(\MW)\to \pi_{2n-2,n-1}(\MSL)$. We continue to analyze the exact sequence \eqref{equation:cofib_cut}. It remains to investigate the groups $\Ker(\pi_{2n,n}(d))\subset \pi_{2n,n}(\MW)$.
\begin{proposition}\label{prop:ker_of_d}
    Let $k$ and $e\neq 2$ be as above. Then the following holds \begin{equation*} \Ker(\pi_{2n,n}(d))[\einv]=\begin{cases} \Cyc_{2n,n}(\MW,\delta)[\einv], & \text{if}\ n\not\equiv 2\pmod*{4}, \\ \Bnd_{2n,n}(\MW,\delta)[\einv], & \text{if}\ n\equiv 2\pmod*{4}. \end{cases} \end{equation*} Moreover, if $f\colon k\to k'$ is a morphism of local Dedekind domains in which $2$ is invertible, then the base change map $\Ker((d_k)_*)[\einv]\to \Ker((d_{k'})_*)[\einv]$ is an isomorphism.
\end{proposition}
\begin{proof}
    We implicitly invert $e$ throughout this proof. Consider the diagram (see Definition \ref{definition:MWL} and beginning of \S \ref{section:4_2} for the morphisms $\forg$ and $\delta$) \[\xymatrix{\pi_{2n,n}(\MW) \ar[r]^-{d_*} \ar[rd]_-{\delta_*} & \pi_{2n-2,n-1}(\MSL) \ar[d]^-{\forg_*} \\ & \pi_{2n-2,n-1}(\MW), }\] which commutes by definition of $\delta_*$. Applying the kernel-cokernel lemma to this diagram we obtain the exact sequence 
    $$ 0\to \Ker(d_*)\to \Cyc_{2n,n}(\MW,\delta)\to \Ker(\forg_*)\to \mathrm{Coker}(d_*). $$
    Suppose that $n$ is congruent to $0$ or $3$ modulo $4$. Then the kernel $\Ker(\forg_*)$ vanishes by Proposition \ref{prop:trickysubgroup_general} since it coincides with $\eta\cdot\pi_{2n-3,n-2}(\MGL)$ by the exact sequence \eqref{equation:cofib}. Thus $\Ker(d_*)=\Cyc_{2n,n}(\MW,\delta)$ for such $n$ by the above exact sequence. If $n$ is congruent to $1$ modulo $4$ then the same argument with Proposition \ref{prop:trickysubgroup_general} shows that the kernel of the map $\forg_*$ is given by the direct sum of fundamental ideals. In particular, this kernel maps injectively into $\pi_{n-1}(\MSL[\eta^{-1}])$ and does not intersect the image of $d_*$, which is annihilated by $\eta$ because of the exact sequence \eqref{equation:cofib} again. Therefore, the homomorphism $\Cyc_{2n,n}(\W,\delta)\to \Ker(\forg_*)$ is zero and the same conclusion holds.
    
    Now, let us assume that $n$ is congruent to $2$ modulo $4$. In this case, the boundary map $d_*$ is surjective since the next group in the exact sequence \eqref{equation:cofib} vanishes by Corollary \ref{cor:eta_stab_rev}. Therefore, $\mathrm{Coker}(d_*)=0$ and we obtain the short exact sequence. The second homomorphism in the exact sequence fits into the diagram \[\xymatrix{ \Cyc_{2n,n}(\MW,\delta) \ar[r] \ar[rd] & \Ker(\forg_*)=\eta\cdot\pi_{2n-3,n-2}(\MSL) \\ & \Hml_{2n,n}(\MW,\delta). \ar[u]^-\simeq }\] Here the vertical map is the boundary homomorphism from Lemma \ref{lemma:derived_exact} and the diagram commutes by the construction of this map. Moreover, the boundary homomorphism is bijective by the proof of Proposition \ref{prop:trickysubgroup_general}. Hence, the kernel of the top arrow in the diagram is equal to $\Bnd_{2n,n}(\MW,\delta)$, and we conclude by the above exact sequence.

    These groups are stable under base change along $f$ by Theorem \ref{theorem:conner_floyd_hom}.
\end{proof}
As an immediate application, we compute the image of the map $\pi_{2*,*}(\MSL)[\einv]\to \pi_{2*,*}(\MGL)[\einv]$. This corollary generalizes \cite[Theorem B(2)]{LYZ}.
\begin{corollary}\label{cor:image_msl_mgl}
    Let $k$ be a local Dedekind domain and let $e\neq 2$ be the exponential characteristic of the residue field of $k$. Then the image of $\pi_{2n,n}(\MSL)[\einv]$ in $\pi_{2n,n}(\MGL)[\einv]$ is equal to $\Cyc_{2n,n}(\MW,\delta)[\einv]$ if $n\not\equiv 2\pmod*{4}$, and to $\Bnd_{2n,n}(\MW,\delta)[\einv]$ if $n\equiv 2\pmod*{4}$.
\end{corollary}
\begin{proof}
    The homomorphism $\pi_{2n,n}(\MSL)\to \pi_{2n,n}(\MGL)$ factors through $\pi_{2n,n}(\MW)$ since the morphism $\MSL\to\MGL$ factors as $\MSL\to\MW\to\MGL$ by the construction of $\MW$, see Definition \ref{definition:MWL}. Moreover, the map $\pi_{2n,n}(\MW)\to\pi_{2n,n}(\MGL)$ becomes an inclusion after inverting $e$; see the proof of Proposition \ref{prop:wall_geom_part}. Hence, we need to compute the image of $\pi_{2n,n}(\MSL)[\einv]\to \pi_{2n,n}(\MW)[\einv]$. This is done in the previous proposition since $\mathrm{Im}(\forg_*)=\Ker(d_*)$ by exactness of \eqref{equation:cofib}.
\end{proof}
\begin{remark}
    Recall that by Theorem \ref{theorem:conner_floyd_hom} these groups are known. Roughly speaking, it means that the image of $\pi_{2*,*}(\MSL)$ in $\pi_{2*,*}(\MGL)$ is the same as the image of $\pi_{2*}(\MSU)$ in $\pi_{2*}(\MU)$ at least after inverting $e$; see Remark \ref{appendix:image_msu}.
\end{remark}
\subsection{Pullback square of rings}\label{section_6-3}
First we compute the ring $\pi_{0,0}(\MSL)[\einv]$. If $k$ is a field, this can be deduced (even integrally) from Morel's computation \cite[Theorem 1.23]{Morel_book} using \cite[Example 16.35]{BH21}. However, if $k$ is a discrete valuation ring, this method is not available. 
Recall, that there is a canonical morphism of $\Einf$-ring spectra $\MSL\to\KQ$; see \cite[Remark 7.10]{HJN}.
\begin{proposition}\label{prop:pi_0,0(MSL)}
    Let $k$ and $e\neq 2$ be as above. Then the canonical morphism $\MSL\to \KQ$ induces an isomorphism $$\pi_{0,0}(\MSL)[\einv]\cong\GW(k)[\einv].$$
\end{proposition}
\begin{proof}
    We implicitly invert $e$ below. The morphism $\MSL\to\KQ$ induces the following map of exact sequences
    \[ \xymatrix{ 0 \ar[r] & \eta\cdot\pi_{-1,-1}(\MSL) \ar[r] \ar[d] & \pi_{0,0}(\MSL) \ar[r] \ar[d] & \pi_{0,0}(\MW) \ar[d] \\ 0 \ar[r] & \eta\cdot\pi_{-1,-1}(\KQ) \ar[r] & \pi_{0,0}(\KQ) \ar[r] & \pi_{0,0}(\mathrm{KGL}) ,} \] where the top (respectively bottom) exact sequence is \eqref{equation:cofib} (respectively \eqref{wood_cofib}), and the right vertical map is induced by the map of $\eta$-cofibers $\MW\cong\MSL/\eta\to \KQ/\eta\cong \mathrm{KGL}$ (see Theorem \ref{thm:fiber_seq} and the Wood cofiber sequence \eqref{wood_cofib}). The bottom right arrow can be identified with the rank homomorphism (in particular, it is surjective) and the top right arrow is surjective because the next term $\pi_{-2,-1}(\MSL)$ in the exact sequence \eqref{equation:cofib} vanishes, as follows from the exact sequence \eqref{equation:cofib_cut}, Proposition \ref{prop:trickysubgroup_general}, and Proposition \ref{prop:ker_of_d}. Hence, both rows of the above diagram are short exact sequences. We claim that the left and right vertical maps are isomorphisms.
    \begin{enumerate}
        \item We begin with the right homomorphism. The morphism $\MW\to \mathrm{KGL}$ factors through $\MGL$ and we obtain the commutative triangle
        \[\xymatrix{\pi_{0,0}(\MW) \ar[rr] \ar[rd] & & \pi_{0,0}(\MGL) \ar[ld] \\ & \pi_{0,0}(\mathrm{KGL}). & } \]
        The homomorphism $\pi_{2*,*}(\MGL)\to\pi_{2*,*}(\mathrm{KGL})$ classifies the multiplicative formal group law over $\pi_{2*,*}(\mathrm{KGL})\cong \Z[\beta^{\pm 1}]$. In particular it is an isomorphism in degree zero. To prove that the top homomorphism in the above diagram is bijective consider the exact sequence \eqref{equation:fiber} for $(i,j)=(0,0)$:
        \[ \pi_{-3,-2}(\MGL)\to \pi_{0,0}(\MW)\to \pi_{0,0}(\MGL)\to \pi_{-4,-2}(\MGL). \]
        The left and right terms are trivial by Lemma \ref{lemma:vanishing_mgl}.
        \item Define the $\eta$-Bockstein homology for $\mathrm{KGL}$ in the same way as for $\MW$ (see the beginning of \S \ref{section:4_2}) and consider the following morphism of the $\eta$-Bocktein exact sequences (see Lemma \ref{lemma:derived_exact} for the top one)
        \[\xymatrix{ 0\to \eta\cdot\pi_{-1,-1}(\MSL) \ar[r] \ar[d] & \eta\cdot\pi_{0,0}(\MSL) \ar[r] \ar[d] & \Hml_{0,0}(\MW,\delta) \ar[d] \\ 0\to \eta\cdot\pi_{-1,-1}(\KQ) \ar[r] & \eta\cdot\pi_{0,0}(\KQ) \ar[r] & \Hml_{0,0}(\mathrm{KGL},\mathrm{hyp}). }\] Here we used the vanishing of the groups $\Hml_{2,1}(\MW,\delta)$ and $\Hml_{2,1}(\mathrm{KGL},\mathrm{hyp})$, which a priori appear on the left (see Theorem \ref{theorem:conner_floyd_hom} and Proposition \ref{prop:homology_top} for $\MW$, while the vanishing for $\mathrm{KGL}$ follows from the injectivity of the corresponding $\eta$-Bockstein homomorphism). We are interested in the left vertical map. The middle vertical map is an isomorphism since these groups are identified with the $0$-th homotopy groups of their respective $\eta$-periodizations (see the proofs of Corollary \ref{cor:eta_stab_rev} and Lemma \ref{appendix_lemma_witt}) and $\pi_0(\MSL[\eta^{-1}])\to \pi_{0}(\mathrm{KW})$ is an isomorphism by \cite[Corollaries 1.3(3), 8.11]{BHop} and \cite[Proposition 5.6]{BacDVR}. In turn, the bijectivity of the right morphism can be easily deduced from the first item.
    \end{enumerate}
    The claim is proven. It follows that the morphism $\MSL\to\KQ$ induces an isomorphism on $\pi_{0,0}$. It remains to use the fact that $\KQ$ represents hermitian K-theory, see Appendix \ref{appendix_A}.
\end{proof}
The geometric diagonal $\pi_{2*,*}(\MSL)$ is a graded $\pi_{0,0}(\MSL)$-algebra. In particular, the previous proposition endows $\pi_{2*,*}(\MSL)[\einv]$ with the structure of a graded $\GW(k)$-algebra. If $k$ is a field, then this $\GW(k)$-algebra structure coincides with the one that comes from the motivic sphere spectrum because of the isomorphism $\pi_{0,0}(\sph)\cong\pi_{0,0}(\KQ)$. 
Now we turn to the geometric diagonal. We put $(\mathrm{I}_\MSL(k))_n:=\eta\cdot\pi_{2n-1,n-1}(\MSL)$ for $n\equiv 0\pmod*{4}$, and $(\mathrm{I}_\MSL(k))_n:=0$ otherwise. This defines a graded subgroup $\mathrm{I}_\MSL(k)$ of $\pi_{2*,*}(\MSL)$.
\begin{lemma}
    Let $k$ and $e\neq 2$ be as above. Then $\mathrm{I}_\MSL(k)[\einv]$ is a graded ideal of $\pi_{2*,*}(\MSL)[\einv]$.
\end{lemma}
\begin{proof}
    We implicitly invert $e$ throughout this proof. Let us choose a ring homomorphism $k\to L$, where $L$ is a quadratically closed field (e.g., the algebraic closure of the fraction field). Consider the base change of the exact sequence \eqref{equation:cofib_cut} along $f$ \[\xymatrix{0\to \eta\cdot\pi_{2n-1,n-1}(\MSL_k) \ar[r] \ar@{->>}[d]_-{f^*} & \pi_{2n,n}(\MSL_k) \ar[r] \ar[d]_-{f^*} & \Ker(\pi_{2n,n}(d_k))\to 0 \ar[d]^-\simeq_-{f^*} \\ 0\to \eta\cdot\pi_{2n-1,n-1}(\MSL_L) \ar[r] & \pi_{2n,n}(\MSL_L) \ar[r] & \Ker(\pi_{2n,n}(d_{L}))\to 0. }\] The right vertical map is an isomorphism (see Proposition \ref{prop:ker_of_d}) and the left vertical morphism is surjective with kernel $\mathrm{I}_\MSL(k)$ by Proposition \ref{prop:trickysubgroup_general}. Therefore, the subgroup $\mathrm{I}_\MSL(k)$ is the kernel of the homomorphism $\pi_{2*,*}(\MSL_k)\to \pi_{2*,*}(\MSL_L)$ of the graded $\GW(k)$-algebras.
\end{proof}
\begin{proposition}\label{thm:complete_answer}
    Let $k$ and $e\neq 2$ be as above. Suppose that $f\colon k\to k'$ is a homomorphism of local Dedekind domains. Then the base change along $f$ induces an isomorphism of graded $\GW(k)$-algebras \[\bigslant{\pi_{2*,*}(\MSL_k)}{\mathrm{I}_\MSL(k)}[\einv]\xrightarrow{\simeq} \bigslant{\pi_{2*,*}(\MSL_{k'})}{\mathrm{I}_\MSL(k')}[\einv].\]
\end{proposition}
\begin{proof}
    The pullback of the exact sequence \eqref{equation:cofib_cut} gives the result similarly to the proof of the previous lemma.
\end{proof}
\begin{remark}
    Note that the isomorphism stated in the previous proposition can also be viewed as an isomorphism of $\GW(k')$-algebras. Indeed, the action of $\GW(k')$ on the right hand side is induced by the rank homomorphism since $\mathrm{I}_{\MSL}(k')[\einv]$ contains the fundamental ideal in degree zero, see Proposition \ref{prop:trickysubgroup_general}. Hence, this isomorphism becomes an isomorphism of $\GW(k')$-algebras if we endow the left hand side with the obvious action of $\GW(k')$.
\end{remark}
\begin{theorem}\label{theorem:complete_modulo_imslk}
    Let $k$ and $e\neq 2$ be as above. Then there is an isomorphism of graded $\GW(k)$-algebras \[\bigslant{\pi_{2*,*}(\MSL)}{\mathrm{I}_\MSL(k)}[\einv]\cong \pi_{2*}(\MSU)[\einv],\] where the $\GW(k)$-algebra structure on the right hand side is induced by the rank $\mathrm{rk}\colon\GW(k)\to \Z$. If $k=\C$ then the complex Betti realization functor induces such an isomorphism. 
\end{theorem}
\begin{proof}
    First, we treat $k=\C$. Applying the complex Betti realization functor to the exact sequence \eqref{equation:cofib_cut} we obtain the following commutative diagram with exact rows (the bottom one is exact by Proposition \ref{appendix:prop_cofib_1}) \[\xymatrix{0\to \eta\cdot\pi_{2n-1,n-1}(\MSL_k) \ar[r] \ar[d]_-{\BettiC} & \pi_{2n,n}(\MSL_k) \ar[r] \ar[d]_-{\BettiC} & \Ker(\pi_{2n,n}(d_k))\to 0 \ar[d]^-\simeq_-{\BettiC} \\ 0\to \etatop\cdot\pi_{2n-1}(\MSU) \ar[r] & \pi_{2n}(\MSU) \ar[r] & \Ker(\pi_{2n}(d))\to 0.}\] The right vertical map is an isomorphism by Proposition \ref{prop:ker_of_d}. Thus we need to prove that the arrow $\BettiC\colon\eta\cdot\pi_{2n-1,n-1}(\MSL_\C)\to \etatop\cdot\pi_{2n-1}(\MSU)$ is an isomorphism. If $n\not\equiv 1\pmod*{4}$ then both groups are trivial by Lemma \ref{lemma:trickysubgroup_quad_closed} and Theorem \ref{thm:torsion_top}. In the remaining case, the left hand side is isomorphic to $\Hml_{2n+2,n+1}(\MW_k,\delta_k)$ via the boundary map from the exact sequence from Lemma \ref{lemma:derived_exact}, and the same holds for the right hand side with the topological counterpart $\Hml_{2n+2}(\W,\delta)$. Consequently, the desired isomorphism is induced by $\BettiC\colon\Hml_{2n+2,n+1}(\MW_k,\delta_k)\xrightarrow{\simeq}\Hml_{2n+2}(\W,\delta)$.

    The remaining cases follow from the various base changes using the previous proposition (similarly to the proof of Lemma \ref{lemma:charclasses_and_operations}) and the previous remark.
\end{proof}
    Denote by $R$ the ring $\pi_{2*,*}(\MSL)[\einv]$, and by $I$ and $J$ the ideals $\mathrm{I}_\MSL(k)[\einv]$ and ${}_{\eta}\pi_{2*,*}(\MSL)[\einv]$ respectively. The quotient $R/(I+J)$ is isomorphic to $\bigslant{R/J}{(I+J)/J}$, which is given by \[\bigslant{\W(k)[\einv][y_4,y_8,\dots]}{\mathrm{I}(k)[\einv][y_4,y_8,\dots]}\cong \Z/2[y_4,y_8,\dots].\] Here we use the definition of $\mathrm{I}_\MSL(k)$ and Proposition \ref{prop:trickysubgroup_general} to identify $(I+J)/J$ with $\mathrm{I}(k)[\einv][y_4,y_8,\dots]$. Furthermore, the projection $R/J\to \bigslant{R/J}{(I+J)/J}$ is induced by the rank homomorphism.
\begin{theorem}\label{theorem:complete_as_pullback}
    Suppose that $k$ is a local Dedekind domain and $e\neq 2$ is the exponential characteristic of the residue field of $k$. Then the following diagram is a pullback square of graded $\GW(k)$-algebras \[\xymatrix{\pi_{2*,*}(\MSL)[\einv] \ar[r] \ar[d] & \bigslant{\pi_{2*,*}(\MSL)}{\mathrm{I}_\MSL(k)}[\einv] \ar[d] \\ \W(k)[\einv][y_4,y_8,\dots] \ar[r]^-{\overline{\mathrm{rk}}} & \Z/2[y_4,y_8,\dots], }\] where the left map is the quotient by the annihilator of $\eta$, and the right homomorphism is the quotient by $(I+J)/I$.
\end{theorem}
\begin{proof}
    The argument is similar to that in the proof of Proposition \ref{prop:pi_0,0(MSL)}. Indeed, in the above notation, we have $R/(I\cap J)\cong R/I\times_{R/(I+J)} R/J$, where the right hand side is the desired pullback. In turn, $I\cap J=0$ since $I$ maps injectively into $\pi_*(\MSL[\eta^{-1}])[\einv]$, while $J$ is annihilated by $\eta$.
\end{proof}
\begin{corollary}\label{corollary:local_away_from2}
    Let $k$ and $e\neq 2$ be as above. Then there is an isomorphism of graded $\GW(k)$-algebras \[\pi_{2*,*}(\MSL)[\nicefrac{1}{2e}]\cong \Z[\nicefrac{1}{2e}][x_2,x_3,\dots]\times \W(k)[\nicefrac{1}{2e}][y_4,y_8,\dots],\] where $\mathrm{deg}(x_i)=(2i,i)$ and $\mathrm{deg}(y_j)=(2j,j)$.
\end{corollary}
\begin{proof}
    Follows from the previous theorem, Theorem \ref{theorem:complete_modulo_imslk} and Theorem \ref{appendix:thm_msu_1/2}.
\end{proof}
\begin{remark}\label{rmk:away_2_LYZ}
    It follows from the definition of $\mathrm{I}_\MSL(k)$, exact sequence \eqref{equation:cofib_cut}, Proposition \ref{prop:trickysubgroup_general} and Proposition \ref{prop:ker_of_d}, that the plus part $\pi_{2*,*}(\MSL)/\mathrm{I}_\MSL(k)[\nicefrac{1}{2e}]$ is isomorphic to $\Cyc_{2*,*}(\MW,\delta)[\nicefrac{1}{2e}]$. This graded group is a subring of $\pi_{2*,*}(\MGL)[\nicefrac{1}{2e}]$ by Theorem \ref{theorem:conner_floyd_hom}.
\end{remark}
\section{Characteristic numbers revisited}\label{section-7}
In this section, we prove the motivic version (see Theorem \ref{theorem:motivic_abp}) of the Anderson--Brown--Peterson theorem \cite{ABP66}; see also a brief overview in Appendix \ref{appendix_B}. Then we compute the characteristic numbers of cobordism classes that are represented by smooth projective Calabi--Yau varieties, and show that such classes generate the ring $\pi_{2*,*}(\MSL)/\mathrm{I}_\MSL(k)[\nicefrac{1}{2e}]$, see Corollary \ref{cor:calabi-yau_gen}.

\subsection{Motivic Anderson--Brown--Peterson theorem} In this subsection we use the symbol $G$ either for $\GL$ or $\SL$. If $\omega=(\omega_1,\dots,\omega_k)$ is a partition, we denote by $c_\omega\in \HZ^{2|\omega|,|\omega|}(\mathrm{B}G)$ the product of the Chern classes $c_\omega=c_{\omega_1}\dots c_{\omega_k}$.
\begin{definition}[cf. Definition \ref{def:char_numbers}]\label{definition:char_numbers_2}
     For a partition $\omega$ the \textit{Chern characteristic number} of $\alpha\in \pi_{i,j}(\mathrm{M}G)$ is the Kronecker pairing $$\la{c_\omega,\eta_R(\alpha)\cap \thc\ra}\in \pi_{i-2|\omega|,j-|\omega|}(\HZ).$$ Here cap product with $\thc$ is the Thom isomorphism $\HZ\wedge\mathrm{M}G\xrightarrow{\simeq} \HZ\wedge\Sigma^\infty_+ \mathrm{B}G$ (see Lemma \ref{lemma:thom_iso_rev}), and $\eta_R$ is the right unit map $\mathrm{M}G=\sph\wedge\mathrm{M}G\to\HZ\wedge\mathrm{M}G$. Diagrammatically, it is given by \[ \Sigma^{i,j}\sph\xrightarrow{\alpha} \mathrm{M}G\xrightarrow{\eta_R} \HZ\wedge\mathrm{M}G\xrightarrow{\simeq} \HZ\wedge\Sigma^\infty_+\mathrm{B}G\xrightarrow{\id\wedge\,c_{\omega}}\HZ\wedge \Sigma^{2|\omega|,|\omega|}\HZ\xrightarrow{m_\HZ}\Sigma^{2|\omega|,|\omega|}\HZ.\] This construction defines a homomorphism of groups $c_\omega\colon\pi_{i,j}(\mathrm{M}G)\to \pi_{i-2|\omega|,j-|\omega|}(\HZ)$ for $i,j\in \Z$.
\end{definition}
For $(i,j)=(2n,n)$ the only nontrivial homomorphisms appear if $|\omega|=n$ \[c_\omega\colon\pi_{2n,n}(\mathrm{M}G)\to \pi_{0,0}(\HZ)\cong \Z.\] Note that this situation differs from the Pontryagin characteristic numbers, where $p_\omega$ might be non-zero in other cases. Given a partition $\omega$ with $|\omega|=n$, we have the following diagram \[\xymatrix{\pi_{2n,n}(\MSL) \ar[rr] \ar[dr]_-{c_\omega} & & \pi_{2n,n}(\MGL) \ar[dl]^-{c_\omega} \\ & \pi_{0,0}(\HZ)\cong \Z. & }\]
\begin{lemma}
    Let $k$ be a local Dedekind domain and let $e$ be the exponential characteristic of the residue field of $k$. Then the $\Proj^1$-diagonal of $\MGL[\einv]$ is determined by the Chern characteristic numbers $c_\omega[\einv]$, i.e., the following homomorphisms are injective
    $$ (c_\omega[\einv])\colon\pi_{2n,n}(\MGL)[\nicefrac{1}{e}]\to \prod_{|\omega|=n} \Z[\einv].$$
\end{lemma}
\begin{proof}
    The same proof as in Proposition \ref{prop:pontr_control_eta-period} shows that if all Chern characteristic numbers of an element $\alpha\in \pi_{2*,*}(\MGL)[\einv]$ are trivial, then the image of $\alpha$ along the right unit map is zero: $\eta_R(\alpha)=0$. The result follows from the injectivity of $\eta_R\colon\pi_{2*,*}(\MGL)[\einv]\to \HZ_{2*,*}(\MGL)[\einv]$ (to see this combine \cite[Diagram 6.2]{Hoy15}, Proposition \ref{prop:hopf_algebroids} and the isomorphism $\Z[b_1,b_2,\dots]\cong\HZ_{2*,*}(\MGL)$).
\end{proof}
\begin{remark}
    There is an approach to construct algebraic cobordism using characteristic numbers, pioneered by Merkurjev \cite{Mer02}. It is parallel to the Levine--Morel theory \cite{LevMor}, but does not involve resolution of singularities and works in arbitrary characteristic.
\end{remark}
\begin{theorem}\label{theorem:motivic_abp}
    Suppose that $k$ is a local Dedekind domain and $e$ is the exponential characteristic of the residue field. Then the geometric part of $\MSL[\einv]$ is determined by the Chern characteristic numbers and the $\KQ$-Pontryagin characteristic numbers.
\end{theorem}
\begin{proof}
    We implicitly invert $e$ throughout this proof. Consider $\alpha\in \pi_{2n,n}(\MSL)$ with trivial Chern numbers and Pontryagin numbers. From the previous lemma and the exact sequence \eqref{equation:cofib_cut}, we see that the image of $\alpha$ in $\Ker(d_*)\subset \pi_{2n,n}(\MGL)$ is trivial. Hence, $\alpha$ lies in $\eta\cdot\pi_{2n-1,n-1}(\MSL)$. By Proposition \ref{prop:trickysubgroup_general}, this group is non-zero only if $n\equiv 0\pmod*{4}$ or $n\equiv 1\pmod*{4}$. 
    
    In the first case, the multiplication by the Hopf element is injective (see exact sequence \eqref{equation:cofib_hardplace}) \[\eta\colon\eta\cdot\pi_{2n-1,n-1}(\MSL)\hookrightarrow \eta\cdot\pi_{2n,n}(\MSL)=\pi_{2n+1,n+1}(\MSL),\] and the target is controlled by the Pontryagin characteristic numbers by Corollary \ref{cor:pontryagin_8n+1}.
    
    It remains to show that $\eta\cdot\pi_{8m+1,4m}(\MSL)$ is determined by $p_\omega$. First, assume that $k$ is a quadratically closed field. Then we have the commutative diagram \[\xymatrix{\pi_{8m+1,4m+1}(\MSL)=\eta\cdot\pi_{8m,4m}(\MSL) \ar[r]^-{\etatop}_-{\simeq} \ar[d]_-{p_\omega} & \eta\cdot\pi_{8m+1,4m}(\MSL) \ar[d]_-{p_\omega} \\ \pi_{8m+1,4m+1}(\KQ)=\eta\cdot\pi_{8m,4m}(\KQ) \ar[r]^-{\etatop}_-{\simeq} & \eta\cdot\pi_{8m+1,4m}(\KQ), }\] where the top and the bottom arrows are isomorphisms by the proof of Corollary \ref{cor:2tors_multiples} and by Corollary \ref{cor:etatop_iso}. Thus this case follows from Corollary \ref{cor:pontryagin_8n+1} again. Now, let $k$ be an arbitrary local Dedekind domain and choose a homomorphism $f\colon k\to L$ with a quadratically closed field $L$. Then the base change along $f$ induces an isomorphism $\eta\cdot\pi_{8m+1,4m}(\MSL_k)\xrightarrow{\simeq} \eta\cdot\pi_{8m+1,4m}(\MSL_L)$ (see Proposition \ref{prop:trickysubgroup_general}). Since the Pontryagin numbers are stable under base change by construction, we get $f^*(p_\omega(\alpha))=p_\omega(\alpha_L)$, and the element $\alpha_L\in \pi_{8m+2,4m+1}(\MSL_L)$ is trivial by the previous case. This concludes the proof.
\end{proof}
\subsection{Classes of Calabi--Yau varieties}
Let $X$ be a smooth scheme over $k$ that admits a trivialization of the determinant of its tangent bundle $\theta_X\colon \det(T_{X/k})\xrightarrow{\simeq}\struct_X$. Then one can construct the class $[X,\theta_X]$ in the group $\pi_{2n,n}(\MSL)$, see \cite[\S 3]{LYZ}. This construction can be viewed as an  algebraic version of the classical Pontryagin--Thom construction. We also remark that varieties that admit such a trivialization (i.e., $K_X=0$) are called (generalized) Calabi--Yau varieties.
\begin{proposition}\label{prop:char_numbers}
    Let $k$ be a field. Suppose that $\pi_X\colon X\to \Spec(k)$ is a smooth projective Calabi--Yau variety of dimension $n$, and $\theta_X\colon\det(T_{X/k})\xrightarrow{\simeq} \struct_X$ is an isomorphism of line bundles. Then the characteristic numbers of the class $[X,\theta_X]\in \pi_{2n,n}(\MSL)$ can be computed by the following formulas \begin{gather*} c_\omega[X,\theta_X]=\mathrm{deg}_\HZ(c_\omega(-T_{X/k})), \\ p_\omega[X,\theta_X]=\mathrm{deg}_\KQ(p_\omega(-T_{X/k})), \end{gather*} where $\mathrm{deg}_\EE$ denotes the pushforward in the $\EE$-cohomology associated with the trivialization $\theta_X$.
\end{proposition}
\begin{proof}
    For the Chern numbers this is proved in \cite[Proposition 3.3(2)]{LYZ}, and the proof works for the Pontryagin numbers verbatim. Note that it uses $\A^1$-representability of $\SL$-torsors over smooth affine varieties; see \cite[Theorem 1.3]{AHW2}.
\end{proof}
\begin{corollary}\label{cor:calabi-yau_gen}
    Let $k$ be an infinite field of exponential characteristic $e\neq 2$. Then the ring \[\bigslant{\pi_{2*,*}(\MSL)}{\mathrm{I}_\MSL(k)}[\nicefrac{1}{2e}]\cong\Z[\nicefrac{1}{2e}][x_2,x_3,\dots]\] is generated by the classes of smooth projective Calabi--Yau varieties. If in addition $k$ is not formally real, then the ring $\pi_{2*,*}(\MSL)[\nicefrac{1}{2e}]$ is itself generated by such classes. 
\end{corollary}
\begin{proof}
    From the proof of Theorem \ref{theorem:motivic_abp} it follows, that the desired quotient ring is determined by the Chern numbers. Combining this with \cite[Theorem B(3)]{LYZ}, we obtain the values of the $\HZ$-characteristic numbers that an element of $\pi_{2*,*}(\MSL)/\mathrm{I}_\MSL(k)[\nicefrac{1}{2e}]$ must have in order to represent a polynomial generator. The construction of linear combinations of smooth projective Calabi--Yau varieties with predicted Chern numbers is known from topology; see \cite[Lemma 6.14]{LYZ}. Notice that here we use Bertini's theorem to construct smooth hypersurfaces of given multi-degree, so $k$ must be infinite.

    If $k$ is not formally real, then its Witt ring is $2$-torsion; see \cite[Chapter II, Theorem 6.4(i)]{Sch}. Therefore, the ideal $\mathrm{I}_\MSL(k)[\nicefrac{1}{2e}]$ is trivial, whence the claim.
\end{proof}
\begin{remark}
    Let $X$ be a Calabi--Yau variety over a field $k$ and let $\theta_X$ be a trivialization of the determinant of its tangent bundle. Then the image of $[X,\theta_X]$ in the plus part does not depend on the choice of $\theta_X$. This happens because the Chern numbers are independent of $\theta_X$. It is reasonable to expect that the minus part is responsible for the choice of orientation via ``quadratic data''.
\end{remark}
We finish this section with a description of the classes of smooth projective Calabi--Yau varieties of dimension $\leq 2$. Note that in contrast to the previous corollary, here we work without inverting $2$. Recall the canonical morphism of $\Einf$-ring spectra $\MSL\to \KQ$ from the beginning of \S \ref{section_6-3}.
\begin{proposition}\label{prop:dim_zero_and_one}
    Suppose that $k$ is a field of exponential characteristic $e\neq 2$. 
    The morphism $\MSL\to\KQ$ induces isomorphisms (see Lemma \ref{lemma:kq_additive})
    \begin{equation*}
        \pi_{2n,n}(\MSL)[\einv]\cong \begin{cases} \GW(k)[\einv], & \text{if}\ n=0, \\ \Z/2, & \text{if}\ n=1, \\ 2\Z[\einv], & \text{if}\ n=2. \end{cases}
    \end{equation*}
    The classes of Calabi--Yau varieties of dimension $\leq 2$ map via these isomorphisms as follows:
    \begin{enumerate}
        \item Let $L_i/k$ be finite separable field extensions and view each $a_i\in L_i^\times$ as an isomorphism of line bundles $\det(T_{L_i/k})\cong \struct_{L_i}$. Then the class $[\coprod_{i=1}^m \Spec(L_i),\sum_{i=1}^m a_i]$ maps via the first isomorphism above to \begin{gather*} \sum_{i=1}^m \mathrm{Tr}_{L_i/k}(\langle a_i\rangle), \end{gather*} where $\mathrm{Tr}_{L_i/k}(-)$ is the Scharlau transfer on the Grothendieck--Witt groups $\GW(L_i)\to\GW(k)$.
        \item Let $C$ be a smooth projective curve with trivial canonical divisor $K_C=0$. Then the class of $C$ (for any choice of orientation $\theta_C$) maps via the second isomorphism above to   \begin{gather*} h^0(C,\struct_C)\ \pmod*{2}. \end{gather*} In particular, if $C$ is an elliptic curve (e.g., a smooth cubic in $\Proj^2$), then the class $[C,\theta_C]$ represents the generator $\eta\etatop$ of $\pi_{2,1}(\MSL)[\nicefrac{1}{e}]$ (see Corollary \ref{cor:2tors_multiples}).
        \item Let $S$ be a smooth projective surface with trivial canonical divisor $K_S=0$. Then the class of $S$ (for any choice of orientation $\theta_S$) maps via the third isomorphism above to
        \begin{gather*}
            2h^0(S,\struct_S)-h^1(S,\struct_S).
        \end{gather*}
        In particular, if $S$ is a K3 surface (e.g., a smooth quartic in $\Proj^3$), then the class $[S,\theta_S]$ represents a generator of $\pi_{4,2}(\MSL)[\einv]$.
    \end{enumerate}
\end{proposition}
\begin{proof}
    We implicitly invert $e$ throughout this proof. Denote by $\varphi$ the canonical morphism $\MSL\to\KQ$.  
    The map $\pi_{0,0}(\varphi)$ is bijective by Proposition \ref{prop:pi_0,0(MSL)}, while $\pi_{2,1}(\varphi)$ is an isomorphism because both groups have order $2$ and are generated by $\eta\etatop$ (see Corollary \ref{cor:2tors_multiples} and Lemma \ref{lemma:kq_generators}). In turn, for $n=2$ the homomorphism $\pi_{4,2}(\varphi)$ is a map between infinite cyclic groups (see Proposition \ref{theorem:complete_modulo_imslk} and Lemma \ref{lemma:kq_additive}), and it is enough to check that it is surjective. Instead of proving this directly, we will prove the formula stated in (3) and then conclude, because the image of a K3 surface will be a generator of $\pi_{4,2}(\KQ)$.
    
    Now we claim that the induced map $\pi_{i,j}(\varphi)\colon\pi_{i,j}(\MSL)\to\pi_{i,j}(\KQ)$ coincides with the zeroth Pontryagin number $p_0$ for any $i,j\in\Z$. Given an element $\alpha\in\pi_{i,j}(\MSL)$ we have
    $$p_0(\alpha)=\langle p_0,\eta_R(\alpha)\cap \thc\rangle = \langle \varphi, \eta_R(\alpha)\rangle=\varphi\circ\alpha,$$ where the first equality is the definition of $p_0$ (see Definition \ref{def:char_numbers}), the second one follows from $p_0=1$ and $\thc=\varphi$, and the third one can be easily checked using that $\varphi$ is a ring map and $\eta_R$ is the right unit map. 
    The claim is proven and it remains to compute $p_0[X,\theta]$ for a smooth projective Calabi--Yau variety $(X,\theta_X)$ of dimension $\leq 2$. According to Proposition \ref{prop:char_numbers}, it is given by the pushforward of $p_0(-T_{X/k})=1$ along the structure morphism. 
    \begin{enumerate}
        \item For the pushforward of $1$ in $\KQ$-cohomology along $\pi_{L_i}\colon\Spec(L_i)\to\Spec(k)$ (with orientation of $\Spec(L_i)$ given by $a_i$) see e.g., \cite[\S 8]{BW_Euler}.
        \item 
        The element $1$ in the hermitian K-theory of $C$ is represented by the symmetric object $(\struct_C,m)$ for the canonical pairing $m\colon\struct_C\otimes\struct_C\to\struct_C$.
        The pushforward of $(\struct_C,m)$ is given by the complex (see \cite[\S8D]{Levine-Raksit} for an overview of the pushforwards in the hermitian K-theory) $$H^0(C,\struct_C)\xrightarrow{0} H^1(C,\struct_C)$$ with the form
        \[\xymatrix{H^0(C,\struct_C) \ar[r]^0 \ar[d]^-\simeq & H^1(C,\struct_C) \ar[d]^-\simeq \\ H^1(C,\struct_C)^\vee \ar[r]^0 & H^0(C,\struct_C)^\vee,} \]
        where the vertical isomorphisms are induced by the Serre duality. It is easy to see that this pair is the hyperbolic form on $H^0(C,\struct_C)$. We are done in this case since the hyperbolic map $\pi_{2,1}(\mathrm{KGL})\to \pi_{2,1}(\KQ)$ is identified with the projection $\Z\twoheadrightarrow \Z/2$ (see Appendix \ref{appendix_A}).
        \item Analogously to the previous case, the pushforward is given by the complex $$H^0(S,\struct_S)\xrightarrow{0} H^1(S,\struct_S)\xrightarrow{0} H^2(S,\struct_S)$$ with a skew-symmetric form defined via a diagram similar to the one above. The Serre duality implies $$h^0(S,\struct_S)=h^2(S,\struct_S)$$ and we obtain that the rank of $(\pi_S)_*(1)$ is given by the desired formula. This case follows since the isomorphism $\pi_{4,2}(\KQ)\cong 2\Z$ is induced by the rank homomorphism. \qedhere
    \end{enumerate}
\end{proof}
\begin{remark}
        Let $k$ be a field of exponential characteristic $e\neq2$. Below we implicitly invert $e$. According to our computation and the main result of \cite{RSO19} (see also \cite{RSO24}) there is a short exact sequence \[0\to \K^{\mathrm{M}}_1(k)/24\to \pi_{2,1}(\sph)\to \pi_{2,1}(\MSL)\cong\Z/2\to 0.\] If $E$ is an elliptic curve over $k$ and $\theta_E$ is a trivialization of its tangent bundle, we can construct the class $[E,\theta_E]\in\pi_{2,1}(\sph)$ similarly to the cases of $\MSL$ and $\MGL$. Moreover, the unit map $\sph\to\MSL$ sends such a class to $[E,\theta_E]\in\pi_{2,1}(\MSL)$. Then the second claim of the previous proposition shows that $[E,\theta_E]\in\pi_{2,1}(\sph)$ is nontrivial and at least generates the right term of the above extension. This statement corresponds to the classical fact that the generator $\etatop^2$ of $\pi_2(\sph_{\mathrm{top}})\cong\Z/2$ can be represented by the torus $S^1\times S^1$ with standard framing. However, the difference between the topological picture and the motivic situation is that the surjection $\pi_2(\sph_{\mathrm{top}})\to\pi_2(\MSU)$ has a trivial kernel, unlike $\pi_{2,1}(\sph)\to \pi_{2,1}(\MSL)$.
\end{remark}
\appendix
\section{\texorpdfstring{$\Proj^1$}{P1}-diagonal of the hermitian K-theory}\label{appendix_A}
In this appendix, we recall some basic facts about the hermitian $\K$-theory spectrum and compute its geometric diagonal, see Theorem \ref{theorem:geom_diag_KQ}. We also consider the case of the very effective cover of the hermitian $\K$-theory.

Let $S$ be a regular scheme such that $2$ is invertible in $S$. Recall that there exist motivic spaces $\GW^{[n]}\in \mathbf{H}(S)$ which represent hermitian $\K$-theory; see \cite{Sch17}, \cite{ST15} and \cite{PW18}. They come with canonical equivalences $\Omega^{2,1} \GW^{[n]}\cong \GW^{[n-1]}$ and Bott periodicity isomorphisms $\GW^{[n+4]}\cong\GW^{[n]}$. The periodicity equivalences can be defined via multiplication by the Bott element $\beta\in \GW^{[-4]}_0(S)$; see \cite[Definition 5.3]{PW18}. The above data defines the \textit{hermitian $\K$-theory spectrum} $\KQ_S\in \SHk$. When $S$ is clear from the context we denote it simply by $\KQ$. This motivic spectrum admits a canonical $\Einf$-ring structure; see \cite[Lemma 7.4]{HJN}. By construction, $\Omega^\infty \Sigma^{2n,n}\KQ\cong \GW^{[n]},$ and it follows from the Bott periodicity that $\KQ$ has $(8,4)$-periodic homotopy groups. We also put $\KW:=\KQ[\eta^{-1}]$. This spectrum represents Balmer--Witt theory \cite{Balmer}; see \cite[Proposition 7.2]{Sch17} and \cite[Theorem 6.5]{Ana16b}.

Since $\KQ$ represents hermitian K-theory, there are canonical isomorphisms $\pi_{0,0}(\KQ)\cong\GW^{[0]}_0(S)$ and $\pi_{-4,-2}(\KQ)\cong \GW^{[2]}_0(S)$; see \cite[Corollary 7.3]{PW18} for a precise statement. In turn, $\GW^{[0]}_0(S)$ and $\GW^{[2]}_0(S)$ coincide with the Grothendieck--Witt groups of symmetric and skew-symmetric forms respectively; see \cite[Remark 3.14]{Sch17} and \cite[Theorem 6.1]{Walter}. In particular, if $S$ is the spectrum of a regular local ring $R$ with $1/2\in R$, we have \[\pi_{8n,4n}(\KQ)\cong\GW^{[0]}_0(R)=\GW(R)\ \ \text{and}\ \ \pi_{8n-4,4n-2}(\KQ)\cong\GW^{[2]}_0(R)\cong2\Z,\] where the first identifications are given by the representability and Bott periodicity and the last isomorphism is induced by the rank homomorphism. To move further, consider the Wood cofiber sequence (see \cite[Theorem 3.4]{RO16}) \begin{equation}\label{wood_cofib}\Sigma^{1,1}\KQ\xrightarrow{\eta}\KQ\xrightarrow{\mathrm{forg}}\mathrm{KGL}.\end{equation} Here $\mathrm{KGL}\in\SH(R)$ is the algebraic K-theory spectrum and $\mathrm{forg}$ is the forgetful morphism. The boundary map in this cofiber sequence factors as \[\mathrm{KGL}\xrightarrow{\simeq}\Sigma^{2,1}\mathrm{KGL}\xrightarrow{\Sigma^{2,1}\mathrm{hyp}}\Sigma^{2,1}\KQ,\] where $\mathrm{hyp}$ is the hyperbolic map. 
\begin{lemma}[see also {\cite[Proposition 6.3]{Sch17}}]\label{appendix_lemma_witt}
    Suppose that $R$ is a regular local ring with $1/2\in R$. Then we have natural in $R$ isomorphisms \[\pi_{2n+1,n+1}(\KQ)\cong\pi_n(\KW)\cong\W^{[-n]}(R)\cong \begin{cases} \W(R), & \text{if}\ n\equiv 0\pmod*{4}, \\ 0, & \text{otherwise}, \end{cases}\] where $\W^{[-n]}(R)$ denotes the $(-n)$-th Balmer--Witt group \cite{Balmer}. Furthermore, the multiplication by the motivic Hopf element \[\eta:\pi_{8n,4n}(\KQ)\to \pi_{8n+1,4n+1}(\KQ)\] coincides with the canonical projection $\GW(R)\twoheadrightarrow\W(R)$ under the above identifications.
\end{lemma}
\begin{proof}
    The iterated multiplication by the motivic Hopf element $\eta$ yields the sequence of homomorphisms
    \[\pi_{2n+1,n+1}(\KQ)\to\pi_{2n+2,n+2}(\KQ)\to\pi_{2n+3,n+3}(\KQ)\to\cdots\] and the colimit of this sequence coincides with the group $\pi_n(\KW)$ by definition of $\KW$ (recall our convention about $\eta$-periodic spectra, see Definition \ref{definition_eta_periodic}). By the Wood cofiber sequence and the vanishing of the negative K-groups, all maps above are isomorphisms and we obtain the first desired isomorphism. The second one follows from the representability of the Balmer--Witt theory (see references above). The computation of the Balmer--Witt groups in the local case is well-known (see e.g., \cite[Theorem 1.5.22]{Balmer}). For the last statement, consider the following exact sequence, which is induced by \eqref{wood_cofib}: \[\pi_{2,1}(\mathrm{KGL})\cong\pi_{0,0}(\mathrm{KGL})\xrightarrow{\mathrm{hyp}_*}\pi_{0,0}(\KQ)\xrightarrow{\eta}\pi_{1,1}(\KQ)\to 0.\] The hyperbolic map sends the trivial rank $1$ vector bundle to the standard hyperbolic plane by construction. The result follows from the Bott periodicity.
\end{proof}
Combining the previous lemma with the Wood cofiber sequence we obtain an exact sequence \[ \pi_{0,0}(\KQ)\xrightarrow{\mathrm{forg}_*}\pi_{0,0}(\mathrm{KGL})\to\pi_{-2,-1}(\KQ)\to 0.\] The forgetful map corresponds to the rank homomorphism under $\pi_{0,0}(\KQ)\cong \GW(R)$ and $\pi_{0,0}(\mathrm{KGL})\cong \mathrm{K}_0(R)\cong\Z$. Therefore, the group $\pi_{-2,-1}(\KQ)$ is trivial. Analogously, the hyperbolic map induces an isomorphism $\pi_{-6,-3}(\KQ)\cong\Z/2$ since skew-symmetric forms have even ranks. Combining these computations with the Bott periodicity we get the following lemma.
\begin{lemma}[see also {\cite[Theorem 10.1]{Walter}}]\label{lemma:kq_additive}
    Let $R$ be as above. Then the geometric part of the homotopy groups of $\KQ_R$ is given by \begin{equation*} \pi_{2n,n}(\KQ)\cong\begin{cases} \GW(R), & \text{if}\ n\equiv 0\pmod*{4}, \\ \Z/2, & \text{if}\ n\equiv 1\pmod*{4}, \\ 2\Z, & \text{if}\ n\equiv 2\pmod*{4}, \\ 0, & \text{if}\ n\equiv 3\pmod*{4}, \end{cases}
    \end{equation*} where the isomorphisms are constructed above. Moreover, if $f\colon R\to R'$ is a homomorphism of regular local rings in which $2$ is invertible, then the base change maps $\pi_{2n,n}(\KQ_R)\to\pi_{2n,n}(\KQ_{R'})$ are isomorphisms for $n\not\equiv 0\pmod*{4}$ and coincide with the obvious map $\GW(R)\to\GW(R')$ under the above identification for $n\equiv 0\pmod*{4}$.
\end{lemma}
\begin{proof}
    The isomorphisms are already constructed. Let us check the last item. The base change maps are as stated for $n\equiv 0,2\pmod*{4}$ since the isomorphisms $\pi_{0,0}(\KQ_R)\cong \GW^{[0]}_0(R)$ and $\pi_{-4,-2}(\KQ_R)\cong \GW^{[2]}_0(R)$ are natural in $R$. The remaining case follows by comparing the exact sequences induced by the Wood cofiber sequences over $R$ and $R'$ and using the fact that $\K_0(R)\cong\K_0(R')$.
\end{proof}
Recall that the very effective motivic stable homotopy category $\SH^{\mathrm{veff}}(S)$ is the full subcategory of $\SH(S)$ generated under colimits and extensions by suspension spectra of smooth $S$-schemes \cite[Definition 5.5]{SO12}. We denote by $\tilde{\effcov}_q(-)$ the functor of the $q$-th very effective cover; see \cite[\S 4]{Bac17}.
\begin{lemma}
    Let $S$ be a base scheme and let $\EE\in \SHk$ be a motivic spectrum over $S$. Then the canonical map $\tilde{\effcov}_m(\EE)\to \EE$ induces an isomorphism $\pi_{2n,n}(\tilde{\effcov}_q(\EE))\xrightarrow{\simeq} \pi_{2n,n}(\EE)$ for $n\geq q$.
\end{lemma}
\begin{proof}
    The argument is similar to the one for the effective covers (see e.g., \cite[Lemma 6.6(1)]{FS18}) and we add it for the convenience of the reader. Recall that the morphism $\tilde{\effcov}_q(\EE)\to \EE$ is the counit of the adjunction
    $$ \tilde{\mathrm{i}}_q\colon \Sigma^{2q,q}\SH^{\mathrm{veff}}(S)\rightleftarrows \SHk\colon \tilde{\mathrm{r}}_q, $$
    where $\tilde{\mathrm{i}}_q$ is the inclusion. We claim that the sphere $\Sigma^{2n,n}\sph$ is contained in $\Sigma^{2q,q}\SH^{\mathrm{veff}}(S)$ for $n\geq q$. This follows from the fact that $\Sigma^{2,1}\sph\cong\Sigma^\infty(\Proj^1,\infty)$ is very effective since $\SH^{\mathrm{veff}}(S)$ is closed under the smash product \cite[Lemma 5.6]{SO12}. The result follows by the universal property of the counit.
\end{proof}
Following \cite[Definition 2.1]{ARO}, we put $\mathrm{kq}:=\tilde{\effcov}_0(\KQ)\in\SHk$. This motivic spectrum has a natural structure of an $\Einf$-ring spectrum such that $\mathrm{kq}\to \KQ$ is a morphism of $\Einf$-ring spectra; see \cite[Proposition 5.3]{GRSO}. As a special case of the previous lemma we get.
\begin{lemma}
    The canonical map $\mathrm{kq}\to\KQ$ induces an isomorphism $\pi_{2n,n}(\mathrm{kq})\xrightarrow{\simeq}\pi_{2n,n}(\KQ)$ for $n\geq 0$.
\end{lemma}
Consider the element $\eta\etatop:=\eta\cdot\etatop\in \pi_{2,1}(\sph)$ given by the product of the motivic Hopf element $\eta$ and the topological Hopf element $\etatop\in \pi_{1,0}(\sph)$.
\begin{lemma}\label{lemma:kq_generators}
    Let $R$ be a regular local ring with $1/2\in R$. Then the $\GW(R)$-modules $\pi_{2n,n}(\KQ)$ are generated by \begin{enumerate} \item  $\beta^{\frac{n}{4}}$ if $n\equiv 0\pmod*{4}$, \item $\eta\etatop\cdot\beta^{\frac{n-1}{4}}$ if $n\equiv 1\pmod*{4}$, \item $H\cdot\beta^{\frac{n-2}{4}}$ if $n\equiv 2\pmod*{4}$. \end{enumerate} Here $H\in \pi_{4,2}(\KQ)$ corresponds to the standard skew-symmetric form via the identification $\pi_{4,2}(\KQ)\cong\pi_{-4,-2}(\KQ)\cong \GW^{[2]}_0(R)$.
\end{lemma}
\begin{proof}
    Since the multiplication by the Bott element $\beta\in \pi_{8,4}(\KQ)$ induces $(8,4)$-periodicity of the homotopy groups of $\KQ$, we need to deal with $\pi_{0,0}(\KQ)$, $\pi_{2,1}(\KQ)$, and $\pi_{4,2}(\KQ)$. The first case is obvious and the last one follows tautologically by the definition of $H\in \pi_{4,2}(\KQ)\cong \GW^{[2]}_0(R)\cong 2\Z$. Thus it remains to show that $\eta\etatop$ generates $\pi_{2,1}(\KQ)$.
    
    First, let $R$ be a field of characteristic different from $2$. By the previous lemma, we need to show that $\eta\etatop$ generates $\pi_{2,1}(\mathrm{kq})$. From the computation of the first homotopy module of the sphere spectrum, $\eta\etatop$ generates $\pi_{2,1}(\mathrm{kq})[\einv]$, where $e$ is the exponential characteristic of $R$; see \cite{RSO19, RSO24} and \cite[Theorem 2.5]{Ron20}. The same is true integrally as well, since inverting of $e$ does not change the group $\pi_{2,1}(\mathrm{kq})\cong \pi_{2,1}(\KQ)\cong \Z/2$ by Lemma \ref{lemma:kq_additive}. 

    In general, the base change induces an isomorphism $\pi_{2,1}(\KQ_R)\to \pi_{2,1}(\KQ_{\mathrm{Frac}(R)})$ by Lemma \ref{lemma:kq_additive}. The result then follows because the element $\eta\etatop$ is compatible with base change.
\end{proof}
\begin{corollary}\label{cor:etatop_iso}
    Let $R$ be as above. Then the following diagram commutes \[\xymatrix{\pi_{8n+1,4n+1}(\KQ) \ar[r]^-\etatop \ar[d]^\simeq & \pi_{8n+2,4n+1}(\KQ) \ar[d]^\simeq \\ \W(R) \ar@{->>}[r]^{\overline{\mathrm{rk}}} & \Z/2,}\] where the vertical isomorphisms are taken from Lemmas \ref{appendix_lemma_witt}, \ref{lemma:kq_additive}. In particular, if $R$ is a quadratically closed field then the map $\etatop\colon\pi_{8n+1,4n+1}(\KQ)\to \pi_{8n+2,4n+1}(\KQ)$ is an isomorphism.
\end{corollary}
\begin{proof}
    By the previous lemma the natural map $\eta\etatop\colon\pi_{8n,4n}(\KQ)\to \pi_{8n+2,4n+1}(\KQ)$ is surjective. Since it factors as \[\pi_{8n,4n}(\KQ)\xrightarrow{\eta}\pi_{8n+1,4n+1}(\KQ)\xrightarrow{\etatop}\pi_{8n+2,4n+1}(\KQ),\] the second homomorphism is a surjection as well. Hence, by Lemma \ref{appendix_lemma_witt} the desired diagram commutes if $R$ is a quadratically closed field. The general case follows by base change.
\end{proof}
\begin{theorem}\label{theorem:geom_diag_KQ}
    Let $R$ be a regular local ring with $1/2\in R$. Then there is an isomorphism of graded $\GW(R)$-algebras \[\pi_{2*,*}(\KQ)\cong \bigslant{\GW(R)[\eta\etatop,H,\beta^{\pm 1}]}{I},\] where the ideal $I$ is generated by the relations \[ 2\cdot\eta\etatop,\ (\eta\etatop)^2,\ \mathrm{I}(R)\cdot\eta\etatop,\ \eta\etatop\cdot H,\ \mathrm{I}(R)\cdot H,\ H^2-2h\cdot\beta.\]
\end{theorem}
\begin{proof}
    Let us check the above relations. The first two of these follow from Lemma \ref{lemma:kq_additive}, and $\mathrm{I}(R)\cdot\eta\etatop=0$ holds by the proof of the previous corollary and Lemma \ref{appendix_lemma_witt}. The element $\eta\etatop\cdot H$ lies in $\pi_{6,3}(\KQ)=0$. The relation $\mathrm{I}(R)\cdot H=0$ is a consequence of the identification $\pi_{4,2}(\KQ)\cong \GW^{[2]}_0(R)$. For the last one see \cite[5.2.d]{FH20}.

    Therefore, there is a surjective homomorphism of the graded $\GW(R)$-algebras from the respective quotient onto $\pi_{2*,*}(\KQ)$. This map is an isomorphism by Lemmas \ref{lemma:kq_additive} and \ref{lemma:kq_generators}.
\end{proof}
\begin{remark}
    Under the same assumptions on $R$ there is an isomorphism \[\pi_{2*,*}(\mathrm{kq})\cong \bigslant{\GW(R)[\eta\etatop,H,\beta]}{I},\] where the ideal $I$ is generated by the relations above. Moreover, the canonical map $\mathrm{kq}\to \KQ$ induces an isomorphism $\pi_{2*,*}(\KQ)\cong \pi_{2*,*}(\mathrm{kq})[\beta^{-1}]$; see e.g., \cite[Proposition 7.7]{HJN}.
\end{remark}
\begin{remark}\label{appendix_a_remark_betti_real}
    The answer in the theorem is quite analogous to the even homotopy groups of the real $\K$-theory spectrum $\mathrm{KO}$. If $R=\C$ the complex Betti realization sends $\KQ_\C$ to $\mathrm{KO}$ and induces an isomorphism $\pi_{2*,*}(\KQ_\C)\xrightarrow{\simeq} \pi_{2*}(\mathrm{KO})$. Similarly, $\BettiC(\mathrm{kq}_\C)\cong \mathrm{ko}$ and $\BettiC\colon\pi_{2*,*}(\mathrm{kq}_\C)\xrightarrow{\simeq} \pi_{2*}(\mathrm{ko})$; see \cite[Lemma 2.13]{ARO} for the first equivalence.
\end{remark}
\section{\texorpdfstring{$c_1$}{c1}-spherical cobordism spectrum}\label{appendix_B}
In this appendix we summarise all topological results used in the main part of the text, without any claim to originality. First, we recall the construction of the Thom functor and define the $c_1$-spherical cobordism spectrum. Then we present the main parts of the computation of the homotopy groups of $\MSU$ due to Novikov \cite{Nov}, Conner and Floyd \cite{CF}. Finally, we briefly recall the Anderson--Brown--Peterson theorem on characteristic numbers of $\mathrm{SU}$-manifolds \cite{ABP66}.
\subsection{Thom functor and \texorpdfstring{$c_1$}{c1}-spherical cobordism spectrum}
Denote by $\Spc$ the $\infty$-category of spaces. Recall that the Thom functor is a functor $\mathrm{M}\colon\Spc_{/\Pic(\Spt)}\to \Spt$ given by the formal colimit construction \[\mathrm{M}(f\colon X\to \Pic(\Spt)):=\colim (X\xrightarrow{f} \Pic(\Spt)\hookrightarrow \Spt).\] To see the relation with the classical Thom spectra, let us consider the $\mathrm{J}$-homomorphism $\mathrm{J}\colon\BO\to \Pic(\Spt)$. The spectrum $\mathrm{M(J)}$ is equivalent to the classical Thom spectrum built out of the orthogonal groups \[\mathrm{M}(\mathrm{J}\colon\BO\to\Pic(\Spt))\cong \mathrm{MO}.\] To recover other Thom spectra, which are usually defined by a topological group $G$ (e.g., $\mathrm{U}$, $\mathrm{SO}$, $\mathrm{SU}$, $\mathrm{Sp}$), we need to apply $\mathrm{M}$ to the composition $\mathrm{B}G\to \BO\to\Pic(\Spt)$.

Denote by $\mathrm{BW}$ the fiber of the following morphism of spaces \[\BU\times\CP^1\xrightarrow{\det-\,\mathrm{in}}\CP^\infty.\] Here we use the $\Einf$-space structure on $\CP^\infty$ to take the difference of two morphisms. The space $\mathrm{BW}$ comes equipped with the forgetful maps $\mathrm{BSU}\to \mathrm{BW}\to \BU$. We also denote by $\mathrm{BW}(n)$ the fiber of \[\BU(n)\times \CP^1\xrightarrow{\det-\,\mathrm{in}}\CP^\infty\] and by $\mathrm{TBW}(n)$ the Thom space of the vector bundle classified by $\mathrm{BW}(n)\to \BU(n).$ Obviously, we have \[\mathrm{BW}\cong \colim_n \mathrm{BW}(n).\]
\begin{remark}\label{remark:bwn_s1bundle}
    It follows from $\CP^\infty\cong\mathrm{B}S^1$ that $\mathrm{BW}(n)$ is equivalent to the total space of the principal $S^1$-bundle $S(\det\mathrm{EU}(n)\boxtimes \struct(1))$ over $\BU(n)\times \CP^1$.
\end{remark}
\begin{definition}
    The \textit{$c_1$-spherical cobordism spectrum} $\W$ is the Thom spectrum associated with \[\mathrm{BW}\to \BU\to \BO\xrightarrow{\mathrm{J}}\Pic(\Spt).\] By construction, there are canonical forgetful morphisms  $$\MSU\xrightarrow{\forg}\W\xrightarrow{\bforg}\MU.$$
\end{definition}
Now we briefly describe cobordism theory that corresponds to the $c_1$-spherical cobordism spectrum under the Pontryagin--Thom construction. Let $M$ be a stable complex manifold with structure map $\xi\colon M\to \BU$. The $\CP^1$-structure on $(M,\xi)$ is a lift of $\xi$ to a map $M\to \mathrm{BW}$. This data is equivalent to a morphism $l\colon M\to \CP^1$ with an isomorphism $\det(\xi^*(\mathrm{EU}))\cong l^*(\struct(-1))$. Roughly speaking, it means that the determinant of the stable normal bundle, which comes from $\CP^\infty$ by formal reasons, actually comes from $\CP^1\hookrightarrow\CP^\infty$. Using the standard notion of cobordism in this context, we obtain the cobordism groups of manifolds with $\CP^1$-structure $\Omega^{\CP^1}_*$.
\begin{theorem}
   The Pontryagin--Thom construction gives an isomorphism $\Omega^{\CP^1}_*\cong\pi_*(\W)$.
\end{theorem}
\begin{proof}
    \cite[Chapter VIII]{Sto}
\end{proof}
\begin{proposition}\label{appendix:prop_cofib_1}
    There is a cofiber sequence $\Sigma^1\MSU\xrightarrow{\etatop} \MSU\xrightarrow{\forg} \W$ in $\Spt$, where $\etatop\in \pi_1(\MSU)$ is the Hopf element.
\end{proposition}
\begin{proof}
    \cite[Proposition 2.2]{CP23}
\end{proof}
We put $\delta:=(\Sigma^2\forg)\circ d\colon\W\to\Sigma^2\W$, where $d$ is the boundary morphism in the above cofiber sequence.
\begin{lemma}\label{appendix:lemma_boundary}
    Let $\partial\in \MU^2(\MU)$ be a cohomological operation that corresponds to the characteristic class $c_1(\det\mathrm{EU}^\vee)$ under the Thom isomorphism $\MU^*(\MU)\cong \MU^*(\BU)$. Then the following diagram commutes up to homotopy
    \[\xymatrix{\W \ar[r]^-{\delta} \ar[d]_-{\bforg} & \Sigma^{2}\W \ar[d]^-{\Sigma^{2}\bforg} \\ \MU \ar[r]^-{-\partial} & \Sigma^{2}\MU. }\]
\end{lemma}
\begin{proof}
    \cite[17.3]{CF}, see also \cite[Proposition 2.5]{CP23}
\end{proof}
\begin{proposition}\label{appendix:split_cofib}
    The forgetful morphism $\bforg$ induces the following cofiber sequence of spectra  \[\W\xrightarrow{\bforg} \MU\xrightarrow{\Delta}\Sigma^4\MU,\] where $\Delta$ is the operation that corresponds to $c_1(\det\mathrm{EU})\cdot c_1(\det\mathrm{EU}^\vee)$ under the Thom isomorphism $\MU^*(\MU)\cong \MU^*(\BU)$. Moreover, $\Delta$ has a right inverse, and the cofiber sequence splits. 
\end{proposition}
\begin{proof}
    \cite[Proposition 2.11]{CP23}
\end{proof}
In particular, the homotopy groups of $\W$ can be computed as $\Ker(\Delta_*\colon\pi_*(\MU)\to \pi_{*-4}(\MU))$. Thus, they are free abelian and concentrated in even degrees.
\subsection{Homotopy groups of \texorpdfstring{$\MSU$}{MSU}}
By construction, $(\delta)^2=\Sigma^2\delta\circ\delta=0$, and there is a chain complex of abelian groups \[\cdots\to \pi_{n+2}(\W)\xrightarrow{\delta_*}\pi_n(\W)\xrightarrow{\delta_*}\pi_{n-2}(\W)\to \cdots.\]
Denote by $\Hml_n(\W,\delta)$ (respectively $\Cyc_n(\W,\delta)$, $\Bnd_n(\W,\delta)$) its homology (respectively cycles, boundaries).
\begin{lemma}\label{appendix:lemma_ker_delta_top}
    Suppose that $a$ and $b$ are elements of $\pi_*(\W)$. Then we have \begin{gather*} \Delta_*(a\cdot b)=-2\cdot\partial_*(a)\cdot\partial_*(b), \\
    \partial_*(a\cdot b)=\partial_*(a)\cdot b+a\cdot\partial_*(b)+a_{1,1}\cdot\partial_*(a)\cdot\partial_*(b),\end{gather*} where multiplication is performed in $\pi_*(\MU)$ and $a_{1,1}=-[\CP^1]\in\pi_2(\MU)$. Combining these formulas with Lemma \ref{appendix:lemma_boundary}, we get that $\Cyc_*(\W,\delta)$ is a subring of $\pi_*(\MU)$. 
\end{lemma}
\begin{proof}
    \cite[Chapter X]{Sto}, see also \cite[Lemma 6.5]{CLP}. 
\end{proof}
In particular, the ring homomorphism $\pi_*(\MSU)\to \pi_*(\MU)$ factors through $\Cyc_*(\W,\delta)$. In fact, the induced map $\pi_*(\MSU)\to \Cyc_*(\W,\delta)$ becomes a ring isomorphism after tensoring with $\Z[\nicefrac{1}{2}]$. This can be seen from the Adams--Novikov spectral sequence.
\begin{theorem}\label{appendix:thm_msu_1/2}
    There are isomorphisms of graded rings \[\pi_*(\MSU)[\nicefrac{1}{2}]\cong \Cyc_*(\W,\delta)[\nicefrac{1}{2}]\cong \Z[\nicefrac{1}{2}][x_2,x_3,\dots],\] where $\mathrm{deg}(x_i)=2i$. In particular, the homotopy groups of $\MSU$ do not contain odd torsion.
\end{theorem}
\begin{proof}
    \cite[Chapter X]{Sto}, see also \cite[Theorem 5.11]{CLP} for a modern exposition. 
\end{proof}
The $2$-primary torsion subgroup of $\pi_*(\MSU)$ was analysed by Conner and Floyd using the homology groups $\Hml_*(\W,\delta)$; see \cite{CF}. We summarize the answer below.
\begin{proposition}\label{prop:homology_top}
    The group $\Hml_{2n}(\W,\delta)$ is isomorphic to $(\Z/2)^{p(\frac{n}{4})}$ if $n\equiv 0\pmod*{4}$, to $(\Z/2)^{p(\frac{n-2}{4})}$ if $n\equiv 2\pmod*{4}$, and trivial otherwise.
\end{proposition}
\begin{proof}
    \cite{CF}, \cite[Chapter X]{Sto}.
\end{proof}
\begin{theorem}\label{thm:torsion_top}
    Every torsion element in $\pi_*(\MSU)$ has order $2$. There is no $2$-torsion except in degrees $8k+1$ and $8k+2$, where the $2$-torsion subgroup is isomorphic to $(\Z/2)^{p(k)}$. Moreover, the maps \[\etatop\colon\pi_{8k}(\MSU)\to \pi_{8k+1}(\MSU)\ \ \text{and}\ \ \etatop^2\colon\pi_{8k}(\MSU)\to \pi_{8k+2}(\MSU)\] are surjective onto the $2$-torsion subgroups. 
\end{theorem}
\begin{proof}
    \cite{CF}, see also \cite[\S 1.5]{CLP} for a modern exposition. 
\end{proof}
\begin{remark}\label{appendix:image_msu}
    Stong constructed a non-canonical multiplication on $\pi_*(\W)$ such that the quotient $\pi_*(\MSU)/{}_2\pi_*(\MSU)$ is a subring of $\pi_*(\W)$ \cite[Chapter X]{Sto}, see also \cite[Theorem 5.11]{CLP}. Moreover, the image of $\pi_n(\MSU)$ in $\pi_n(\W)$ is given by $\Cyc_n(\W,\delta)$ if $n\not\equiv 4\pmod*{8}$ and by $\Bnd_n(\W,\delta)$ if $n\equiv 4\pmod*{8}$. However, to the best of our knowledge an explicit description of the ring $\pi_*(\MSU)/{}_2\pi_*(\MSU)$ is unknown.
\end{remark}
    Recall that there are Pontryagin classes of oriented real vector bundles with values in the real $\K$-theory $p_i\in\mathrm{KO}^*(\mathrm{BSO}(2m))$; see \cite[Chapter X]{Sto}. They induce $\mathrm{KO}$-Pontryagin classes of special unitary bundles via $\mathrm{BSU}(m)\to \mathrm{BSO}(2m)$. For a partition $\omega=(\omega_1,\omega_2,\dots,\omega_k)$ the respective Pontryagin characteristic number of a stable $\mathrm{SU}$-manifold $\xi:M\to \mathrm{BSU}$ is the Kronecker pairing $\la{\xi^*(p_\omega),[M]\ra}=\la{\xi^*(p_{\omega_1}\dots p_{\omega_k}),[M]\ra}$. Note that the classical notation for the characteristic class $p_\omega$ is $\pi^\omega$.
\begin{theorem}[Anderson--Brown--Peterson]
    Let $M$ be a stable special unitary manifold. Then the class of $M$ in the $\mathrm{SU}$-cobordism ring $\pi_*(\MSU)$ is completely determined by the Chern numbers $c_\omega[M]$ and the Pontryagin numbers $p_{\omega}[M]$.
\end{theorem}
\begin{proof}
    \cite[Theorem 2.1]{ABP66}.
\end{proof}

\bibliographystyle{amsalpha}
\bibliography{MSL}

@article{Ana15,
    author = {Ananyevskiy, A.},
    title = {The special linear version of the projective bundle theorem},
    journal = {Compos. Math.},
    volume = {151:3},
    year = {2015},
    pages = {pp. 461-501},
    note = {\href{https://doi.org/10.1112/S0010437X14007702}{doi:10.1112/S0010437X14007702}}
}

@article{Ana16,
    author = {Ananyevskiy, A.},
    title = {On the push-forwards for motivic cohomology theories with invertible stable {H}opf element},
    journal = {Manuscripta Math.},
    volume = {150},
    year = {2016},
    pages = {pp. 21-44},
    note = {\href{https://doi.org/10.1007/s00229-015-0799-6}{doi:10.1007/s00229-015-0799-6}}
}

@article{Ana16b,
    author = {Ananyevskiy, A.},
    title = {On the relation of special linear algebraic cobordism to {W}itt groups},
    journal = {Homology Homotopy Appl.},
    year = {2016},
    volume = {18:1},
    pages = {pp. 205-230},
    note = {\href{https://dx.doi.org/10.4310/HHA.2016.v18.n1.a11}{doi.org/10.4310/HHA.2016.v18.n1.a11}}
}

@incollection{Ana20,
    author = {Ananyevskiy, A.},
    title = {{$SL$}-oriented cohomology theories},
    booktitle = {Motivic Homotopy Theory and Refined Enumerative Geometry},
    publisher = {Contemp. Math.},
    year = {2020},
    volume = {745},
    pages = {1-20},
    note = {\href{https://doi.org/10.1090/conm/745}{doi:10.1090/conm/745}}
}

@article{Ana21,
    author = {Ananyevskiy, A.},
    title = {Thom isomorphisms in triangulated motivic categories},
    journal = {Algebr. Geom. Topol.},
    year = {2021},
    volume = {21:4},
    pages = {pp. 2085-2106},
    note = {\href{https://doi.org/10.2140/agt.2021.21.2085}{doi: 10.2140/agt.2021.21.2085}}
}

@article{ALP,
    author = {Ananyevskiy, A. AND Levine, M. AND Panin, I.},
    title = {Witt sheaves and the $\eta$-inverted sphere spectrum},
    journal = {J. Topol.},
    year = {2017},
    volume = {10:2},
    pages = {pp. 370-385},
    note = {\href{https://doi.org/10.1112/topo.12015}{doi:10.1112/topo.12015}}
}

@article{ARO,
    author = {Ananyevskiy, A. AND R{\"o}ndigs, O. AND {\O}stv{\ae}r, P. A.},
    title = {On very effective hermitian {K}-theory},
    journal = {Math. Z.},
    volume = {294},
    year = {2020},
    pages = {pp. 1021-1034},
    note = {\href{https://doi.org/10.1007/s00209-019-02302-z}{doi:10.1007/s00209-019-02302-z}},
}

@article{ABP66,
    author = {Anderson, D. W. AND Brown, E. H. AND Peterson, F. P.},
    title = {{$SU$}-cobordism, {$KO$}-characteristic Numbers, and the {K}ervaire Invariant},
    journal = {Ann. of Math. (2)},
    year = {1966},
    volume = {83:1},
    pages = {pp. 54-67},
    note = {\href{https://doi.org/10.2307/1970470}{doi:10.2307/1970470}}
}

@article{AHW,
    author = {Asok, A. AND Hoyois, M. AND Wendt, M.},
    title = {Affine representability results in $\mathbb{A}^1$–homotopy theory, {II}: principal bundles and homogeneous spaces},
    journal = {Geom. Topol.},
    year = {2018},
    volume = {22:2},
    pages = {pp. 1181-1225},
    note = {\href{https://doi.org/10.2140/gt.2018.22.1181}{doi:10.2140/gt.2018.22.1181}}
}

@article{AHW2,
    author = {Asok, A. AND Hoyois, M. AND Wendt, M.},
    title = {Affine representability results in $\mathbb{A}^1$-homotopy theory {III}: finite fields and complements},
    journal = {Algebr. Geom.},
    year = {2020},
    volume = {7:5},
    pages = {pp. 634-644},
    note = {\href{http://doi.org/10.14231/AG-2020-023}{doi:10.14231/AG-2020-023}}
}

@article{Bac17,
    author = {Bachmann, T.},
    title = {The generalized slices of {H}ermitian {K}-theory},
    journal = {J. Topol.},
    year = {2017},
    volume = {10:4},
    pages = {pp. 1124-1144},
    note = {\href{https://doi.org/10.1112/topo.12032}{doi:10.1112/topo.12032}}
}

@article{BacDVR,
    author = {Bachmann, T.},
    title = {$\eta$-periodic motivic stable homotopy theory over {D}edekind domains},
    journal = {J. Topol.},
    year = {2022},
    volume = {15:2},
    pages = {pp. 950-971},
    note = {\href{https://doi.org/10.1112/topo.12234}{doi:10.1112/topo.12234}}
}

@article{BH21,
    author = {Bachmann, T. AND Hoyois, M.},
    title = {Norms in motivic homotopy theory},
    journal = {Ast\'erisque},
    year = {2021},
    volume = {425},
    note = {\href{https://doi.org/10.24033/ast.1147}{doi:10.24033/ast.1147}}
}

@unpublished{BHop,
    author = {Bachmann, T. AND Hopkins, M. J.},
    title = {$\eta$-periodic motivic stable homotopy theory over fields},
    note = {preprint, \href{https://arxiv.org/abs/2005.06778}{arXiv:2005.06778}},
    year = {2021}
}

@article{Chowt,
    author = {Bachmann, T. AND Kong, H. J. AND Wang, G. AND Xu, Z.},
    title = {The {C}how t-structure on the $\infty$-category of motivic spectra},
    journal = {Ann. of Math. (2)},
    year = {2022},
    volume = {195:2},
    pages = {pp. 707-773},
    note = {\href{https://doi.org/10.4007/annals.2022.195.2.5}{doi:10.4007/annals.2022.195.2.5}}
}

@incollection{Balmer,
    author = {Balmer, P.},
    title = {Witt groups},
    booktitle = {Handbook of K-theory},
    publisher = {Springer},
    year = {2005},
    chapter = {III.1},
    pages = {539-576},
    note = {\href{https://doi.org/10.1007/978-3-540-27855-9}{doi.org/10.1007/978-3-540-27855-9}}
}

@article{CP23,
    author = {Chernykh, G. AND Panov, T.},
    title = {{$SU$}-linear operations in complex cobordism and the $c_1$-spherical bordism theory},
    journal = {Izv. RAN. Ser. Mat.},
    year = {2023},
    volume = {87:4},
    pages = {pp. 133-165},
    note = {\href{https://doi.org/10.4213/im9334e}{doi:10.4213/im9334e}}
}

@article{CLP,
    author = {Chernykh, G. AND Limonchenko, I. AND Panov, T.},
    title = {{$SU$}-bordism: structure results and geometric representatives},
    journal = {Russ. Math. Surv.},
    year = {2019},
    volume = {74},
    pages = {pp. 461-524},
    note = {\href{https://doi.org/10.1070/rm9883}{doi:10.1070/rm9883}}
}

@article{CF,
    author = {Conner, P. E. AND Floyd, E. E.},
    title = {Torsion in {$SU$}-bordism},
    journal = {Mem. Amer. Math. Soc.},
    year = {1966},
    volume = {60}
}

@article{DFJK,
    author = {D{\'e}glise, F. AND Fasel, J. AND Jin, F. AND Khan, A. A.},
    title = {On the rational motivic homotopy category},
    journal = {J. \'Ec. polytech. Math.},
    year = {2021},
    volume = {8},
    pages = {pp. 533-583},
    note = {\href{https://doi.org/10.5802/jep.153}{doi:10.5802/jep.153}}
}

@article{DI13,
    author = {Dugger, D. AND Isaksen, D. C.},
    title = {Motivic {H}opf elements and relations},
    journal = {New York J. Math.},
    year = {2013},
    volume = {19},
    pages = {pp. 823-871}
}

@article{EHK,
    title = {Modules over algebraic cobordism}, 
    volume = {8},
    journal = {Forum Math. Pi}, 
    author = {Elmanto, E. and Hoyois, M. and Khan, A. A. and Sosnilo, V. and Yakerson, M.}, 
    year = {2020}, 
    pages = {e14},
    note = {\href{https://doi.org/10.1017/fmp.2020.13}{doi:10.1017/fmp.2020.13}}
}

@article{FH20,
    author = {Fasel, J. AND Haution, O.},
    title = {The stable {A}dams operations on Hermitian {K}-theory},
    journal = {Geom. Topol.},
    year = {2025},
    volume = {29:1},
    pages = {pp. 127–169},
    note = {\href{https://doi.org/10.2140/gt.2025.29.127}{doi:10.2140/gt.2025.29.127}}
}

@article{FS18,
    author = {Frankland, M. AND Spitzweck, M.},
    title = {Towards the dual motivic {S}teenrod algebra in positive characteristic},
    journal = {Math. Z.},
    year = {2026},
    volume = {312:3},
    pages = {83},
    note = {\href{https://doi.org/10.1007/s00209-026-03967-z}{doi:10.1007/s00209-026-03967-z}}
}

@article{Geisser,
    author = {Geisser, T.},
    title = {Motivic cohomology over {D}edekind rings},
    journal = {Math. Z.},
    year = {2004},
    volume = {248},
    pages = {pp. 773-794},
    note = {\href{https://doi.org/10.1007/s00209-004-0680-x}{doi.org/10.1007/s00209-004-0680-x}}
}

@article{GRSO,
    author = {Guti{\'e}rrez, J. J. AND R{\"o}ndigs, O. AND Spitzweck, M. AND {\O}stv{\ae}r, P. A.},
    title = {Motivic slices and coloured operads},
    journal = {J. Topol.},
    year = {2012},
    volume = {5:3},
    pages = {pp. 727-755}, 
    note = {\href{https://doi.org/10.1112/jtopol/jts015}{doi.org/10.1112/jtopol/jts015}}
}

@article{HW19,
    author = {Hornbostel, J. AND Wendt, M.},
    title = {Chow--{W}itt rings of classifying spaces of symplectic and special linear groups},
    journal = {J. Topol.},
    year = {2019},
    volume = {12:3},
    pages = {pp. 916-966},
    note = {\href{https://doi.org/10.1112/topo.12103}{doi:10.1112/topo.12103}}
}

@article{Hoy15,
    author = {Hoyois, M.},
    title = {From algebraic cobordism to motivic cohomology},
    journal = {J. Reine Angew. Math.},
    year = {2015},
    volume = {702},
    pages = {pp. 173-226},
    note = {\href{https://doi.org/10.1515/crelle-2013-0038}{doi:10.1515/crelle-2013-0038}}
}

@article{HJN,
    author = {Hoyois, M. AND Jelisiejew, J. AND Nardin, D. AND Yakerson, M.},
    title = {Hermitian {K}-theory via oriented {G}orenstein algebras},
    journal = {J. Reine Angew. Math.},
    year = {2022},
    volume = {793},
    pages = {pp. 105-142},
    note = {\href{https://doi.org/10.1515/crelle-2022-0063}{doi:10.1515/crelle-2022-0063}}
}

@article{Levine,
    author = {Levine, M.},
    title = {Comparison of cobordism theories},
    journal = {J. Algebra},
    year = {2009},
    volume = {322:9},
    pages = {pp. 3291-3317},
    note = {\href{https://doi.org/10.1016/j.jalgebra.2009.03.032}{doi:10.1016/j.jalgebra.2009.03.032}}
}

@book{LevMor,
    author = {Levine, M. AND Morel, F.},
    title = {Algebraic {C}obordism},
    publisher = {Springer Monographs in Mathematics, Springer--Verlag Berlin Heidelberg},
    year = {2007},
    pages = {1-246},
    note = {\href{https://doi.org/10.1007/3-540-36824-8}{doi:10.1007/3-540-36824-8}}
}

@article{LevPand,
    author = {Levine, M. AND Pandharipande, R.},
    title = {Algebraic cobordism revisited},
    journal = {Invent. Math.},
    year = {2009},
    volume = {176},
    pages = {pp. 63-130},
    note = {\href{https://doi.org/10.1007/s00222-008-0160-8}{doi:10.1007/s00222-008-0160-8}}
}

@article{LYZ,
    author = {Levine, M. AND Yang, Y. AND Zhao, G.},
    title = {Algebraic elliptic cohomology and flops {II}: {$SL$}-cobordism},
    journal = {Adv. Math.},
    year = {2021},
    volume = {384},
    pages = {107726},
    note = {\href{https://doi.org/10.1016/j.aim.2021.107726}{doi:10.1016/j.aim.2021.107726}}
}

@unpublished{LHA,
    author = {Lurie, J.},
    title = {Higher {A}lgebra},
    year = {2017},
    note = {book \href{http://www.math.harvard.edu/~lurie/papers/HA.pdf}{math.harvard.edu/~lurie/papers/HA.pdf}}
}

@incollection{Mer02,
    author = {Merkurjev, A. S.},
    title = {Algebraic oriented cohomology theories},
    booktitle = {Algebraic number theory and algebraic geometry},
    publisher = {Contemp. Math.},
    year = {2002},
    volume = {300},
    pages = {171-193},
}

@article{Mil60,
    author = {Milnor, J.},
    title = {On the cobordism ring ${\Omega}_*$ and a complex analogue {I}},
    journal = {Amer. J. Math.},
    year = {1960},
    volume = {82},
    pages = {pp 505-521}, 
    note = {\href{https://doi.org/10.2307/2372970}{doi:10.2307/2372970}}
}

@inproceedings{Mor03,
    author = {Morel, F.},
    title = {An introduction to $\mathbb{A}^1$-homotopy theory},
    booktitle = {Contemporary developments in algebraic K-theory, ICTP Lecture Notes, vol. XV, (Abdus Salam International Center for Theoretical Physics, Trieste)},
    year = {2004},
    pages = {357-441}
}

@article{MV99,
    author = {Morel, F. AND Voevodsky, V.},
    title = {$\mathbb{A}^1$-homotopy theory of schemes},
    journal = {Publ. Math. Inst. Hautes \'Etudes Sci.},
    year = {1999},
    volume = {90},
    pages = {pp. 45-143},
    note = {\href{https://doi.org/10.1007/BF02698831}{doi:10.1007/BF02698831}}
}

@unpublished{Nan23,
    author = {Nandy, A.},
    title = {An interpolation between Special Linear and General Algebraic cobordism {$MSL$} and {$MGL$}},
    year = {2023},
    note = {preprint, \href{https://arxiv.org/abs/2310.15721}{arXiv:2310.15721}}
}

@article{NSOLandw,
    author = {Naumann, N. AND Spitzweck, M. AND {\O}stv{\ae}r, P. A.},
    title = {Motivic {L}andweber exactness},
    journal = {Doc. Math.},
    year = {2009},
    volume = {14},
    pages = {pp. 551-593},
    note = {\href{https://doi.org/10.4171/dm/282}{doi:10.4171/dm/282}}
}

@article{NSO,
    author = {Naumann, N. AND Spitzweck, M. AND {\O}stv{\ae}r, P. A.},
    title = {Chern classes, {K}-theory and {L}andweber exactness over nonregular base schemes},
    journal = {Fields Institute Communications},
    year = {2009},
    volume = {56},
    pages = {pp. 307-317},
    note = {\href{http://doi.org/10.1090/fic/056/14}{doi:10.1090/fic/056/14}}
}

@article{Nov0,
    author = {Novikov, S. P.},
    title = {Some problems in the topology of manifolds connected with the theory of {T}hom spaces},
    journal = {Soviet Math. Dokl.},
    year = {1960},
    volume = {1},
    pages = {pp. 717-720}
}

@article{Nov,
    author = {Novikov, S. P.},
    title = {Homotopy properties of {T}hom complexes},
    journal = {Mat. Sb.},
    year = {1962},
    volume = {57:4},
    pages = {pp. 407-442}
}

@article{Pan09,
    author = {Panin, I.},
    title = {Oriented cohomology theories of algebraic varieties {II}},
    journal = {Homology Homotopy Appl.},
    year = {2009},
    volume = {11:1},
    pages = {pp. 349-405},
    note = {\href{https://doi.org/10.4310/HHA.2009.V11.N1.A14}{doi:10.4310/HHA.2009.V11.N1.A14}}
}

@article{PW18,
    author = {Panin, I. AND Walter, C.},
    title = {On the motivic commutative ring spectrum {$BO$}},
    journal = {St. Petersburg Math. J.},
    year = {2018},
    volume = {30:6},
    pages = {pp. 933-972},
    note = {\href{https://doi.org/10.1090/spmj/1578}{doi:10.1090/spmj/1578}}
}

@article{PW19,
    author = {Panin, I. AND Walter, C.},
    title = {Quaternionic Grassmannians and {B}orel classes in algebraic geometry},
    journal = {St. Petersburg Math. J.},
    year = {2022},
    volume = {33:1},
    pages = {pp. 97-140},
    note = {\href{https://doi.org/10.1090/spmj/1692}{doi:10.1090/spmj/1692}}
}

@article{PW22,
    author = {Panin, I. AND Walter, C.},
    title = {On the algebraic cobordism spectra {$MSL$} and {$MSp$}},
    journal = {St. Petersburg Math. J.},
    year = {2023},
    volume = {34:1},
    pages = {pp. 144-187},
    note = {\href{https://doi.org/10.1090/spmj/1748}{doi:10.1090/spmj/1748}}
}

@article{PPR08,
    author = {Panin, I. AND Pimenov, K. AND R{\"o}ndigs, O.},
    title = {A universality theorem for {V}oevodsky's algebraic cobordism spectrum},
    journal = {Homology Homotopy Appl.},
    year = {2008},
    volume = {10:2},
    pages = {pp. 211-226},
    note = {\href{http://doi.org/10.4310/HHA.2008.v10.n2.a11}{doi:10.4310/HHA.2008.v10.n2.a11}}
}

@article{Quillen,
    author = {Quillen, D.},
    title = {On the formal group laws of unoriented and complex cobordism theory},
    journal = {Bull. Amer. Math. Soc.},
    year = {1969},
    volume = {75:6},
    pages = {pp. 1293-1298}
}

@book{Ra04,
    author = {Ravenel, D.},
    title = {Complex cobordism and stable homotopy groups of spheres},
    note = {2nd ed.},
    publisher = {Providence, RI: AMS Chelsea Publishing},
    year = {2004}
}

@incollection{Ron20,
    author = {R{\"o}ndigs, O.},
    title = {Remarks on motivic {M}oore spectra},
    booktitle = {Motivic Homotopy Theory and Refined Enumerative Geometry},
    publisher = {Contemp. Math.},
    year = {2020},
    volume = {745},
    pages = {199-215},
    note = {\href{https://doi.org/10.1090/conm/745}{doi:10.1090/conm/745}}
}

@article{RO16,
    author = {R{\"o}ndigs, O. AND {\O}stv{\ae}r, P. A.},
    title = {Slices of hermitian {K}-theory and {M}ilnor's conjecture on quadratic forms},
    journal = {Geom. Topol.},
    year = {2016},
    volume = {20:2},
    pages = {pp. 1157--1212},
    note = {\href{https://doi.org/10.2140/gt.2016.20.1157}{doi:10.2140/gt.2016.20.1157}}
}

@article{RSO19,
	Author = {R{\"o}ndigs, O. AND Spitzweck, M. AND {\O}stv{\ae}r, P. A.},
	Journal = {Ann. of Math. (2)},
	Pages = {pp. 1--74},
	Title = {The first stable homotopy groups of motivic spheres},
	Volume = {189:1},
	Year = {2019},
    note = {\href{https://doi.org/10.4007/annals.2019.189.1.1}{doi:10.4007/annals.2019.189.1.1}}
}

@article{RSO24,
    author = {R{\"o}ndigs, O. AND Spitzweck, M. AND {\O}stv{\ae}r, P. A.},
    title = {The second stable homotopy groups of motivic spheres},
    journal = {Duke Math. J.},
    year = {2024},
    volume = {173:6},
    pages = {pp. 1017-1084},
    note = {\href{https://doi.org/10.1215/00127094-2023-0023}{doi:10.1215/00127094-2023-0023}}
}

@book{Sch,
    author = {Scharlau, W.},
    title = {Quadratic and {H}ermitian Forms},
    publisher = {Springer-Verlag, Berlin-Heidelberg-New York},
    volume = {270},
    year = {1985}
}

@article{Sch17,
    author = {Schlichting, M.},
    title = {Hermitian {K}-theory, derived equivalences and {K}aroubi's fundamental theorem},
    journal = {J. Pure Appl. Algebra},
    year = {2017}, 
    volume = {221:7},
    pages = {pp. 1729-1844},
    note = {\href{https://doi.org/10.1016/j.jpaa.2016.12.026}{doi:10.1016/j.jpaa.2016.12.026}}
}

@article{ST15,
    author = {Schlichting, M. AND Tripathi, G. S.},
    title = {Geometric models for higher {G}rothendieck--{W}itt groups in $\mathbb{A}^1$-homotopy theory},
    journal = {Math. Ann.},
    year = {2015},
    volume = {362},
    pages = {pp. 1143-1167},
    note = {\href{https://doi.org/10.1007/s00208-014-1154-z}{doi:10.1007/s00208-014-1154-z}}
}

@article{SS21,
    author = {Sechin, P. AND Semenov, N.},
    title = {Applications of the {M}orava {K}-theory to algebraic groups},
    journal = {Ann. Sci. \'Ec. Norm. Sup\'er.},
    year = {2021},
    volume = {54:4},
    pages = {pp. 945-990},
    note = {\href{https://doi.org/10.24033/asens.2474}{doi:10.24033/asens.2474}}
}

@article{SpiHZ,
    author = {Spitzweck, M.},
    title = {A commutative $\mathbb{P}^1$-spectrum representing motivic cohomology over {D}edekind domains},
    journal = {M{\'e}m. Soc. Math. Fr.},
    year = {2018},
    volume = {157},
    note = {\href{http://doi.org/10.24033/msmf.465}{doi:10.24033/msmf.465}}
}

@article{SpiMGL,
    author = {Spitzweck, M.},
    title = {Algebraic {C}obordism in mixed characteristic},
    journal = {Homology Homotopy Appl.},
    year = {2020},
    volume = {22:2},
    pages = {pp. 91-103},
    note = {\href{https://doi.org/10.4310/HHA.2020.v22.n2.a5}{doi:10.4310/HHA.2020.v22.n2.a5}}
}

@article{SO12,
    author = {Spitzweck, M. AND {\O}stv{\ae}r, P. A.},
    title = {Motivic twisted {K}–theory},
    journal = {Algebr. Geom. Topol.},
    year = {2012},
    volume = {12:1},
    pages = {pp. 565-599},
    note = {\href{https://doi.org/10.2140/agt.2012.12.565}{doi:10.2140/agt.2012.12.565}}
}

@book{Sto,
    author = {Stong, R. E.},
    title = {Notes on {C}obordism {T}heory},
    publisher = {Princeton: Princeton University Press},
    year = {1968}
}

@article{Thom,
    author = {Thom, R.},
    title = {Quelques propri{\'e}t{\'e}s globales des vari{\'e}t{\'e}s diff{\'e}rentiables},
    journal = {Comment. Math. Helv.},
    year = {1954},
    volume = {28},
    pages = {pp. 17-86},
    note = {\href{https://doi.org/10.1007/BF02566923}{doi:10.1007/BF02566923}}
}

@incollection{TT90,
    author = {Thomason, R. W. AND Trobaugh, T.},
    title = {Higher {A}lgebraic {K}-{T}heory of {S}chemes and of {D}erived {C}ategories},
    booktitle = {The Grothendieck Festschrift III},
    publisher = {Modern Birkh\"auser Classics},
    volume = {88},
    year = {1990},
    pages = {247-435},
    note = {\href{https://doi.org/10.1007/978-0-8176-4576-2_10}{doi.org/10.1007/978-0-8176-4576-2\_10}}
}

@incollection{Vishik,
    author = {Vishik, A.},
    title = {Fields of $u$-Invariant $2r+1$},
    booktitle = {Algebra, Arithmetic, and Geometry, Progress in Mathematics},
    publisher = {Birkh{\"a}user Boston},
    volume = {270},
    year = {2009},
    note = {\href{https://doi.org/10.1007/978-0-8176-4747-6_22}{doi:10.1007/978-0-8176-4747-6\_22}}
}

@article{Voe98,
    author = {Voevodsky, V.},
    title = {$\mathbb{A}^1$-homotopy theory},
    journal = {Doc. Math.},
    year = {1998},
    volume = {Extra Vol. I},
    pages = {pp. 579-604},
    note = {\href{https://doi.org/10.4171/dms/1-1/21}{doi:10.4171/dms/1-1/21}}
}

@article{Wall,
    author = {Wall, C. T. C.},
    title = {Determination of the cobordism ring},
    journal = {Ann. of Math. (2)},
    year = {1960},
    volume = {72:2},
    pages = {pp. 292-311},
    note = {\href{https://doi.org/10.2307/1970136}{doi:10.2307/1970136}}
}

@unpublished{Walter,
    author = {Walter, C.},
    title = {Grothendieck--{W}itt groups of triangulated categories},
    note = {preprint, \href{https://www.maths.ed.ac.uk/~v1ranick/papers/trigw.pdf}{maths.ed.ac.uk/~v1ranick/papers/trigw.pdf}},
    year = {2003}
}

@article{Yakerson,
    author = {Yakerson, M.},
    title = {The unit map of the algebraic special linear cobordism spectrum},
    journal = {J. Inst. Math. Jussieu},
    year = {2021},
    volume = {20:6},
    pages = {pp. 1905-1930},
    note = {\href{https://doi.org/10.1017/S1474748019000720}{doi:10.1017/S1474748019000720}}
}

@article{BW_Euler, 
    title={EULER CLASSES: SIX-FUNCTORS FORMALISM, DUALITIES, INTEGRALITY AND LINEAR SUBSPACES OF COMPLETE INTERSECTIONS},
    volume={22:2}, 
    journal={J. Inst. Math. Jussieu},
    author={Bachmann, T. and Wickelgren, K.}, 
    year={2023}, 
    pages={pp. 681–746},
    note={\href{https://doi.org/10.1017/s147474802100027x}{doi:10.1017/s147474802100027x}}
}

@article{Levine-Raksit,
    author = {Levine, M. AND Raksit, A.},
    title = {Motivic {G}auss--{B}onnet formulas},
    journal = {Algebra Number Theory},
    year = {2020},
    pages = {pp. 1801-1851},
    volume = {14:7},
    note = {\href{https://doi.org/10.2140/ant.2020.14.1801}{doi:10.2140/ant.2020.14.1801}}
}

@incollection{HAZEWINKEL,
    title = {Witt vectors. {P}art 1},
    series = {Handbook of Algebra},
    publisher = {North-Holland},
    volume = {6},
    pages = {319--472},
    year = {2009},
    booktitle = {Handbook of Algebra},
    issn = {1570-7954},
    note = {\href{https://doi.org/10.1016/S1570-7954(08)00207-6}{doi:10.1016/S1570-7954(08)00207-6}},
    author = {Hazewinkel, M.},
}

@book{Morel_book,
  title={$\mathbb{A}^1$-Algebraic Topology over a Field},
  author={Morel, F.},
  isbn={9783642295157},
  series={Lecture Notes in Mathematics},
  url={https://books.google.de/books?id=ptYsjgEACAAJ},
  year={2012},
  publisher={Springer Berlin Heidelberg}
}

\end{document}